\documentclass[a4paper,10pt]{amsart}

\usepackage{amsmath}
\usepackage{amsthm}
\usepackage{amssymb}

\usepackage{mathrsfs}
\usepackage[all]{xy}

\newcommand{\set}[2]{\left\{ #1 \mid #2 \right\}}

\newcommand{\HH}{\operatorname{H}}
\newcommand{\rHH}{\widetilde{\operatorname{H}}}
\newcommand{\HHH}{\mathbb{H}}

\newcommand{\CC}{\mathscr{C}}
\newcommand{\DD}{\mathscr{D}}
\newcommand{\GG}{\mathscr{G}}
\newcommand{\VV}{\mathcal{V}}

\newcommand{\Th}{\operatorname{Th}}

\newcommand{\SP}{{Sp}}

\newcommand{\sing}{S}

\newcommand{\Set}{\operatorname{Set}}
\newcommand{\Cat}{\operatorname{Cat}}
\newcommand{\matchfun}{\mathcal{M}}
\newcommand{\match}[1]{\matchfun_{#1}}
\newcommand{\grph}{\mathscr{G}raph}

\newcommand{\sgn}{\operatorname{sgn}}
\newcommand{\interior}{\operatorname{int}}
\newcommand{\tGamma}{\Gamma}

\newcommand{\map}{\operatorname{map}}

\newcommand{\Mod}{\operatorname{Mod}}
\newcommand{\Out}{\operatorname{Out}}

\newcommand{\uc}{\widetilde}

\newcommand{\ad}{\operatorname{ad}}
\newcommand{\dquot}{/\hspace{-3pt}/}

\newcommand{\Ss}{\mathcal{S}}

\newcommand{\Ker}{\operatorname{Ker}}
\newcommand{\Fib}{\mathcal{F}ib}

\newcommand{\cX}{\widetilde{X}}
\newcommand{\QQ}{\mathbb{Q}}
\newcommand{\RR}{\mathbb{R}}
\newcommand{\ZZ}{\mathbb{Z}}
\newcommand{\LL}{\mathbb{L}}

\newcommand{\Lie}{\mathscr{L}\hspace{-1.5pt}ie}

\newcommand{\Top}{\mathbf{Top}}
\newcommand{\DGL}{\mathbf{DGL}}

\newcommand{\im}{\operatorname{im}}

\newcommand{\colim}{\operatornamewithlimits{colim}}
\newcommand{\hocolim}{\operatorname{hocolim}}

\newcommand{\rank}{\operatorname{rank}}
\newcommand{\Ind}{\operatorname{Ind}}

\newcommand{\tensor}{\otimes}
\newcommand{\ltensor}{\otimes}
\newcommand{\Hom}{\operatorname{Hom}}

\newcommand{\Ext}{\operatorname{Ext}}
\newcommand{\Diff}{\operatorname{Diff}}
\newcommand{\tDiff}{\widetilde{\Diff}}
\newcommand{\Aut}{\operatorname{Aut}}
\newcommand{\aut}{\operatorname{aut}}
\newcommand{\taut}{\widetilde{\aut}}
\newcommand{\Der}{\operatorname{Der}}
\newcommand{\OutDer}{\operatorname{OutDer}}
\newcommand{\Sp}{\operatorname{Sp}}

\newcommand{\gl}{\mathfrak{g}}
\newcommand{\al}{\mathfrak{a}}

\newcommand{\CP}{\mathbb{C}\mathrm{P}}
\newcommand{\Bun}{Bun}

\newcommand{\SDR}[5]{\xymatrix{*[r]{#1} \ar@<1ex>[r]^-{#3} \ar@(ul,dl)[]_{#5} & #2 \ar@<1ex>[l]^-{#4}}}
\newcommand{\bigSDR}[5]{\xymatrix{*[r]{#1} \ar@<1ex>[rr]^-{#3} \ar@(ul,dl)[]_{#5} && #2 \ar@<1ex>[ll]^-{#4}}}
\newcommand{\bigbigSDR}[5]{\xymatrix{*[r]{#1} \ar@<1ex>[rrr]^-{#3} \ar@(ul,dl)[]_{#5} &&& #2 \ar@<1ex>[lll]^-{#4}}}

\newcommand{\rtree}{\xygraph{
!{<0pt,0pt>;<11pt,0pt>:<0pt,11pt>::}
!{(-2,2)}*+{1}="a"
!{(0,2)}*+{2}="b"
!{(-1,1)}="c"
!{(1,1)}*+{3}="d"
!{(0,0)}="e"
!{(2,0)}*+{4}="f"
!{(1,-1)}="g"
!{(1,-2)}*+{{\bf 5}}="h"
"a"-"c"
"c"-"b"
"e"-"c"
"e"-"d"
"g"-"e"
"g"-"f"
"h"-"g"}}

\newcommand{\trtree}{\xygraph{
!{<0pt,0pt>;<11pt,0pt>:<0pt,11pt>::}
!{(-2,2)}*+{2}="a"
!{(0,2)}*+{3}="b"
!{(-1,1)}="c"
!{(1,1)}*+{4}="d"
!{(0,0)}="e"
!{(2,0)}*+{{\bf 5}}="f"
!{(1,-1)}="g"
!{(1,-2)}*+{1}="h"
"a"-"c"
"c"-"b"
"e"-"c"
"e"-"d"
"g"-"e"
"g"-"f"
"h"-"g"}}

\newcommand{\ntree}{\xygraph{
!{<0pt,0pt>;<11pt,0pt>:<0pt,11pt>::}
!{(-2,1)}*+{1}="a"
!{(-1,2)}*+{2}="b"
!{(1,2)}*+{3}="c"
!{(2,1)}*+{4}="e"
!{(0,1)}="d"
!{(1,0)}="f"
!{(0,-1)}="g"
!{(0,-2)}*+{{\bf 5}}="h"
"a"-"g"
"g"-"h"
"g"-"f"
"f"-"e"
"f"-"d"
"d"-"b"
"d"-"c"}}

\newcommand{\nocontentsline}[3]{}
\newcommand{\tocless}[2]{\bgroup\let\addcontentsline=\nocontentsline#1{#2}\egroup}

\newtheorem{theorem}{Theorem}
\newtheorem{conjecture}[theorem]{Conjecture}
\newtheorem{proposition}[theorem]{Proposition}
\newtheorem{corollary}[theorem]{Corollary}
\newtheorem{lemma}[theorem]{Lemma}

\theoremstyle{definition}
\newtheorem{definition}[theorem]{Definition}
\newtheorem{remark}[theorem]{Remark}

\newtheorem{example}[theorem]{Example}

\numberwithin{theorem}{section}
\title{Rational homotopy theory of automorphisms of manifolds}

\author{Alexander Berglund}
\address{Department of Mathematics\\
Stockholm University\\
SE-106 91 Stockholm\\
Sweden}
\email{alexb@math.su.se}

\author{Ib Madsen}
\address{Department of Mathematical Sciences\\
University of Copenhagen\\
Universitetsparken 5\\
DK-2100 Copenhagen \O{}\\
Denmark}
\email{imadsen@math.ku.dk}

\thanks{Supported by the Danish National Research Foundation through the Centre for Symmetry and Deformation (DNRF92). Supported by ERC adv grant no.~228082. Supported by the Swedish Research Council through grant no.~2015-03991.}

\begin{document}

\begin{abstract}
We study the rational homotopy types of classifying spaces of automorphism groups of smooth simply connected manifolds of dimension at least five. We give dg Lie algebra models for the homotopy automorphisms and the block diffeomorphisms of such manifolds.

Moreover, we use these models to calculate the rational cohomology of the classifying spaces of the homotopy automorphisms and block diffeomorphisms of the manifold $\#^g S^d \times S^d$ relative to an embedded disk as $g\to \infty$. The answer is expressed in terms of stable cohomology of arithmetic groups and invariant Lie algebra cohomology. Through an extension of Kontsevich's work on graph complexes, we relate our results to the (unstable) homology of automorphisms of free groups with boundaries.
\end{abstract}

\maketitle

\tableofcontents

\section{Introduction}
This work examines homotopical and homological properties of groups of automorphisms of simply connected smooth manifolds $M^n$ with $\partial M = S^{n-1}$, for $n\geq 5$. We study three types of automorphism groups, namely the homotopy automorphisms $\aut_\partial(M)$, the block diffeomorphisms $\tDiff_\partial(M)$ and the diffeomorphisms $\Diff_\partial(M)$. The subscript $\partial$ indicates that we consider automorphisms that fix the boundary pointwise. The classifying spaces are related by maps
\begin{equation} \label{eq:classifying spaces}
B \Diff_\partial(M) \xrightarrow{I} B \tDiff_\partial(M) \xrightarrow{J} B\aut_\partial(M).
\end{equation}

Let $\aut_{\partial,\circ}(M)$ denote the connected component of $\aut_\partial(M)$ that contains the identity, and write $\tDiff_{\partial,\circ}(M)$ for the subgroup of block diffeomorphisms homotopic to the identity.
For a vector bundle $\xi$ over $M$, let $\aut_{\partial,\circ}^*(\xi)$ be the topological monoid of diagrams
$$
\xymatrix{\xi \ar[r]^-{\widehat{f}} \ar[d] & \xi \ar[d] \\ M \ar[r]^-f & M}
$$
with $f\in \aut_{\partial,\circ}(M)$ and $\widehat{f}$ a fiberwise isomorphism over $f$ that restricts to the identity on the fiber over the basepoint $*\in \partial M$. Then stabilize,
$$\aut_{\partial,\circ}^*(\xi^S) = \hocolim_s \aut_{\partial,\circ}^*(\xi\times \RR^s),$$
where the stabilization maps are given by $(f,\widehat{f}) \mapsto (f, \widehat{f} \times id_\RR)$.

\begin{theorem} \label{thm:first main result}
For a simply connected smooth compact manifold $M$ of dimension $n\geq 5$ with $\partial M = S^{n-1}$ and tangent bundle $\tau_M$, the differential gives rise to a map
$$
D\colon B\tDiff_{\partial,\circ}(M) \rightarrow B\aut_{\partial,\circ}^*(\tau_M^S).
$$
The spaces $B\tDiff_{\partial,\circ}(M)$ and $B\aut_{\partial,\circ}^*(\tau_M^S)$ are nilpotent and the map $D$ is a rational homotopy equivalence. In particular,
\begin{equation*} \label{eq:H_k}
\HH_k(B\tDiff_{\partial,\circ}(M);\QQ) \cong \HH_k(B\aut_{\partial,\circ}^*(\tau_M^S);\QQ),
\end{equation*}
\begin{equation*} \label{eq:pi_k}
\pi_k(B\tDiff_{\partial,\circ}(M))\tensor \QQ \cong \pi_k(B\aut_{\partial,\circ}^*(\tau_M^S))\tensor \QQ,
\end{equation*}
for all $k$.
\end{theorem}


Thus, from the point of view of rational homotopy and homology, $B\tDiff_{\partial,\circ}(M)$ may be replaced by $B\aut_{\partial,\circ}^*(\tau_M^S)$. Building on Quillen's and Sullivan's rational homotopy theory and subsequent work of Schlessinger-Stasheff and Tanr\'e, we proceed to construct a dg Lie algebra model of the latter space. Consider the desuspension of the reduced rational homology,
$$V = s^{-1}\rHH^*(M;\QQ).$$
There is a differential $\delta$ on the free graded Lie algebra $\LL(V)$ such that $(\LL(V),\delta)$ is a minimal dg Lie algebra model for $M$. Moreover, there is a distinguished cycle $\omega \in \LL(V)$ that represents the inclusion of the boundary sphere. Write $\Der_\omega \LL(V)$ for the dg Lie algebra of derivations $\theta$ on $\LL(V)$ such that $\theta(\omega) = 0$, with differential
$$[\delta,\theta] = \delta\circ \theta - (-1)^{|\theta|} \theta \circ \delta,$$
and let $\Der_\omega^+ \LL(V)$ denote the sub dg Lie algebra of positive degree derivations such that $[\delta,\theta] = 0$ if $\theta$ is of degree $1$.

Consider the graded vector space $P = \pi_*(\Omega BO) \tensor \QQ$ and fix generators $q_i\in \pi_{4i-1}(\Omega BO)\tensor \QQ$ by the equation
$$\langle p_i,\sigma(q_i) \rangle = 1$$
where $p_i\in \HH^{4i}(BO;\QQ)$ is the $i$th Pontryagin class and $\sigma(q_i) \in \pi_{4i}(BO)\tensor \QQ$ is the suspension.
Let $p_i(\tau_M)\in \HH^{4i}(M;\QQ)$ denote the Pontryagin classes of the tangent bundle $\tau_M$ of $M$.
There is a distinguished element of degree $-1$ in the tensor product $\rHH^*(M;\QQ)\tensor P$,
$$\tau = \sum_i p_i(\tau_M)\tensor q_i.$$
The action of $\Der_\omega^+ \LL(V)$ on $\LL(V)$ induces an action on $\LL(V)/[\LL(V),\LL(V)] = s^{-1}\rHH_*(M;\QQ)$, and hence on the tensor product $\rHH^*(M;\QQ)\tensor P$. We may then form the dg Lie algebra
$$
\mathcal{M}^\tau = \big(\rHH^*(M;\QQ)\tensor P\big)_{\geq 0} \rtimes_\tau \Der_\omega^+ \LL(V),
$$
where the subscript on the left factor indicates that we discard elements of negative degree. The Lie bracket is given by
$$\big[(x,\theta),(y,\eta)\big] = \big( x\ldotp \eta + \theta \ldotp y, [\theta,\eta] \big),$$
where $x\ldotp \eta$ is the action above and $\theta \ldotp y = - (-1)^{|\theta||y|} y\ldotp \theta$. The differential is given by
$$
\partial^\tau(x,\theta) = \big( \tau\ldotp \theta, [\delta,\theta] \big).
$$

\begin{theorem} \label{thm:models}
For a simply connected smooth compact manifold $M^n$ with $\partial M = S^{n-1}$,
\begin{enumerate}
\item $\big(\Der_\omega^+ \LL(V),[\delta,-] \big)$ is a dg Lie algebra model for $B\aut_{\partial,\circ}(M)$,
\item $\big( \mathcal{M}^\tau,\partial^\tau \big)$ is a dg Lie algebra model for $B\aut_{\partial,\circ}^*(\tau_M^S)$.
\end{enumerate}
\end{theorem}

The first part of Theorem \ref{thm:models} is proved below as Theorem \ref{thm:aut model} and the second part as Theorem \ref{thm:block diff model}.

We next focus attention on highly connected manifolds, for which these models simplify dramatically: if $M$ is $(d-1)$-connected and $2d$-dimensional for some $d\geq 3$, then $\delta = 0$ and the action of $\Der_\omega^+ \LL(V)$ on the reduced homology of $M$ is trivial, for degree reasons. In these cases, we can also analyze the spectral sequences of the coverings
$$B\aut_{\partial,\circ}(M) \to B\aut_\partial(M),$$
$$B\tDiff_{\partial,\circ}(M) \to B\tDiff_\partial(M),$$
which leads to a calculation of the rational cohomology of the base spaces (in a range).

In particular, we consider the generalized surfaces of ``genus'' $g$,
$$M_{g,1} = \#^g S^d\times S^d\setminus \interior(D^{2d}).$$
For $2d>4$, the three spaces in \eqref{eq:classifying spaces} are radically different (the case $2d=4$ is excluded due to the usual difficulties in dimension four, but see Remark \ref{rem:dimension four} below). Still, in all three cases, there is a stable range for the rational cohomology: in degrees less than $(g-4)/2$ the cohomology is independent of $g$. This was proved in \cite{GRW2} for $\Diff_\partial(M_{g,1})$ and we prove it for $\tDiff_\partial(M_{g,1})$ and $\aut_\partial(M_{g,1})$ in this paper\footnote{In an earlier paper \cite{BM} we established a stability range that depended on the dimension of the manifold. The range is greatly improved in this paper.}. We then proceed to study the stable cohomologies and the maps between them,
\begin{equation} \label{eq:stable cohomologies}
\HH^*(B\aut_\partial(M_{\infty,1});\QQ) \xrightarrow{J^*} \HH^*(B\tDiff_\partial(M_{\infty,1});\QQ) \xrightarrow{I^*} \HH^*(B\Diff_\partial(M_{\infty,1});\QQ).
\end{equation}
The desuspension of the reduced homology $V_g = s^{-1}\widetilde{\HH}_*(M_{g,1};\QQ)$, equipped with the intersection form $\langle -,-\rangle$, is a non-degenerate graded anti-symmetric vector space; it admits a graded basis
$$\alpha_1,\ldots,\alpha_g,\beta_1,\ldots,\beta_g,\quad |\alpha_i|=|\beta_i| = d-1,$$
such that $\langle \alpha_i,\alpha_j\rangle = \langle \beta_i,\beta_j \rangle = 0$ and $\langle \alpha_i,\beta_j\rangle =-(-1)^{|\alpha_i||\beta_j|} \langle \beta_j,\alpha_i\rangle =\delta_{ij}$.
It follows directly from Theorem \ref{thm:models} that
\begin{equation} \label{eq:sub Lie algebra}
\gl_g = \Der_\omega^+ \LL(V_g)
\end{equation}
is a dg Lie algebra model for $B\aut_{\partial,\circ}(M_{g,1})$, where $\omega = [\alpha_1,\beta_1] + \cdots + [\alpha_g,\beta_g]$.
The differential $\delta$ is zero, so in particular we get a computation of the rational homotopy groups:
\begin{equation} \label{eq:tanre}
\pi_{*+1} B\aut_\partial(M_{g,1}) \tensor \QQ \cong \Der_\omega^+ \LL(V_g),\quad *>0.
\end{equation}
The Whitehead product on the left hand side corresponds to the commutator bracket on the right hand side.

The fundamental group of $B\aut_\partial(M_{g,1})$, i.e., the homotopy mapping class group, can be determined up to commensurability. The automorphism group,
$$G_g(\QQ) = \Aut(V_g, \langle -,-\rangle),$$
is the $\QQ$-points of an algebraic group, isomorphic to $\Sp_{2g}(\QQ)$ or $O_{g,g}(\QQ)$ depending on the parity of $d$. In \S\ref{sec:wall} we introduce an arithmetic subgroup $\Gamma_g$ of $G_g(\QQ)$ commensurable with the fundamental group of $B\aut_\partial(M_{g,1})$. The fundamental group surjects onto $\Gamma_g$, and under the isomorphism \eqref{eq:tanre} the action on the higher homotopy groups corresponds to the evident action of $\Gamma_g\subset G_g(\QQ)$ on the right hand side. Note that the Chevalley-Eilenberg cohomology $\HH_{CE}^*(\gl_g)$ inherits an action of $\Gamma_g$.
\begin{theorem} \label{thm:aut cohomology}
Let $2d\geq 6$. The stable cohomology of the homotopy automorphisms of $M_{g,1}$ is given by
$$\HH^*(B\aut_\partial(M_{\infty,1});\QQ) \cong \HH^*(\Gamma_\infty;\QQ) \tensor \HH_{CE}^*(\gl_\infty)^{\Gamma_\infty}.$$
\end{theorem}
The situation for block diffeomorphisms is similar. Let
$$\Pi =\QQ\set{\pi_i}{4i>d} \quad \big( = \pi_{*+d}(BO)\tensor \QQ \big),$$
be the graded vector space with basis elements $\pi_i$ in degree $4i-d>0$. Next, let
\begin{equation} \label{eq:a_g}
\al_g = s^{-1}\Pi \tensor \rHH_d(M_{g,1};\QQ),
\end{equation}
considered as an abelian Lie algebra. In the notation of Theorem \ref{thm:models}, we have that $\tau = 0$ and $\delta = 0$ and moreover the action of $\Der_\omega^+ \LL(V)$ on the reduced cohomology $\rHH^*(M_{g,1};\QQ)$ is trivial for degree reasons. It follows that the higher homotopy groups of the block space are given by
\begin{equation} \label{eq:block space}
\pi_{*+1} B\tDiff_{\partial,\circ}(M_{g,1}) \tensor \QQ \cong \gl_g \oplus \al_g,
\end{equation}
and again the fundamental group acts through the projection onto $\Gamma_g$.
\begin{theorem} \label{thm:block cohomology}
Let $2d\geq 6$. The stable cohomology of the block diffeomorphism group of $M_{g,1}$ is given by
$$\HH^*(B\tDiff_\partial(M_{\infty,1});\QQ) \cong \HH^*(\Gamma_\infty;\QQ) \tensor \HH_{CE}^*(\gl_\infty\oplus \al_\infty)^{\Gamma_\infty}.$$
\end{theorem}
Thus, the calculation of the stable cohomology is reduced to the calculation of the cohomology of arithmetic groups and invariant Lie algebra cohomology.

The stable rational cohomology of arithmetic groups was computed by Borel in \cite{Borel1}. For $\Gamma_g$ the result reads
$$\HH^*(\Gamma_\infty;\QQ) = \QQ[x_1,x_2,\ldots],$$
where $|x_i| = 4i-2$ when $d$ is odd and $|x_i| = 4i$ when $d$ is even.

Serendipitously, the invariant Lie algebra cohomology has been considered by Kontsevich, though for entirely different purposes. Indeed, at least for $d$ odd, the Lie algebra $\Der_\omega\LL(V_g)$ is the same as the one studied by Kontsevich in his work on formal non-commutative symplectic geometry \cite{Kontsevich1,Kontsevich2}. Extending Kontsevich's result, we find that the fixed set of the Chevalley-Eilenberg cochains,
$$C_{CE}^*(\gl_g\oplus\al_g)^{\Gamma_g},$$
admits an interpretation in terms of graphs, which we describe next.

For $s,k\geq 0$, let $\GG(s)_k$ denote the rational vector space spanned by connected graphs with $k$ vertices of valence $\geq 3$, decorated by elements of the cyclic Lie operad, and $s$ leaves labeled by $1,\ldots,s$. The graphs are moreover equipped with orientations of the vertices and of the internal edges. There is an action of the symmetric group $\Sigma_s$ given by permuting the leaf labels. Kontsevich's differential
$$
\partial \colon \GG(s)_k \to \GG(s)_{k-1},
$$
is defined as a sum over edge contractions. The subcomplex $\GG(0)$ spanned by graphs without leaves is Kontsevich's original graph complex. There is a decomposition
$$\GG(s) = \bigoplus_{n\geq 0} \GG(n,s),$$
where $\GG(n,s) \subseteq \GG(s)$ is the subcomplex spanned by graphs $G$ with $\rank \HH_1(G) = n$. 
We remark that $\GG$ is closely related to the dual of the `Feynman transform' of the Lie operad \cite{GK2}. 

If $W$ is a graded vector space, then let $W[n]$ or $s^nW$ denote the graded vector space with $W[n]_i = W_{i-n}$. Define the suspension $\Sigma \GG(s)$ by
$$\Sigma \GG(s) = \bigoplus_n (\Sigma\GG)(n,s),\quad (\Sigma\GG)(n,s) = \GG(n,s)[2(n-1)+s]\tensor \sgn_s,$$
and let $\GG^d = \Sigma^{d-1}\GG$. For a graded vector space $W$, we define
$$\GG^d[W] = \bigoplus_{s\geq 0} \GG^d(s) \tensor_{\Sigma_s} W^{\tensor s}.$$
With this notation, we establish isomorphisms
\begin{align} \label{eq:graphs}
C_*^{CE}(\gl_\infty)_{\Gamma_\infty} & \cong \Lambda\GG^d(0), \\
C_*^{CE}(\gl_\infty\oplus \al_\infty)_{\Gamma_\infty} & \cong \Lambda\GG^d[\Pi],
\end{align}
where $\Pi$ is the graded vector space from \eqref{eq:a_g}, and $\Lambda W$ denotes the free graded commutative algebra on $W$. Moreover, in each case the Chevalley-Eilenberg differential on the left hand side corresponds to Kontsevich's differential. (The isomorphism \eqref{eq:graphs} for $d=1$ is equivalent to Kontsevich's theorem.) This leads to
\begin{theorem} \label{thm:CE}
There are isomorphisms
\begin{enumerate}
\item $\HH_*^{CE}(\gl_\infty)_{\Gamma_\infty} \cong \Lambda\big( \HH_*(\GG^d(0),\partial) \big)$,
\item $\HH_*^{CE}(\gl_\infty \oplus \al_\infty)_{\Gamma_\infty} \cong \Lambda\big( \HH_*(\GG^d[\Pi],\partial) \big)$.
\end{enumerate}
\end{theorem}

The graph homology can in turn be related to the cohomology of automorphism groups of free groups. Building on the work of Culler and Vogtmann \cite{CV}, Kontsevich expressed the graph homology (for $d=1$ and $s=0$) in terms of the cohomology of outer automorphism groups of free groups. This was extended by Conant, Kassabov and Vogtmann \cite{CKV} to include the case $s>0$. Let $A_{n,s}$ be the group of homotopy classes of homotopy equivalences of a bouquet of $n$ circles relative to $s$ marked points. Then
$$A_{n,0} \cong \Out F_n,\quad A_{n,1} \cong \Aut F_n,\quad A_{n,s} \cong F_n^{s-1} \rtimes \Aut F_n,$$
where $F_n$ is the free group on $n$ generators. Note that permutation of the marked points yields an action of $\Sigma_s$ on the homology of $A_{n,s}$.

\begin{theorem}[{Kontsevich ($s=0$), Conant-Kassabov-Vogtmann ($s>0$)}] \label{eq:CKV}
For $n+s\geq 2$ and all $k,d$ there is a $\Sigma_s$-equivariant isomorphism
$$\HH_k(\GG^d(n,s),\partial) \cong \HH^{(2(n-1)+s)d-k}(A_{n,s};\QQ) \tensor \sgn_s.$$
\end{theorem}

Our results should be compared with the known results for the diffeomorphism group. The stable cohomology for $B\Diff_\partial(M_{g,1})$ was calculated in \cite{MW} for $2d=2$, verifying the Mumford conjecture, and in \cite{GRW1} for $2d>4$. We recall the description. Let
$$B\subset \QQ[p_1,\ldots,p_{d-1},e] \quad \big( = \HH^*(BSO(2d);\QQ) \big)$$
be the set of monomials $e^np_{i_1}\ldots p_{i_s}$ of degree $>2d$, with $n\geq 0$ and $d/4 < i_\nu < d$. To each $b\in B$, there is a cohomology class
$$
\kappa_b \in \HH^*(B\Diff_\partial(M_{g,1});\QQ),\quad |\kappa_b| = |b|-2d,
$$
and these classes are multiplicative generators for the stable cohomology.
\begin{theorem}[Madsen-Weiss ($2d=2$), Galatius-Randal-Williams ($2d>4$)]
The stable cohomology of the diffeomorphism group,
$$\HH^*(B\Diff_\partial(M_{\infty,1});\QQ),$$
is freely generated as a graded commutative algebra by the classes
$$\kappa_{e^np_{i_1}\ldots p_{i_s}},$$
where $d/4 < i_\nu < d$ if $n+s\geq 2$, and $d/2< i_1 < d$ if $(n,s) = (0,1)$.
\end{theorem}

We remark in passing that the proof of this result is different in spirit from the proofs of the above theorems and does not give any insights into the homotopy groups. It is an open problem of considerable interest to evaluate $\pi_*\Diff_\partial(M_{g,1})$.

In light of this result, the following reformulation of our main result suggests itself. Theorem \ref{thm:CE} combined with Theorem \ref{eq:CKV} imply that elements
$$\xi \in \HH_*(A_{n,s};\QQ),\quad p_{i_1},\ldots,p_{i_s}\in \Pi^\vee,$$
give rise to cohomology classes
$$
\widetilde{\kappa}_{p_{i_1},\ldots,p_{i_s}}^\xi \in \HH_{CE}^*(\gl_\infty\oplus \al_\infty)^{\Gamma_\infty}$$
of degree $2(n-1)d +4i_1 + \cdots + 4i_s - |\xi|$.
These classes, subject to the equivariance and linearity relations
\begin{align*}
\widetilde{\kappa}_{p_{i_1},\ldots,p_{i_s}}^{a\xi+b\zeta} & = a\widetilde{\kappa}_{p_{i_1},\ldots,p_{i_s}}^{\xi} + b\widetilde{\kappa}_{p_{i_1},\ldots,p_{i_s}}^{\zeta},\quad a,b\in \QQ, \\
\widetilde{\kappa}_{p_{i_1},\ldots,p_{i_s}}^{\sigma\xi} & = \widetilde{\kappa}_{p_{i_{\sigma_1}},\ldots,p_{i_{\sigma_s}}}^{\xi},\quad \sigma\in \Sigma_s,
\end{align*}
are the multiplicative generators of $ \HH_{CE}^*(\gl_\infty\oplus \al_\infty)^{\Gamma_\infty}$. The isomorphisms in Theorem \ref{thm:aut cohomology} and Theorem \ref{thm:block cohomology} are not canonical (see the discussion after Lemma \ref{lemma:collapse}), but after choosing suitable lifts of the generators $\widetilde{\kappa}_{p_{i_1},\ldots,p_{i_s}}^\xi$ to $\HH^*(B\tDiff_\partial(M_{\infty,1});\QQ)$,
our results can be reformulated as follows.

\begin{theorem} \label{thm:stable cohomology}
Let $2d\geq 6$. The stable cohomology of the block diffeomorphism group,
$$\HH^*(B\tDiff_\partial(M_{\infty,1});\QQ),$$
is freely generated as a graded commutative algebra by the Borel classes $x_i$, of degree $4i-2$ if $d$ is odd and $4i$ if $d$ is even, and classes
$$\widetilde{\kappa}_{p_{i_1},\ldots,p_{i_s}}^{\xi},\quad \xi \in \HH_*(A_{n,s}),\quad i_\nu > d/4, \quad  n+s\geq 2,$$
of degree $2(n-1)d +4i_1 + \cdots + 4i_s - |\xi|$.
\end{theorem}

\begin{theorem}
Let $2d\geq 6$. The homomorphism
$$\HH^*(B\aut_\partial(M_{\infty,1});\QQ) \xrightarrow{J^*} \HH^*(B\tDiff_\partial(M_{\infty,1});\QQ)$$
is injective. Its image is the subalgebra freely generated by the classes $\widetilde{\kappa}^\xi$ of degree $2(n-1)d-|\xi|$, for $\xi \in \HH_*(A_{n,0};\QQ) = \HH_*(\Out F_n;\QQ)$, for $n\geq 2$.
\end{theorem}

\begin{remark} \label{rem:dimension four}
The referee has pointed out that Theorem \ref{thm:stable cohomology} might hold also in dimension $2d = 4$, because stable surgery works in dimension 4 by \cite{FQ} and because Theorem 1.5 of \cite{GRW3} does not exclude dimension 4. We leave for the interested reader to work out the details.
\end{remark}

The rational homology of the graph complex $\GG$, or equivalently of the groups $A_{n,s}$, is largely unknown (though see \cite{CHKV} for some recent computations). At any rate, certain classes present themselves immediately. If we let $\epsilon_{n,s}$ denote a generator for $\HH_0(A_{n,s})$, then for $n+s\geq 2$, we have the class $\tilde{\kappa}_{p_{i_1},\ldots,p_{i_s}}^{\epsilon_{n,s}}$
of degree $2(n-1)d + 4i_1 + \ldots + 4i_s$. We note that this is the same as the degree of the class $\kappa_{e^n p_{i_1}\ldots p_{i_s}}$. Thus, the free graded commutative algebra generated by the Borel classes
\begin{equation} \label{eq:Borel classes intro}
x_i,\quad 1\leq i < d/2,
\end{equation}
and the classes
\begin{equation} \label{eq:kappa-tilde classes}
\widetilde{\kappa}_{p_{i_1},\ldots,p_{i_s}}^{\epsilon_{n,s}},\quad d/4 < i_\nu < d, \quad n+s\geq 2,
\end{equation}
is abstractly isomorphic to the stable cohomology of the diffeomorphism group.

\begin{conjecture}
The subalgebra of $\HH^*(B\tDiff_\partial(M_{\infty,1});\QQ)$ generated by the classes \eqref{eq:Borel classes intro} and suitable lifts of the classes \eqref{eq:kappa-tilde classes} maps isomorphically onto the cohomology ring $\HH^*(B\Diff_\partial(M_{\infty,1});\QQ)$ under the homomorphism $I^*$.
\end{conjecture}

It was shown in \cite{ERW} that
$$
I^*\colon \HH^*(B\tDiff_\partial(M_{\infty,1});\QQ) \to \HH^*(B\Diff_\partial(M_{\infty,1});\QQ)
$$
is surjective. It is now an easy count of dimensions to check that
$$
\dim \HH^k(B\Diff_\partial(M_{\infty,1});\QQ) = \dim \HH^k(B\tDiff_\partial(M_{\infty,1});\QQ)
$$
when $k<2d$, and that in degree $2d$, there is a difference in dimensions by $1$.
We conclude that $I^*$ is an isomorphism in degrees $<2d$ and that the kernel in degree $2d$ is $1$-dimensional. Interestingly, the range of degrees where $I^*$ is an isomorphism is greater than expected from the relation of $\ker I^*$ to algebraic $K$-theory \cite{WW}. If the conjecture is true, then the extra element $\widetilde{\kappa}^{\epsilon_{2,0}}$, associated to the generator of $\HH_0(\Out F_2;\QQ)$, could be held responsible for the failure of injectivity in degree $2d$. It is a bit surprising that the homology of the groups $A_{n,s}$ in some sense measures the difference between the cohomology of the block diffeomorphism group and that of the diffeomorphism group.

\subsection*{Acknowledgments}
We thank the referee for many pertinent comments that led to an improvement of the paper.

\section{Quillen's rational homotopy theory}
In this section we will briefly review Quillen's rational homotopy theory \cite{Quillen} and set up a spectral sequence for calculating the rational homology of a simply connected space from its rational homotopy groups. The existence of this spectral sequence was pointed out by Quillen \cite[\S 6.9]{Quillen}, but we need a version that incorporates group actions that are not necessarily base-point preserving, so we need to revisit the construction.

\subsection{Quillen's dg Lie algebra}
The Whitehead products on the homotopy groups of a simply connected based topological space $X$,
$$\pi_{p+1}(X) \times \pi_{q+1}(X) \rightarrow \pi_{p+q+1}(X),$$
endow the rational homotopy groups,
$$\pi_*^\QQ(X) = \pi_{*+1}(X)\tensor \QQ,$$
with the structure of a graded Lie algebra. Rationally homotopy equivalent spaces have isomorphic Lie algebras, but $\pi_*^\QQ(X)$ is not a complete invariant; two spaces may have isomorphic Lie algebras without being rationally homotopy equivalent, as witnessed for instance by $\CP^2$ and $K(\ZZ,2)\times K(\ZZ,5)$.

Quillen \cite{Quillen} constructed a functor $\lambda$ from the category of simply connected based topological spaces to the category of dg Lie algebras and established a natural isomorphism of graded Lie algebras
\begin{equation} \label{eq:h lambda}
\HH_*(\lambda(X)) \cong \pi_*^\QQ(X).
\end{equation}
The quasi-isomorphism type of $\lambda(X)$ is a finer invariant than the isomorphism type of $\pi_*^\QQ(X)$. The main result of Quillen's theory is that it is a complete invariant: two simply connected spaces $X$ and $Y$ are of the same rational homotopy type if and only if the dg Lie algebras $\lambda(X)$ and $\lambda(Y)$ are quasi-isomorphic. Here, we say that two dg Lie algebras are quasi-isomorphic if they are isomorphic in the homotopy category of dg Lie algebras. Concretely, this means that there exists a zig-zag of quasi-isomorphisms that connects them.

\subsection{The Quillen spectral sequence} \label{sec:Quillen ss}
Let $L$ be a dg Lie algebra. The Chevalley-Eilenberg complex of $L$ is the chain complex
$$C_*^{CE}(L) = (\Lambda sL , \delta).$$
Here $\Lambda sL$ denotes the free graded commutative algebra on the suspension of $L$. Elements of $sL$ are denoted $sx$, where $x\in L$, with $|sx| = |x|+1$. The differential $\delta = \delta_0 + \delta_1$ is defined by the following formulas
$$\delta_0(sx_1\wedge \ldots \wedge sx_n) = \sum_{i=1}^n (-1)^{1+\epsilon_i} sx_1 \wedge \ldots sdx_i \ldots \wedge sx_n,$$
$$\delta_1(sx_1\wedge \ldots \wedge sx_n) = \sum_{i<j} (-1)^{|sx_i|+\eta_{ij}} s[x_i,x_j]\wedge sx_1 \wedge \ldots \widehat{sx_i} \ldots \widehat{sx_j} \ldots \wedge sx_n,$$
where
$$\epsilon_i = |sx_1|+\cdots+|sx_{i-1}|,$$
and the sign $(-1)^{\eta_{ij}}$ is determined by graded commutativity;
$$sx_1\wedge \ldots \wedge sx_n = (-1)^{\eta_{ij}}sx_i\wedge sx_j \wedge sx_1 \ldots \widehat{sx_i}\ldots \widehat{sx_j} \ldots \wedge sx_n.$$
We let $\HH_*^{CE}(L)$ denote the homology of this chain complex.

If the differential of $L$ is trivial, then there is a decomposition of the Chevalley-Eilenberg homology as
$$\HH_n^{CE}(L) = \bigoplus_{p+q=n} \HH_{p,q}^{CE}(L),$$
where $\HH_{p,q}^{CE}(L)$ is the homology in word-length $p$ and total degree $p+q$.
$$\xymatrix{\cdots \ar[r] & (\Lambda^{p+1} sL)_q \ar[r]^-{\delta_1} & (\Lambda^p sL)_q \ar[r]^-{\delta_1} & (\Lambda^{p-1} sL)_q \ar[r] & \cdots.}$$
For arbitrary $L$, we may filter the Chevalley-Eilenberg complex by word-length;
$$F_p = \Lambda^{\leq p}sL.$$
The associated spectral sequence has
\begin{equation} \label{eq:dgl ss}
E_{p,q}^2(L) = \HH_{p,q}^{CE}(\HH_*(L)) \Rightarrow \HH_{p+q}^{CE}(L)
\end{equation}
If $L$ is positively graded the filtration is finite in each degree, which ensures strong convergence of the spectral sequence.

There is a coproduct on $\Lambda sL$, called the shuffle coproduct, which is uniquely determined by the requirement that it makes $\Lambda sL$ into a graded Hopf algebra with space of primitives $sL$. The differential $\delta$ is a coderivation with respect to the shuffle coproduct, making $C_*^{CE}(L)$ into a dg coalgebra, and \eqref{eq:dgl ss} is a spectral sequence of coalgebras.

We will now interpret the above for the dg Lie algebra $\lambda(X)$. A fundamental property of Quillen's functor is the existence of a natural isomorphism of graded coalgebras
\begin{equation} \label{eq:hce lambda}
\HH_*^{CE}(\lambda(X)) \cong \HH_*(X;\QQ).
\end{equation}
By \eqref{eq:h lambda} and \eqref{eq:hce lambda} the spectral sequence of Quillen's dg Lie algebra $\lambda(X)$ may be written as follows
\begin{equation} \label{eq:quillen ss}
E_{p,q}^2(X) = \HH_{p,q}^{CE}(\pi_*^\QQ(X)) \Rightarrow \HH_{p+q}(X;\QQ).
\end{equation}
We will refer to this as the Quillen spectral sequence.

\subsection{Functoriality for unbased maps}
It is clear from the construction that the Quillen spectral sequence is natural for base-point preserving maps. But in fact the functoriality can be extended to unbased maps. The homotopy groups $\pi_n(X) = [S^n,X]_*$ depend on the base-point of $X$, and are a priori only functorial for base-point preserving maps. However, if $X$ is simply connected, the canonical map
$$\pi_n(X)\rightarrow [S^n,X]$$
is a bijection, and we may use this to extend $\pi_*^\QQ(X)$ to a functor defined on unbased simply connected spaces. Quillen's functor $\lambda$ can also be extended to unbased maps, but only up to homotopy.

Suppose that $X$ and $Y$ are simply connected spaces with base-points $x_0$ and $y_0$. Given a not necessarily base-point preserving map $f\colon X\rightarrow Y$, we may choose a path $\gamma$ from $y_0$ to $f(x_0)$. Then we obtain based maps
$$(X,x_0) \stackrel{f}{\rightarrow} (Y,f(x_0)) \stackrel{ev_1}{\leftarrow} (Y^I,\gamma) \stackrel{ev_0}{\rightarrow} (Y,y_0).$$
The maps $ev_i$ are weak homotopy equivalences, so the above may be interpreted as a morphism $\overline{f}$ from $(X,x_0)$ to $(Y,y_0)$ in the homotopy category of based spaces. It is easily checked that $\overline{f}$ only depends on the homotopy class of $f$, and that compositions are respected in the sense that $\overline{gf} = \overline{g}\overline{f}$ as maps in the homotopy category.

We may apply Quillen's functor to get a diagram of dg Lie algebras
$$\lambda(X,x_0) \stackrel{f_*}{\rightarrow} \lambda(Y,f(x_0)) \stackrel{(ev_1)_*}{\leftarrow} \lambda(Y^I,\gamma) \stackrel{(ev_0)_*}{\rightarrow} \lambda(Y,y_0),$$
where the maps $(ev_i)_*$ are quasi-isomorphisms. In homology, we obtain an induced morphism of graded Lie algebras
$$(ev_0)_* (ev_1)_*^{-1} f_* \colon \pi_*^\QQ(X) \rightarrow \pi_*^\QQ(Y).$$
Under the identification $\pi_n(X) \cong [S^n,X]$, this map agrees with $f_*\colon [S^n,X]\rightarrow [S^n,Y]$, because $ev_0$ and $ev_1$ are homotopic as unbased maps. Since the spectral sequence \eqref{eq:dgl ss} is natural with respect to morphisms of dg Lie algebras, the above considerations imply the following.

\begin{proposition}
Let $X$ be a simply connected space. There is a spectral sequence of coalgebras
$$E_{p,q}^2 = \HH_{p,q}^{CE}(\pi_*^\QQ(X)) \Rightarrow \HH_{p+q}(X;\QQ).$$
The spectral sequence is natural with respect to unbased maps of simply connected spaces.
\end{proposition}

In particular, if $X$ has a not necessarily base-point preserving action of a group $\pi$, then the Quillen spectral sequence \eqref{eq:quillen ss} is a spectral sequence of $\pi$-modules (from the $E^1$-page and on). An important special case is when $X = \widetilde{Y}$ is the universal cover of a path connected space $Y$ and $\pi$ is the group of deck transformations. By the above, we obtain a spectral sequence of coalgebras with a $\pi$-action,
$$E_{p,q}^2 = \HH_{p,q}^{CE}(\pi_*^\QQ(\widetilde{Y})) \Rightarrow \HH_{p+q}(\widetilde{Y};\QQ).$$
It is an exercise in covering space theory to check that, under the standard identifications
$$\pi \cong \pi_1(Y),\quad \pi_n(\widetilde{Y}) \cong \pi_n(Y),\quad n\geq 2,$$
the action of $\pi$ on $\pi_n(\widetilde{Y})$ obtained as above corresponds to the usual action of $\pi_1(Y)$ on the higher homotopy groups $\pi_n(Y)$.

\subsection{Formality and collapse of the Quillen spectral sequence}
The spectral sequence \eqref{eq:dgl ss} is natural with respect to morphisms of dg Lie algebras. Evidently, a quasi-isomorphism induces an isomorphism from the $E^1$-page and on, so quasi-isomorphic dg Lie algebras have isomorphic spectral sequences. It is also evident that the spectral sequence of a dg Lie algebra with trivial differential collapses at the $E^2$-page. These simple observations have an interesting consequence. Namely, if the dg Lie algebra $L$ is \emph{formal}, meaning that it is quasi-isomorphic to its homology $\HH_*(L)$ viewed as a dg Lie algebra with trivial differential, then the spectral sequence for $L$ collapses at the $E^2$-page. Collapse of the spectral sequence is weaker than formality in general, although the difference is subtle.

\begin{definition} \label{def:rationally perfect}
Let us say that a group $\pi$ is \emph{rationally perfect} if $\HH^1(\pi;V) = 0$ for every finite dimensional $\QQ$-vector space $V$ with an action of $\pi$.
\end{definition}

Note that if $\pi$ is rationally perfect, then every short exact sequence of finite dimensional $\QQ[\pi]$-modules splits, cf.~Appendix B.

\begin{proposition} \label{prop:perfect}
Let $\pi$ be a group acting on a simply connected space $X$ with degree-wise finite dimensional rational cohomology groups. If $\pi$ is rationally perfect and if Quillen's dg Lie algebra $\lambda(X)$ is formal, then there is an isomorphism of graded $\pi$-modules
$$\HH_n(X;\QQ) \cong \bigoplus_{p+q=n} \HH_{p,q}^{CE}(\pi_*^\QQ(X)),$$
for every $n$.
\end{proposition}

\begin{proof}
If the rational cohomology groups of a simply connected space are finite dimensional, then so are the rational homotopy groups. It follows that the Quillen spectral sequence \eqref{eq:quillen ss} is a spectral sequence of finite dimensional $\QQ[\pi]$-modules. Since $\lambda(X)$ is formal, the Quillen spectral sequence collapses, and since $\pi$ is rationally perfect, all extensions relating $E_{*,*}^\infty$ and $\HH_*(X;\QQ)$ are split.
\end{proof}

\begin{remark}
A simply connected space $X$ such that Quillen's dg Lie algebra $\lambda(X)$ is formal is called \emph{coformal} in the literature. The name \emph{formal} is reserved for spaces where Sullivan's minimal model is formal. The two notions are not the same, they are Eckman-Hilton dual. Spaces that are simultaneously formal and coformal can be characterized in terms of Koszul algebras, see \cite{Berglund}.
\end{remark}

\section{Classification of fibrations}
The purpose of this section is to review some fundamental results on the classification of fibrations in the categories of topological spaces and dg Lie algebras.

The classification of fibrations up to fiber homotopy equivalence was pioneered by Stasheff \cite{Stasheff} and given a systematic treatment by May \cite{May}. For a more recent modern approach, see \cite{BC}. The classification of fibrations for dg Lie algebras is implicit in the work of Sullivan \cite{Sullivan} and in a widely circulated preprint of Schlessinger-Stasheff (recently made available \cite{SS}). A detailed account is given in Tanr\'e's book \cite{Tanre}. There is also a more recent approach due to Lazarev \cite{Lazarev}, which uses the language of $L_\infty$-algebras.

\subsection{Fibrations of topological spaces}
Let $X$ be a simply connected space of the homotopy type of a finite CW-complex. Let $\aut(X)$ denote the topological monoid of homotopy automorphisms of $X$, with the compact-open topology, and let $\aut_*(X)$ denote the submonoid of base-point preserving homotopy automorphisms. It is well known that the classifying space $B\aut(X)$ classifies fibrations with fiber $X$. Let us recall the precise meaning of this statement.

An $X$-fibration over a space $B$ is a fibration $E\rightarrow B$ such that for every point $b\in B$ there is a homotopy equivalence $X\rightarrow E_b$. An elementary equivalence between two $X$-fibrations $E\rightarrow B$ and $E'\rightarrow B$ is a map $E\rightarrow E'$ over $B$ such that for every $b\in B$ the induced map $E_b\rightarrow E_b'$ is a homotopy equivalence. We let $\Fib(B,X)$ denote the set of equivalence classes of $X$-fibrations over $B$ under the equivalence relation generated by elementary equivalences.

\begin{theorem}[See \cite{May}] \label{thm:top fibrations}
There is an $X$-fibration,
\begin{equation} \label{eq:universal fibration}
E_X \rightarrow B_X,
\end{equation}
which is universal in the sense that the map
$$[B,B_X] \rightarrow \Fib(B,X), \quad [\varphi]\mapsto [\varphi^*(E_X)\rightarrow B],$$
is a bijection for every space $B$ of the homotopy type of a CW-complex. Furthermore, the universal fibration \eqref{eq:universal fibration} is weakly equivalent to the map
$$B\aut_*(X)\rightarrow B\aut(X)$$
induced by the inclusion of monoids $\aut_*(X)\rightarrow \aut(X)$.
\end{theorem}

\subsection{Fibrations of dg Lie algebras}
There is a parallel story for dg Lie algebras. According to Quillen \cite[\S5]{Quillen}, the category of positively graded dg Lie algebras admits a model structure where the weak equivalences are the quasi-isomorphisms and the fibrations are the maps that are surjective in degrees $>1$. The cofibrations are the `free maps', see \cite[Proposition II.5.5]{Quillen}. In particular, a dg Lie algebra is cofibrant if and only if its underlying graded Lie algebra is free. Schlessinger and Stasheff have given an explicit construction of a classifying space for fibrations in this context, which we now will recall.

Let $L$ be a dg Lie algebra. A derivation of degree $p$ is a linear map $\theta\colon L_*\rightarrow L_{*+p}$ such that
$$\theta [x,y] = [\theta(x),y] + (-1)^{p|x|} [x,\theta(y)],$$
for all $x,y\in L$. The derivations of $L$ are the elements of a dg Lie algebra $\Der L$, whose Lie bracket and differential $D$ are defined by
$$[\theta,\eta] = \theta \circ \eta - (-1)^{|\theta||\eta|}\eta \circ \theta,\quad D(\theta) = d\circ\theta-(-1)^{|\theta|}\theta \circ d,$$
where $d$ is the differential in $L$.

Given a morphism of dg Lie algebras $f\colon L\rightarrow L'$, an \emph{$f$-derivation} of degree $p$ is a map $\theta\colon L_*\rightarrow L_{*+p}'$ such that
$$\theta [x,y] = [\theta(x),f(y)] + (-1)^{p|x|} [f(x),\theta(y)],$$
for all $x,y\in L$. The $f$-derivations assemble into a chain complex $\Der_f(L,L')$, whose differential $D$ is defined by
$$D(\theta) = d_{L'} \circ \theta - (-1)^{|\theta|} \theta \circ d_L.$$
In general there is no natural Lie algebra structure on $\Der_f(L,L')$.

The Jacobi identity for $L$ implies that the map $\ad_x\colon L\rightarrow L$, sending $y$ to $[x,y]$, is a derivation of degree $|x|$ for each $x\in L$. The map $\ad \colon L\rightarrow \Der L$ sending $x$ to $\ad_x$ is a morphism of dg Lie algebras. Let $\Der L \dquot \ad L$ denote the mapping cone of $\ad \colon L\rightarrow \Der L$, i.e.,
$$\Der L \dquot \ad L = sL \oplus \Der L,$$
with differential given by
$$
\tilde{D}(\theta) = D(\theta),\quad  \tilde{D}(sx) = \ad_x - sd(x),
$$
for $\theta \in \Der L$ and $x\in L$. There is a Lie bracket on $\Der L \dquot \ad L$, which is defined as the extension of the Lie bracket on $\Der L$ that satisfies
$$
[\theta,sx] = (-1)^{|\theta|} s\theta(x), \quad [sx,sy] = 0,
$$
for $\theta \in \Der L$ and $x,y\in L$. The Schlessinger-Stasheff classifying dg Lie algebra of $L$ is defined to be the positive truncation,
$$B_L = (\Der L \dquot \ad L)^+.$$
Here, the positive truncation of a dg Lie algebra $L$ is the sub dg Lie algebra $L^+$ with
$$L_i^+ = \left\{ \begin{array}{cc}  L_i,& i\geq 2 \\ \ker(d\colon L_1\rightarrow L_0), & i = 1 \\ 0, & i\leq 0. \end{array} \right.$$

An $L$-fibration over $K$ is a surjective map of dg Lie algebras $\pi\colon E\rightarrow K$ together with a quasi-isomorphism $L\rightarrow \Ker \pi$. An elementary equivalence between two $L$-fibrations $\pi\colon E\rightarrow K$ and $\pi'\colon E'\rightarrow K$ is a quasi-isomorphism of dg Lie algebras $E\rightarrow E'$ over $K$ such that the diagram
$$
\xymatrix{L\ar[r] \ar[dr] & \Ker \pi \ar[d] \\ & \Ker \pi'}
$$
commutes. Let $\Fib(K,L)$ denote the set of equivalence classes of $L$-fibrations over $K$ under the equivalence relation generated by elementary equivalence.

\begin{theorem}[See Tanr\'e {\cite{Tanre}}] \label{thm:dgl fibrations}
Let $L$ be a cofibrant dg Lie algebra and let $B_L$ denote its Schlessinger-Stasheff classifying dg Lie algebra. There is an $L$-fibration
\begin{equation} \label{eq:universal lie}
E_L \rightarrow B_L,
\end{equation}
which is universal in the sense that for every cofibrant dg Lie algebra $K$, the map
$$[K,B_L] \rightarrow \Fib(K,L),\quad [\varphi]\mapsto [\varphi^*(E_L)]$$
is a bijection. Furthermore, the morphism $E_L\rightarrow B_L$ is weakly equivalent to the morphism
$$\Der^+ L \rightarrow (\Der L \dquot \ad L)^+.$$
\end{theorem}

By combining Theorem \ref{thm:top fibrations} and Theorem \ref{thm:dgl fibrations}, together with Quillen's equivalence of homotopy theories between $\Top_{*,1}^\QQ$ and $\DGL_1$, it is not difficult to derive the following consequence.

\begin{corollary}[See {\cite[Corollaire VII.4(4)]{Tanre}}] \label{cor:tanre}
Let $X$ be a simply connected space of the homotopy type of a finite CW-complex. Let $\LL_X$ be a cofibrant model of Quillen's dg Lie algebra $\lambda(X)$. The positive truncation of the morphism of dg Lie algebras
$$\Der \LL_X \rightarrow \Der \LL_X \dquot \ad \LL_X$$
is a dg Lie algebra model for the map of simply connected covers
$$B\aut_*(X)\langle 1 \rangle \rightarrow B\aut(X)\langle 1\rangle.$$
\end{corollary}

\subsection{Relative fibrations}
Given a non-empty subspace $A\subset X$, we may consider the monoid $\aut(X;A)$ of homotopy self-equivalences of $X$ that restrict to the identity map on $A$. We will assume that the inclusion map from $A$ into $X$ is a cofibration. As follows from the theory of \cite{May} (see, e.g., \cite[Appendix B]{HL} for details), the classifying space $B\aut(X;A)$ classifies fibrations with fiber $X$ under the trivial fibration with fiber $A$.

Similarly, for a cofibration of cofibrant dg Lie algebras $K\subset L$, the positive truncation of the dg Lie algebra $\Der(L;K)$ of derivations on $L$ that restrict to zero on $K$, acts as a classifying space for fibrations of dg Lie algebras with fiber $L$ under the trivial fibration with fiber $K$. This result seems not to have appeared in the literature, but the proof is a straightforward generalization of \cite[Chapitre VII]{Tanre}. The following is a consequence.

\begin{theorem} \label{thm:relative classification}
Let $A\subset X$ be a cofibration of simply connected spaces of the homotopy type of finite CW-complexes, and let $\LL_A \subset \LL_X$ be a cofibration between cofibrant dg Lie algebras that models the inclusion of $A$ into $X$. Then the positive truncation of the dg Lie algebra $\Der(\LL_X;\LL_A)$, consisting of all derivations on $\LL_X$ that restrict to zero on $\LL_A$, is a dg Lie algebra model for the simply connected cover of $B\aut(X;A)$.
\end{theorem}

A detailed proof of this result, following a different route, can be found in \cite{BSaleh}.

\subsection{Derivations and mapping spaces} \label{sec:derivations}
Given a morphism of dg Lie algebras $f\colon \LL\to L$, we let $\Der_f(\LL,L)$ denote the chain complex of \emph{$f$-derivations}. Its elements of degree $p$ are by definition all maps $\theta\colon \LL \rightarrow L$ of degree $p$ that satisfy
$$\theta[x,y] = [\theta(x),f(y)] + (-1)^{|x|p} [f(x),\theta(y)]$$
for all $x,y\in \LL$. The differential $D$ is defined by
$$D(\theta) = d_L \circ \theta - (-1)^p \theta \circ d_\LL.$$
We include here a lemma for later reference. It is presumably well known, but we indicate the proof for completeness.
\begin{lemma} \label{lemma:qi}
Let $\phi\colon \LL \to \LL'$ and $\psi\colon L \to L'$ be quasi-isomorphisms of dg Lie algebras. Suppose that $\LL$ and $\LL'$ are cofibrant and concentrated in strictly positive homological degrees.
\begin{enumerate}
\item \label{eq:item 1} For every morphism of dg Lie algebras $f\colon \LL\to L$, the induced chain map
$$\psi_*\colon \Der_f(\LL,L) \to \Der_{\psi f}(\LL,L'), \quad \theta \mapsto \psi\circ \theta,$$
is a quasi-isomorphism.

\item \label{eq:item 2} For every morphism of dg Lie algebras $g\colon \LL' \to L$, the induced chain map
$$\phi^*\colon \Der_g(\LL',L)\to \Der_{g\phi}(\LL,L),\quad \eta\mapsto \eta \circ \phi,$$
is a quasi-isomorphism.
\end{enumerate}
\end{lemma}

\begin{proof}
There is a complete filtration,
$$\Der_f(\LL,L)  = F^1 \supseteq F^2 \supseteq \cdots,$$
where $F^p$ consists of those $f$-derivations $\theta\colon \LL\to L$ that vanish on $\LL_{<p}$. This filtration gives rise to a spectral sequence with
$$E_2^{p,-q} = \Hom\big(\HH_p(Q\LL),\HH_{q}(L)\big) \Rightarrow \HH_{-p+q}(\Der_f(\LL,L)).$$
Here $Q\LL = \LL/[\LL,\LL]$ denotes the chain complex of indecomposables in the dg Lie algebra $\LL$. It is well-known that a morphism $\phi\colon \LL\to \LL'$ between positively graded cofibrant dg Lie algebras is a quasi-isomorphism if and only if the induced map on indecomposables $Q\phi\colon Q\LL\to Q\LL'$ is a quasi-isomorphism (see, e.g., \cite[Proposition 22.12]{FHT-RHT}). Bearing this in mind, both claims may be deduced through an application of the comparison theorem for spectral sequences.
\end{proof}

Let $G$ be a topological group with the neutral element $e$ as base-point. The Samelson product
$$\pi_p(G)\times \pi_q(G)\rightarrow \pi_{p+q}(G)$$
is a natural operation on the homotopy groups of $G$. It may be defined as follows. Given based maps $f\colon S^p\rightarrow G$ and $g\colon S^q\rightarrow G$, the composite map
$$\xymatrix{S^p\times S^q \ar[r]^-{f\times g} & G\times G \ar[r]^-{[-,-]} & G},$$
where $[-,-]\colon G\times G\to G$ is the commutator $[x,y] = xyx^{-1}y^{-1}$,
is trivial when restricted to $S^p\vee S^q$. It therefore induces a based map $[f,g]\colon S^{p+q} \cong S^p\times S^q/S^p\vee S^q \rightarrow G$. The homotopy class of $[f,g]$ is the Samelson product of the classes $[f]$ and $[g]$.

The map $G\rightarrow G$ sending $x$ to $gxg^{-1}$ preserves the base-point, and defines a homomorphism $\phi_g\colon \pi_k(G) \rightarrow \pi_k(G)$. This defines an action of the group $\pi_0(G)$ on $\pi_k(G)$, and this action preserves Samelson products. Under the standard isomorphism $\pi_{k+1}(BG) \cong \pi_k(G)$, the Whitehead product on $\pi_{*+1}(BG)$ corresponds to the Samelson product on $\pi_*(G)$, and the standard action of $\pi_1(BG)$ on $\pi_{k+1}(BG)$ corresponds to action of $\pi_0(G)$ on $\pi_k(G)$ described above, see \cite{Whitehead}. The above holds true for $G$ a group-like topological monoid, because every such may be replaced by a homotopy equivalent group. In particular it applies to monoids of homotopy automorphisms.

\begin{theorem}[Lupton-Smith {\cite[Theorem 3.1]{LS}}] \label{thm:lupton-smith}
Let $f\colon X\rightarrow Y$ be a map between simply connected CW-complexes with $X$ a finite CW-complex, and let $\varphi\colon \LL_X\rightarrow \LL_Y$ be a Lie model for $f$. There is a natural isomorphism for all $k\geq 2$,
\begin{equation} \label{eq:aut der}
\beta \colon \pi_k(\map_*(X,Y),f)\tensor \QQ \xrightarrow{\cong} \HH_k(\Der_\varphi(\LL_X,\LL_Y)),
\end{equation}
If $f=id_X$, it is valid also for $k=1$.
\end{theorem}

\begin{proof}
We indicate the definition of $\beta$, following \cite{LS} (with a minor modification), and refer the reader to \cite{LS} for a proof that it is a bijection. Let $Z\ltimes X$ denote the half-smash product $(Z\times X) / (Z \times *)$ and let $i\colon X\to Z\ltimes X$ denote the map sending $x\in X$ to the class of $(*,x)$. If $\LL_X = (\LL V,\delta)$, then $S^k\ltimes X$ has dg Lie model $(\LL(V \oplus s^k V),\delta')$, where $\delta'$ is determined by the conditions that
\begin{itemize}
\item the inclusion $\iota\colon (\LL V,\delta)\to (\LL(V \oplus s^k V),\delta')$ is a chain map, and
\item the $\iota$-derivation $s^k\colon \LL V \to \LL(V\oplus s^k V)$ that extends $v\mapsto s^k v$ satisfies $\delta' s^k = (-1)^k s^k \delta$.
\end{itemize}
Given a map $h\colon S^k\rightarrow \map_*(X,Y)$ sending the base-point of $S^k$ to $f$, there is an adjoint map $h^\#\colon S^k \ltimes X \to Y$ such that $h^\# \circ i = f$.
By \cite[Proposition A.3]{LS} we can find a dg Lie model
$$\psi_h\colon (\LL V,\delta)\to  (\LL(V \oplus s^k V),\delta')$$
for $h^\#$ such that $\psi_h \circ \iota = \varphi$.
The composite $\theta_h=\psi_h \circ s^k \colon \LL_X \to \LL_Y$ is then a $k$-cycle in the chain complex $\Der_\varphi(\LL_X,\LL_Y)$, and one sets
$$\beta[h] = [\theta_h].$$
\end{proof}

We will need the following addendum to Theorem \ref{thm:lupton-smith}.
\begin{proposition} \label{prop:samelson}
Under the isomorphism
\begin{equation} \label{eq:aut der lie}
\beta\colon \pi_k(\aut_*(X),id_X)\tensor \QQ \cong \HH_k(\Der(\LL_X)),
\end{equation}
the Samelson product corresponds to the Lie bracket on derivations.
\end{proposition}

For the proof we will use the following lemma.

\begin{lemma} \label{lemma:commutator}
Let $G$ be a topological group and let $f\colon S^p\to G$ and $g\colon S^q\to G$ be based maps. The map $S^p\times S^q \xrightarrow{f\times g} G\times G \xrightarrow{[-,-]} G$ is homotopic to the composite
$$\{f,g\}\colon S^p\times S^q \xrightarrow{\Xi} S^p\times S^q \times S^p \times S^q \xrightarrow{f\times g\times f\times g} G\times G\times G\times G \xrightarrow{\mu} G,$$
where $\Xi(x,y) = (x,y,m_p(x),m_q(y))$, where $m_k$ denotes a degree $-1$ map on $S^k$ and $\mu$ is the multiplication map.
\end{lemma}

\begin{proof}
This follows readily from the fact that the inverse map $j\colon G\to G$, $x\mapsto x^{-1}$, induces multiplication by $(-1)$ on $\pi_k(G)$, i.e., the diagram
$$
\xymatrix{S^k \ar[r]^-f \ar[d]^-{m_k} & G \ar[d]^-j \\ S^k \ar[r]^-f & G}
$$
commutes up to homotopy for every based map $f\colon S^k \to G$.
\end{proof}

\begin{proof}[Proof of Proposition \ref{prop:samelson}]
Let $f\colon S^p \to \aut_*(X)$, $g\colon S^q \to \aut_*(X)$ be based maps.
It follows from Lemma \ref{lemma:commutator} that the Samelson product $[f,g]$ is characterized up to homotopy by homotopy commutativity of the diagram
$$
\xymatrix{S^p\times S^q \ar[r]^-{\{f,g\}} \ar[d]_-c & \aut_*(X) \\
S^{p+q} \ar[ur]_-{[f,g]}}
$$
or, equivalently, of the diagram
$$
\xymatrix{\big(S^p\times S^q\big) \ltimes X \ar[r]^-{\{f,g\}^\#}  \ar[d]_-{c\ltimes 1} & X \\ S^{p+q}\ltimes X. \ar[ur]_-{[f,g]^\#}}
$$
One checks that the diagram
\begin{equation} \label{eq:pre-samelson}
\xymatrix{\big(S^p \times S^q\big) \ltimes X \ar[r]^-{\{f,g\}^\#} \ar[d]_-{\Xi\ltimes 1} & X \\
\big(S^p \times S^q\times S^p \times S^q\big)\ltimes X \ar[r]^-\cong & S^p \ltimes \big( S^q \ltimes \big( S^p \ltimes \big( S^q \ltimes X \big) \big) \big) \ar[u]|-{f^\# \circ (1\ltimes g^\#) \circ (1\ltimes 1 \ltimes f^\# ) \circ (1\ltimes 1\ltimes 1 \ltimes g^\#)}}
\end{equation}
is commutative. By iterated use of $\big(S^k\times Y) \ltimes X \cong S^k\ltimes (Y \ltimes X)$ and the dg Lie model for $S^k \ltimes X$ described above, one works out that the dg Lie model for $Z\ltimes X$, where $Z$ is a product of spheres, has the form $\big(\LL(H_*(Z)\tensor V),\delta'')$. Moreover, one finds that the map $\Xi \ltimes 1$ has dg Lie model
$$\Xi_* \colon \LL(H_*(S^p \times S^q)\tensor V) \to \LL(H_*(S^p \times S^q \times S^p\times S^q)\tensor V)$$
induced by $\Xi_*\colon H_*(S^p \times S^q) \to H_*(S^p \times S^q \times S^p\times S^q)$ (we omit the details, but this is true because the map $\Xi$ is formal).
After picking Lie models $\psi_f$ and $\psi_g$ for $f^\#$ and $g^\#$ as in the proof of Theorem \ref{thm:lupton-smith}, one sees that a Lie model for the right vertical map in \eqref{eq:pre-samelson} is given by
$$\gamma\colon \LL(H_*(S^p \times S^q \times S^p\times S^q)\tensor V) \to \LL V,$$
$$\gamma((a\times b\times c\times d) v) = \psi_f(a\psi_g(b\psi_f(c\psi_g(dv)))),$$
for homology classes $a,c\in H_*(S^p)$, $b,d,\in H_*(S^q)$. It follows that we may take
$$\psi_{\{f,g\}} \colon \LL(H_*(S^p \times S^q)\tensor V) \to \LL V$$
to be the composite $\gamma\circ\Xi_*$. Explicitly, for $v\in V$,
\begin{align*}
\psi_{\{f,g\}}((s^p\times s^q)v) & = \gamma(\Xi_*(s^p\times s^q)v) \\
& = \gamma((s^p\times s^q\times 1\times 1)v - (-1)^{pq}(1\times s^q\times s^p \times 1)v  \\
& \quad - (s^p\times 1 \times 1 \times s^q)v + (1\times 1\times s^p \times s^q)v) \\
& = \theta_f\theta_g(v) - (-1)^{pq}\theta_g\theta_f(v) - \theta_f \theta_g(v) + \theta_f\theta_g(v) \\
& = [\theta_f,\theta_g](v),
\end{align*}
and similar calculations show
$$\psi_{\{f,g\}}(v) = v, \quad \psi_{\{f,g\}}(s^pv) = 0, \quad \psi_{\{f,g\}}(s^q v) = 0.$$
In particular, the morphism $\psi_{\{f,g\}}$ factors through the morphism induced by the collapse map $c_*$,
$$
\xymatrix{\LL(H_*(S^p\times S^q)\tensor V) \ar[r]^-{\psi_{\{f,g\}}} \ar[d]_-{c_*} & \LL V \\
\LL(H_*(S^{p+q})\tensor V), \ar@{-->}[ur]_-\lambda}
$$
and we may take $\psi_{[f,g]}$ to be $\lambda$. Thus, for $v\in V$ we get
$$\theta_{[f,g]}(v) = \psi_{[f,g]}(s^{p+q}v) = \psi_{\{f,g\}}((s^p\times s^q)v) = [\theta_f,\theta_g](v),$$
which proves the proposition.

\end{proof}

\subsection{Homotopy automorphisms of manifolds} \label{sec:ham}
Let $M^n$ be a simply connected compact manifold with boundary $\partial M = S^{n-1}$.
Let $\aut_\partial(M)$ denote the topological monoid of homotopy automorphisms of $M$ that restrict to the identity on $\partial M$, with the compact-open topology.
Let $\aut_{\partial,\circ}(M)$ denote the connected component of the identity. There is a homotopy fibration sequence
$$
B\aut_{\partial,\circ}(M) \to B\aut_\partial(M) \to B\pi_0(\aut_\partial(M)).
$$
Hence, up to homotopy $B\aut_{\partial,\circ}(M)$ may be identified with the simply connected cover of $B\aut_\partial(M)$. The goal of this section is to establish a tractable dg Lie algebra model for $B\aut_{\partial,\circ}(M)$.

An \emph{inner product space of degree $n$} is a finite dimensional graded vector space $V$ together with a degree $-n$ map of graded vector spaces,
$$V\tensor V \to \QQ,\quad x\tensor y \mapsto \langle x,y\rangle,$$
which is non-singular in the sense that the adjoint map,
$$V \to \Hom(V,\QQ),\quad x\mapsto \langle x, - \rangle,$$
is an isomorphism of graded vector spaces (of degree $-n$).
Note that $\langle x, y\rangle$ is automatically zero unless $|x|+|y| = n$.

We call an inner product space as above \emph{graded symmetric} if
$$\langle x,y \rangle = (-1)^{|x||y|} \langle y,x \rangle,$$
for all $x,y\in V$ and \emph{graded anti-symmetric} if
$$\langle x,y \rangle = - (-1)^{|x||y|} \langle y,x \rangle,$$
for all $x,y\in V$.

For example, if $M^n$ is a simply connected compact manifold with boundary $\partial M = S^{n-1}$, then
the reduced homology $H = \rHH_*(M;\QQ)$ together with the intersection form is a graded
symmetric inner product space of degree $n$. The desuspension of the reduced rational homology,
$$V = s^{-1} H,$$
becomes a graded anti-symmetric inner product space of degree $n-2$ by setting
$$\langle s^{-1} e,s^{-1} f \rangle = (-1)^{|e|} \langle e, f\rangle.$$

Now, let $V$ be a graded anti-symmetric inner product space of degree $n-2$ and choose a graded basis $\alpha_1,\ldots,\alpha_r$. The dual basis $\alpha_1^\#,\ldots,\alpha_r^\#$ is characterized by
$$\langle \alpha_i,\alpha_j^\# \rangle = \delta_{ij}.$$
There is a canonical element $\omega = \omega_V\in V^{\tensor 2}$ defined by
$$\omega = \sum_i \alpha_i^\# \tensor \alpha_i.$$
Up to sign, the element $\omega$ corresponds to the inner product $\langle -,-\rangle\in \Hom(V^{\tensor 2},\QQ)$ under the isomorphism $V^{\tensor 2} \cong \Hom(V^{\tensor 2},\QQ)$ induced by the inner product on $V^{\tensor 2}$;
$$\langle v\tensor w,v'\tensor w' \rangle = (-1)^{|v'||w|} \langle v,v'\rangle \langle w,w'\rangle.$$
Indeed, one checks that $\langle \omega, x\tensor y \rangle = (-1)^{|x||y|+|x|+1} \langle x,y\rangle$. In particular, $\omega$ is independent of the choice of basis.
Since $V$ is anti-symmetric, the transposition $\tau$ acts by $\tau \omega = -\omega$. This implies that $\omega$ may be written as a sum of graded commutators $[x,y] = x\tensor y - (-1)^{|x||y|} y\tensor x$ as follows:
\begin{equation} \label{eq:omega}
\omega = \frac{1}{2} \sum_i \big[\alpha_i^\#,\alpha_i\big].
\end{equation}
In this way, $\omega$ may be regarded as an element of the free graded Lie algebra $\LL V$.

Let $\Der \LL V$ denote the graded Lie algebra of derivations on $\LL V$. 
Consider the map of degree $2-n$,
$$\theta_{-,-} \colon \LL V \tensor V \to \Der \LL V, \quad \theta_{\xi,x}(y) = \xi \langle x,y \rangle.$$
Since the form is non-degenerate and since a derivation on a free graded Lie algebra is determined by its values on generators, the map $\theta_{-,-}$ is an isomorphism.

The following proposition plays a key role.

\begin{proposition} \label{prop:bracket}
Let $V$ be a graded anti-symmetric inner product space with canonical element $\omega\in \LL V$. The diagram
$$\xymatrix{\Der \LL V \ar[r]^-{ev_\omega} & \LL V \\ \LL V\tensor V \ar[u]^-{\theta_{-,-}} \ar[ur]_-{[-,-]}}$$
is commutative.
\end{proposition}

\begin{proof}
Note that every element $x\in V$ may be written as
\begin{equation} \label{eq:x}
x= \sum_i \langle x,\alpha_j^\#\rangle \alpha_j.
\end{equation}
If $\theta$ is a derivation, then
$$\theta(\omega) = \sum_i \big[\theta(\alpha_i^\#),\alpha_i\big].$$
To see this, first use \eqref{eq:omega} to get
$$\theta(\omega) =  \frac{1}{2} \sum_i \left( \big[\theta(\alpha_i^\#),\alpha_i\big] + (-1)^{|\theta||\alpha_i^\#|}\big[\alpha_i^\#,\theta(\alpha_i)\big]\right).$$
Rewriting the right summands using graded anti-symmetry of the bracket, \eqref{eq:x} on $x=\alpha_i^\#$, and then \eqref{eq:x} on $x=\alpha_j^\#$ backwards, we get
\begin{align*}
\sum_i (-1)^{|\theta||\alpha_i^\#|}\big[\alpha_i^\#,\theta(\alpha_i)\big]
& = \sum_i (-1)^{|\alpha_i||\alpha_i^\#|+1} \big[\theta(\alpha_i),\alpha_i^\#\big] \\
& = \sum_{i,j} (-1)^{|\alpha_i||\alpha_i^\#|+1} \big[ \theta(\alpha_i),\langle \alpha_i^\#,\alpha_j^\# \rangle \alpha_j \big] \\
& = \sum_j \big[ \theta\big( \sum_i (-1)^{|\alpha_i||\alpha_i^\#|+1} \langle \alpha_i^\#,\alpha_j^\# \rangle \alpha_i\big), \alpha_j\big] \\
& = \sum_j \big[\theta(\alpha_j^\# ),\alpha_j\big].
\end{align*}
Thus,
$$ev_\omega(\theta_{\xi,x})  = \sum_i \big[\theta_{\xi,x}(\alpha_i^\#),\alpha_i\big] = \sum_i \big[\xi \langle x, \alpha_i^\# \rangle,\alpha_i\big]  = \big[\xi,\sum_i \langle x,\alpha_i^\# \rangle \alpha_i\big] = \big[\xi,x\big].$$
\end{proof}

\begin{corollary} \label{cor:imev}
The image of the map $ev_\omega \colon \Der \LL V \to \LL V$ is the space of decomposables $[\LL V, \LL V]$.
In other words, for every $\zeta\in \LL^{\geq 2} V$, there is a derivation $\theta$ on $\LL V$ such that $\theta(\omega) = \zeta$.
\end{corollary}

The following result is essentially due to Stasheff \cite{Stasheff2}.

\begin{theorem} \label{thm:stasheff}
Let $M^n$ be a simply connected compact manifold with boundary $\partial M = S^{n-1}$ and let $V$ denote the graded anti-symmetric inner product space $s^{-1}\widetilde{\HH}_*(M;\QQ)$.
There is a differential $\delta$ on $\LL V$ such that 
\begin{enumerate}
\item $\big( \LL V, \delta \big)$ is a minimal Quillen model for $M$.
\item The canonical element $\omega \in \LL V$ is a cycle that represents $(-1)^n$ times the homotopy class of the inclusion of the boundary.
\end{enumerate}
\end{theorem}

\begin{proof}
Consider the closed manifold $X=M\cup_\partial D^n$. Fix an orientation of $X$ and let $\mu \in \HH_n(X;\QQ)$ be the fundamental class. Choose a basis $e_1,\ldots,e_r$ for $\widetilde{\HH}_*(M;\QQ)$, and let $e_1^\#,\ldots,e_r^\#$ be the dual basis with respect to the intersection form, in the sense that $\langle e_i,e_j^\# \rangle = \delta_{ij}$. Identifying $\widetilde{\HH}_*(X;\QQ) = \widetilde{\HH}_*(M;\QQ) \oplus \QQ\mu$, the reduced diagonal of the fundamental class assumes the form
\begin{equation} \label{eq:diagonal class}
\overline{\Delta}(\mu) = \sum_{i=1}^r e_i^\#\tensor e_i.
\end{equation}
To derive this expression, one can use that the intersection form on $\HH_*(X;\QQ)$ satisfies
$$\langle x\cap \mu , y\cap \mu \rangle = \langle x\cup y, \mu\rangle = \langle x\tensor y, \Delta(\mu)\rangle,$$
for cohomology classes $x,y\in \HH^*(X;\QQ)$, where $- \cap \mu\colon \HH^k(X;\QQ) \to \HH_{n-k}(X;\QQ)$ is the Poincar\'e duality isomorphism and $\langle -,-\rangle$ in the right hand side denotes the standard pairing between cohomology and homology.
Alternatively, it can be derived from Theorem 11.11 and Problem 11-C of \cite{MS}

The minimal Quillen model of $M$ has the form $\big(\LL V,\delta\big)$ and the cell attachment $M\to X$ is modeled by a free map of dg Lie algebras
$$\big(\LL V, \delta\big) \to \big( \LL(V\oplus s^{-1}\mu) , \delta \big),$$
where $\delta(s^{-1}\mu) \in \LL V$ represents the attaching map for the top cell, i.e., the class of $S^{n-1} = \partial M \to M$.

It is well-known that the quadratic part $\delta_1$ of the differential in the minimal Quillen model corresponds to the reduced diagonal in the sense that
$$\delta_1(s^{-1} x) = (s^{-1} \tensor s^{-1}) \overline{\Delta}(x),$$
see e.g.~\cite[Corollary 2.14]{Baues-Lemaire}. In particular,
$$\delta_1(s^{-1}\mu) = \sum_{i=1}^r (-1)^{|e_i^\#|} s^{-1}e_i^\#\tensor s^{-1}e_i.$$
If we choose $\alpha_i = s^{-1}e_i$ as our basis for $V$, then
$$\langle s^{-1}e_i, s^{-1} e_j^\# \rangle = (-1)^{|e_i|} \langle e_i,e_j^\#\rangle$$
shows that the dual basis is given by $\alpha_i^\# = (-1)^{|e_i|}s^{-1}e_i^\#$. Hence,
$$\delta_1(s^{-1}\mu) = \sum_{i=1}^r (-1)^{|e_i^\#| + |e_i|} \alpha_i^\#\tensor \alpha_i = (-1)^n \omega.$$

It is an important observation due to Stasheff \cite[Theorem 2]{Stasheff2} that one may assume that $\delta(s^{-1}\mu)$ is purely quadratic. We give a proof for completeness. The key ingredient is Corollary \ref{cor:imev}.

Write $\delta = \delta_1 + \delta_2 + \delta_3 +\cdots$, where $\delta_k$ increases bracket length by exactly $k$.
By Corollary \ref{cor:imev}, there exists a derivation $\theta$ on $\LL V$ such that
$$\theta(\omega) = (-1)^n \delta_2(s^{-1}\mu).$$
We may assume that $\theta$ increases bracket length by exactly $1$. Extend $\theta$ to a derivation on $\LL(V\oplus s^{-1}\mu)$ by setting
$\theta(s^{-1} \mu) = 0$.
Then
$$e^{\theta} = \sum_{k\geq 1} \frac{1}{k!}\theta^k$$
is a Lie algebra automorphism of $\LL(V\oplus s^{-1}\mu)$, and one checks that
$$\delta' = e^{-\theta} \circ \delta \circ e^\theta$$
is a new differential such that $\delta_1' = \delta_1$ and $\delta_2'(s^{-1}\mu) = 0$. Clearly,
$$e^\theta \colon \big( \LL(V\oplus s^{-1}\mu), \delta' \big) \to \big( \LL(V\oplus s^{-1}\mu), \delta \big)$$
is an isomorphism of dg Lie algebras. If $\delta_3'(s^{-1}\mu) \ne 0$, we continue in a similar way by finding a derivation $\theta'$ such that $\theta'(\omega) = (-1)^n \delta_3'(s^{-1}\mu)$, obtaining an isomorphism
$$e^{\theta'} \colon \big( \LL(V\oplus s^{-1}\mu), \delta'' \big) \to \big( \LL(V\oplus s^{-1}\mu), \delta' \big),$$
where $\delta_1'' = \delta_1$, $\delta_2''(s^{-1}\mu) = 0$ and $\delta_3''(s^{-1}\mu) = 0$.
In this way, the non-zero higher terms of $\delta(s^{-1}\mu)$ may be peeled off one at a time.
The process will stop after finitely many steps. Indeed, since $V$ is concentrated in positive homological degrees, the bracket length of any term of $\delta^{(r)}(s^{-1}\mu)$ will be at most $n-2$, i.e., $\delta_k^{(r)}(s^{-1}\mu) = 0$ for $k\geq n-2$, and, by construction, $\delta_k^{(r)}(s^{-1}\mu) = 0$ for $k=2,3,\ldots,r+1$. Thus, we can stop after $r=n-4$ steps.
\end{proof}

Let $\Der_\omega \LL V$ denote the graded Lie subalgebra of $\Der \LL V$ consisting of those derivations $\theta$ such that $\theta(\omega) = 0$. Note that the differential $\delta$ in the minimal Quillen model for $M$ (see Theorem \ref{thm:stasheff}) is an element of $\Der_\omega \LL V$ of degree $-1$. Therefore, $[\delta,-]$ is a differential on $\Der_\omega \LL V$, making it a dg Lie algebra. We let $\big(\Der_\omega \LL V, [\delta,-] \big)^+$ denote the positive truncation of this dg Lie algebra, i.e., it agrees with $\Der_\omega \LL V$ in degrees $>1$, and in degree $1$ it is the kernel of the differential $[\delta,-]$.

\begin{theorem} \label{thm:aut model}
Let $M$ be a simply connected compact manifold with boundary $S^{n-1}$.
A dg Lie algebra model for the classifying space $B\aut_{\partial,\circ}(M)$ is given by
$$\big(\Der_\omega \LL V, [\delta,-] \big)^+.$$
\end{theorem}

\begin{proof}
We will use Theorem \ref{thm:relative classification}.
Let $\rho$ be a generator of degree $n-2$. By Theorem \ref{thm:stasheff}, the morphism of dg Lie algebras
$$\varphi\colon \LL(\rho) \to \big(\LL V,\delta\big),\quad \varphi(\rho) = (-1)^n \omega,$$
is a model for the inclusion of $\partial M$ into $M$. However, it is not a cofibration (i.e.~free map) of dg Lie algebras. To rectify this, we factor $\varphi$ as a free map $q$ followed by a surjective quasi-isomorphism $p$ as follows:
\begin{equation} \label{eq:factorization}
\LL(\rho) \stackrel{q}{\rightarrow} \big( \LL(V,\rho,\gamma),\delta \big) \stackrel{p}{\rightarrow} \big(\LL V, \delta\big).
\end{equation}
Here $q$ is the obvious inclusion, the map $p$ is defined by $p|_V = id_V$, $p(\rho) = (-1)^n\omega$, and $p(\gamma) = 0$, and the differential $\delta$ is extended to $\rho$ and $\gamma$ by $\delta(\rho) = 0$ and $\delta(\gamma) = (-1)^n \omega - \rho$.
To simplify notation, denote the sequence \eqref{eq:factorization} by
$$
\LL_{\partial M} \xrightarrow{q} \widetilde{\LL}_M \xrightarrow{p} \LL_M.
$$
Now, the map $q\colon \LL_{\partial M}\to \widetilde{\LL}_M$ is a cofibration that models the inclusion of $\partial M$ into $M$. By Theorem \ref{thm:relative classification} the dg Lie algebra
\begin{equation} \label{eq:dg model}
\Der^+(\widetilde{\LL}_M; \LL_{\partial M})
\end{equation}
models $B\aut_{\partial,\circ}(M)$. We will show it is quasi-isomorphic to $\big( \Der_\omega \LL V, [\delta,-] \big)^+$.

There is a pullback diagram of chain complexes
\begin{equation} \label{eq:pullback}
\xymatrix{\Der(p;\LL_{\partial M}) \ar[d]_{pr_2} \ar[r]^-{pr_1} & \Der(\widetilde{\LL}_M;\LL_{\partial M}) \ar[d]^-{p_*} \\
\Der(\LL_M;\LL_{\partial M}) \ar[r]_-{p^*} & \Der_p(\widetilde{\LL}_M,\LL_M;\LL_{\partial M}),}
\end{equation}
where $\Der(\LL_M;\LL_{\partial M})$ denotes the dg Lie algebra of derivations $\eta$ on $\LL_M$ such that $\eta \circ \varphi= 0$, $\Der(p;\LL_{\partial M})$ denotes the chain complex of pairs $(\theta,\eta)$ of derivations $\theta\in \Der(\widetilde{\LL}_M;\LL_{\partial M})$, $\eta\in \Der(\LL_M;\LL_{\partial M})$, with $p_*(\theta) = p^*(\eta)$, and $\Der_p(\widetilde{\LL}_M,\LL_M;\LL_{\partial M})$ denotes the chain complex of $p$-derivations $\theta\in \Der_p(\widetilde{\LL}_M,\LL_M)$ such that $\theta \circ q = 0$. As the reader may check, taking componentwise Lie brackets turns $\Der(p;\LL_{\partial M})$ into a dg Lie algebra and the projections $pr_1(\theta,\eta) = \theta$, $pr_2(\theta,\eta) = \eta$ into morphisms of dg Lie algebras.

Below we will argue that the map $p_*$ in \eqref{eq:pullback} is a surjective quasi-isomorphism. Surjective quasi-isomorphisms of chain complexes are stable under pullbacks, so this will imply that $pr_2$ is a surjective quasi-isomorphism. We will also argue that the map $p^*$ in \eqref{eq:pullback} is a quasi-isomorphism in positive degrees. This, together with the fact that $p_*$ and $pr_2$ are quasi-isomorphisms, will imply that $pr_1$ is a quasi-isomorphism in positive degrees. Taking positive truncations, we obtain a zig-zag of quasi-isomorphisms of dg Lie algebras,
$$\xymatrix{\Der^+(\widetilde{\LL}_M;\LL_{\partial M}) & \ar[l]^-{pr_1}_-\sim \Der^+(p;\LL_{\partial M}) \ar[r]_-{pr_2}^-\sim & \Der^+(\LL_M;\LL_{\partial M}).}$$
The dg Lie algebra $\Der^+(\LL_M;\LL_{\partial M})$ is clearly the same as $\big(\Der_\omega \LL V, [\delta,-]\big)^+$, so this will finish the proof.

To see that $p_*$ is a surjective quasi-isomorphism, consider the following diagram:
$$
\xymatrix{
0 \ar[r] & \Der(\widetilde{\LL}_M;\LL_{\partial M}) \ar[d]^-{p_*} \ar[r] & \Der \widetilde{\LL}_M \ar[d]^-{p_*} \ar[r]^-{q^*} & \Der_q(\LL_{\partial M},\widetilde{\LL}_M) \ar[d]^-{p_*} \ar[r] & 0 \\
0 \ar[r] & \Der_p(\widetilde{\LL}_M,\LL_M;\LL_{\partial M}) \ar[r] & \Der_p(\widetilde{\LL}_M,\LL_M)  \ar[r]^-{q^*} & \Der_\varphi(\LL_{\partial M},\LL_M) \ar[r] & 0
}
$$
Exactness at the left and middle terms is clear. That the maps labeled by $q^*$ are surjective follows because $q$ is a free map. The middle and right vertical maps are quasi-isomorphisms by Lemma \ref{lemma:qi}. A five lemma argument then shows that the left vertical map is a quasi-isomorphism. To see that it is surjective, note that $p$ admits a section $\sigma\colon \LL_M\to \widetilde{\LL}_M$. Given $\theta \in \Der_p(\widetilde{\LL}_M,\LL_M;\LL_{\partial M})$, let $\eta$ be the unique derivation on $\widetilde{\LL}_M$ that agrees with $\sigma\theta$ when restricted to the generators. Then one checks that $\eta\in \Der(\widetilde{\LL}_M;\LL_{\partial M})$ and that $p_*(\eta) = \theta$.

To see that the map $p^*$ in \eqref{eq:pullback} is a quasi-isomorphism in positive degrees, consider the following diagram:
{\small
$$
\xymatrix{0\ar[r] &\Der(\LL_M;\LL_{\partial M}) \ar[r] \ar[d]_-{p^*} & \Der \LL_M \ar[r]^-{\varphi^*} \ar[d]_-{p^*} & \Der_\varphi (\LL_{\partial M},\LL_M) \ar@{=}[d] \ar[r] & 0 \\
0 \ar[r] &\Der_p(\widetilde{\LL}_M,\LL_M;\LL_{\partial M}) \ar[r] & \Der_p(\widetilde{\LL}_M,\LL_M) \ar[r]^-{q^*} & \Der_\varphi(\LL_{\partial M},\LL_M) \ar[r] & 0.}
$$}
Exactness of the bottom row and at the left and middle terms of the top row follows as above. The middle vertical map is a quasi-isomorphism by Lemma \ref{lemma:qi}. 

Here is the crux of the proof: The top right horizontal map $\varphi^*$ is surjective in non-negative degrees. Indeed, up to a sign, $\varphi^*$ may be identified with the map $ev_\omega \colon \Der \LL V \to s^{n-2}\LL V$. Since the space of generators $V = s^{-1}\widetilde{\HH}_*(M;\QQ)$ is concentrated in degrees $\leq n-2$, every element of positive degree in $s^{n-2} \LL V$ is necessarily decomposable, whence in the image of $ev_\omega$ by Proposition \ref{prop:bracket}. Finally, a five lemma argument shows that the left vertical map is a quasi-isomorphism in positive degrees.
\end{proof}

\begin{corollary} \label{cor:coformal}
If $M$ is formal and the reduced rational cohomology ring has trivial multiplication, then $B\aut_{\partial,\circ}(M)$ is coformal and there is an isomorphism of graded Lie algebras
$$\pi_*^\QQ(B\aut_{\partial,\circ}(M)) \cong \Der_\omega^+ \LL V.$$
\end{corollary}

\begin{proof}
Formality is equivalent to the property that the minimal Quillen model may be chosen to have purely quadratic differential, i.e., $\delta(V) \subseteq \LL^2(V)$. The quadratic part of the differential in any minimal Quillen model corresponds to the cup product on the reduced cohomology.
\end{proof}

For example, the conditions hold if $M$ is $(d-1)$-connected of dimension at most $3d-2$ for some $d\geq 2$. In particular they hold for the manifolds $M_{g,1}$.

\section{Block diffeomorphisms} \label{sec:block diffeomorphisms}
This section examines the classifying space of block diffeomorphisms of simply connected smooth compact manifolds with boundary. We assume throughout that $\dim M \geq 5$ so that the $h$-cobordism theorem and surgery theory are available. Moreover, we assume that the manifolds under consideration are oriented and that the maps between them are orientation preserving unless otherwise specified.

\subsection{The surgery fibration} \label{sec:surgery}
This paragraph briefly reviews the surgery fibration introduced in F.~Quinn's thesis \cite{Quinn}. There is a detailed account of the surgery fibration for topological manifolds in the memoir \cite{Nicas}. In contrast, we work in the smooth category and treat only simply connected manifolds. The proofs of \cite{Nicas} carry over to the smooth category with only minor changes.

Recall that the topological group $\tDiff(M)$ of block diffeomorphisms of a manifold $M$ is defined to be the geometric realization of the $\Delta$-group $\tDiff(M)_\bullet$, whose $k$-simplices are the self-diffeomorphisms $\varphi$ of $\Delta^k\times M$ that preserve the faces in the sense that $\varphi(F \times M) = F\times M$ for every face $F$ of $\Delta^k$, cf.~\cite{RS1,RS2}. We assume furthermore that $\varphi$ preserves a collar of each face in which case $\tDiff(M)_\bullet$ is a Kan $\Delta$-group by \cite{BLR}. See \S\ref{sec:partial linearization} below for more details.

The $\Delta$-group of self-diffeomorphisms that preserve the projection onto $\Delta^k$ is equal to the set of smooth singular $k$-simplices of the diffeomorphism group $\Diff(M)$, equipped with the Whitney topology. Since the geometric realization $|\sing_\bullet(X)|$ is weakly homotopy equivalent to $X$, one can view $\tDiff(M)$ as an enlargement of $\Diff(M)$.

If $M$ has a non-empty boundary, then we consider the subgroup of boundary preserving block diffeomorphisms, $\tDiff_\partial(M)$, which is defined as above but with the additional requirement that each $k$-simplex $\varphi$ restricts to the identity on a neighborhood of $\Delta^k\times \partial M$.

If we replace the face preserving self-diffeomorphisms with face preserving homotopy automorphisms in the above definition, we obtain the $\Delta$-set of block homotopy equivalences, $\taut_\partial(M)_\bullet$, with geometric realization $\taut_\partial(M)$. As above, the subcomplex of self-homotopy equivalences over $\Delta^k$ may be identified with the set of singular $k$-simplices in the space $\aut_\partial(M)$ of self-homotopy equivalences equipped with the compact-open topology. In contrast to the case of diffeomorphisms, the inclusion
$$\sing_\bullet \aut_\partial(M) \to \taut_\partial(M)_\bullet$$
induces a homotopy equivalence of geometric realizations. Indeed, if $f\colon \Delta^k\times M \to \Delta^k \times M$ is a face preserving map, then $pr_{\Delta^k} \circ f$ is equivalent to a map
$$f_1\colon M\to \widetilde{G}(\Delta^k)$$
into the monoid of face preserving endomorphisms of $\Delta^k$. But $\widetilde{G}(\Delta^k)$ contracts to $\{id_{\Delta^k}\}$: if $\alpha\colon \Delta^k \to \Delta^k$ is face preserving, then
$$(1-t) \alpha + t id_{\Delta^k}\in \widetilde{G}(\Delta^k)$$
contracts to the identity, and induces a retraction
$$\taut_\partial(M)_\bullet \to S_\bullet(\aut_\partial(M)).$$
This induces a homotopy equivalence of geometric realizations
\begin{equation} \label{eq:taut}
\aut_\partial(M) \simeq \taut_\partial(M).
\end{equation}
In the sequel, we prefer to use the notation $\aut_\partial(M)$ instead of $\taut_\partial(M)$.

Let $M$ be a simply connected $n$-dimensional manifold with simply connected boundary, $n\geq 5$. The product of $M$ with the standard $k$-simplex, $\Delta^k\times M$, is a $(k+3)$-ad with faces of dimension $n+k-1$,
$$
\partial_0 \Delta^k\times M, \quad \partial_1 \Delta^k\times M,\quad \ldots,\quad  \partial_k \Delta^k\times M, \quad \Delta^k\times \partial M.
$$
There are three Kan $\Delta$-sets associated to $M$ and maps between them,
\begin{equation} \label{eq:kan}
\Ss_\partial^{G/O}(M)_\bullet \xrightarrow{\eta_\bullet} \mathcal{N}_\partial^{G/O}(M)_\bullet \xrightarrow{\lambda_\bullet} \LL(M)_\bullet.
\end{equation}
The $k$-simplices of the (smooth) structure space $\Ss_\partial^{G/O}(M)_\bullet$ are pairs $(W,f)$ of a (smooth) $(k+3)$-ad $W$ and a face preserving homotopy equivalence
$$
f\colon W \to \Delta^k \times M
$$
with the additional property that
$$\partial_{k+1}W \to \Delta^k\times \partial M$$
is a diffeomorphism.

Let $K\gg 0$ and let $(M,\partial M)\subseteq (D^{n+K}, \partial D^{n+K})$ be an embedding where $n=\dim M$. The set $\mathcal{N}_\partial^{G/O}(M)_k^K$ consists of embeddings of $(k+3)$-ads
$$
W\subset \Delta^k \times D^{n+K},\quad d_i W = W\cap \partial_i(\Delta^k \times D^{n+K}), \quad i\leq k,
$$
and $\partial_{k+1}W = \Delta^k\times \partial M$, which are collared near each of these faces, together with a $K$-dimensional vector bundle $\zeta^K$ over $\Delta^k \times M$, and a commutative bundle diagram of $(k+3)$-ads,
$$
\xymatrix{\nu^K(W) \ar[d] \ar[r]^-{\widehat{f}} & \zeta^K \ar[d] \\ W \ar[r]^-f & \Delta^k \times M,}
$$
such that the two conditions below are satisfied.
\begin{enumerate}
\item $f$ and its faces $d_i(f)$ have degree one for $i = 0,\ldots, k$ and $\partial_{k+1} W \to \Delta^k \times \partial M$ is a diffeomorphism.
\item $\widehat{f}$ is a fiberwise isomorphism of vector bundles.
\end{enumerate}
Then $\mathcal{N}_\partial^{G/O}(M)_\bullet^K$ is a $\Delta$-set and taking the colimit as $K\to \infty$ using $D^{n+K}\subset D^{n+K+1}$ one obtains the $\Delta$-set in \eqref{eq:kan}.

The third term $\LL(M)_\bullet$ of \eqref{eq:kan} and the $\Delta$-maps $\eta_\bullet$ and $\lambda_\bullet$ are defined in \S2.2 of \cite{Nicas}. The homotopy groups of $\LL(M)_\bullet$ depend only on the fundamental group of $M$ and the dimension $n = \dim M$. In the simply connected case the only non-zero groups are
\begin{equation} \label{eq:L-groups}
\pi_{4k-n}\LL(M)_\bullet = \ZZ,\quad \pi_{4k+2-n} \LL(M)_\bullet = \ZZ/2\ZZ.
\end{equation}

\begin{theorem}[Nicas, Quinn]
The sequence \eqref{eq:kan} of Kan $\Delta$-sets is a homotopy fibration.
\end{theorem}
This follows from Theorem 2.3.4 of \cite{Nicas}, which states that the homotopy fiber of $\lambda_\bullet$ is homotopy equivalent to $\Ss_\partial^{G/O}(M)_\bullet$.

We have left to identify $\mathcal{N}_\partial^{G/O}(M)_\bullet$ in more homotopy theoretic terms. This goes back to Sullivan's proof of the Hauptvermutung for manifolds, outlined in \cite{Rourke}. The result is
\begin{theorem} \label{thm:kan-eq}
There is a homotopy equivalence of Kan $\Delta$-sets
$$
\mathcal{N}_\partial^{G/O}(M)_\bullet \xrightarrow{\simeq} \sing_\bullet \map_*(M/\partial M,G/O),
$$
where $S_\bullet$ denotes the singular $\Delta$-set.
\end{theorem}

The combination of the two theorems above leads to
\begin{corollary} \label{cor:surgery fibration}
There is a homotopy fibration of geometric realizations
\begin{equation} \label{eq:surgery fibration}
\Ss_\partial^{G/O}(M) \xrightarrow{\eta} \map_*(M/\partial M,G/O) \xrightarrow{\lambda} \LL(M).
\end{equation}
\end{corollary}
This is the form of the surgery fibration, listed without proof in \S17A of \cite{W}.

We note in passing and for future reference that since $\map_*(M/\partial M,G/O)$ is an infinite loop space and therefore nilpotent, \cite[Proposition 4.4.1 (i)]{MP} or \cite[Theorem II.2.2]{HMR} applied to the surgery fibration \eqref{eq:surgery fibration} implies that the components of the structure space $\Ss_\partial^{G/O}(M)$ are nilpotent.  We are grateful to the referee for pointing this out to us.

The $k$-th homotopy group of $\Ss_\partial^{G/O}(M)$, with the identity as the basepoint, is the structure set of equivalence classes of homotopy equivalences
$$
f\colon (W,\partial W) \to (D^k \times M, \partial(D^k\times M))
$$
with $\partial f$ a diffeomorphism; $(W,f)$ is equivalent to $(W',f')$ if there exists a diffeomorphism $\varphi$ from $W$ to $W'$ for which
$$
\xymatrix{(W,\partial W) \ar[d]^-\varphi \ar[dr]^-f \\ (W',\partial W') \ar[r]^-{f'} & (D^k\times M,\partial(D^k\times M))}
$$
is homotopy commutative.

For use later in the paper we need to recall the calculation of
\begin{equation} \label{eq:eta}
\eta_*\colon \pi_k(\Ss_\partial^{G/O}(M),id) \to \pi_k(\map_*(M/\partial M,G/O),*).
\end{equation}
Since $\Ss_\partial^{G/O}(M)_\bullet$ is a Kan $\Delta$-set, an element of $\pi_k(\Ss_\partial^{G/O}(M),id)$ is represented by a homotopy equivalence
$$(V,\partial V) \to (D^k\times M, \partial(D^k\times M))$$
which is a diffeomorphism on the boundary. If $M$ is simply connected we may take $V=D^k \times M$. This is a consequence of the $h$-cobordism theorem, as explained in \S3.2 of \cite{BM}.
Suppose more generally that
$$(W,\partial W) \xrightarrow{(f,\partial f)} (X, \partial X)$$
is a pair of a homotopy equivalence $f$ of smooth $n$-manifolds with $\partial f$ a diffeomorphism. We need a description of
$$\eta(f,\partial f) \in [X/\partial X,G/O]_*.$$
To this end, pick a homotopy inverse pair,
$$(g,(\partial f)^{-1})\colon (X,\partial X) \to (W,\partial W),$$
and define $\zeta =  g^*(\nu(W))$, where $\nu(W)$ is the normal bundle of an embedding $(W,\partial W) \subset (\RR^{K+n}, \RR^{K+n-1})$ with $K\gg 0$. Since $f^*(\zeta) \cong \nu(W)$ by an isomorphism which is unique up to homotopy, we obtain a normal map
$$
\xymatrix{\nu(W) \ar[r]^-{\widehat{f}} \ar[d] & \zeta \ar[d] \\ W \ar[r]^-f & X.}
$$
Let
$$
c_W\colon (D^{K+n},S^{K+n-1}) \to (\Th(\nu(W)),\Th(\nu(W)|_{\partial W}))
$$
be the collapse map. The composition of $c_W$ with $\widehat{f}$ induces a reduction (degree one map)
$$
c_\zeta \colon (D^{K+n},S^{K+n-1}) \to (\Th(\zeta),\Th(\zeta|_{\partial X})),
$$
which we compare to the canonical reduction $c_X$ of $\nu(X)$. The restriction of $\zeta$ to $\partial X$ is identified with $((\partial f)^{-1})^*(\nu(\partial W))$ and the normal derivative (see the remarks following Lemma \ref{lemma:4.4} below) induces a linear isomorphism
$$\partial t = D^\nu(\partial f) \colon \zeta|_{\partial X} \xrightarrow{\cong} \nu(X)|_{\partial X}.$$
The Atiyah-Wall uniqueness theorem extends $\partial t$ to a proper homotopy equivalence
$$(t,\partial t) \colon (\zeta,\zeta|_{\partial X}) \to (\nu(X),\nu(X)|_{\partial X})$$
compatible with the two reductions, cf.~\cite{Wall4,W}. Let $\xi = \zeta \oplus \tau(X)$ and let $\theta = t\oplus id_{\tau(X)}$. The restriction
$\partial \theta = \partial t \oplus id$ defines a framing of $\xi|_{\partial X}$ and induces a bundle $\xi/\partial \theta$ over $X/\partial X$. Moreover, $\theta$ defines a proper homotopy equivalence $\overline{\theta}\colon \xi/\partial \theta \to \epsilon_X^{n+K}$. The pair $(\xi/\partial \theta,\overline{\theta})$ is classified by $G/O$, providing a unique element
$$\eta[f,\partial f] \in [X/\partial X,G/O]_*.$$
Under the map induced by $j\colon G/O\to BO$,
$$
j_*\colon [X/\partial X,G/O]_* \to [X/\partial X,BO]_*,
$$
the element $j_*\eta(f,\partial f)$ classifies the bundle $\xi/\partial \theta$.

We have left to explain the linear isomorphism
$$\partial t\colon \zeta|_{\partial X} \to \nu(\partial X).$$
To this end we consider a diffeomorphism $\varphi\colon M\to N$ of closed $m$-manifolds. We choose embeddings of $M$ and $N$ in $\RR^{m+K}$ with normal bundles $\nu(M)$ and $\nu(N)$ and make the associated identifications
$$
\nu(M) \oplus \tau(M) = \epsilon_M^{K+m},\quad \nu(N) \oplus \tau(N) = \epsilon_N^{K+m},
$$
with the $(K+m)$-dimensional product bundles.

Given vector bundles $\xi$ and $\eta$ over the space $X$, let $\Bun(\xi,\eta)$ denote the space of fiberwise isomorphisms. It is the space of sections in the fiber bundle $GL(\xi,\eta)$ over $X$, whose fiber at $x\in X$ is the space of isomorphisms from $\xi_x$ to $\eta_x$. We are interested in the set of homotopy classes in $\Bun(\xi,\eta)$, or equivalently the connected components of $\Gamma(X,GL(\xi,\eta))$. For $K$ sufficiently large it turns out that $\Bun(\nu(M),\varphi^*(\nu(N)))$ is non-empty and we consider the map
$$\alpha\colon \Bun(\nu(M),\varphi^*(\nu(N)))\to \Bun(\nu(M)\oplus \varphi^*(\tau(N)),\epsilon_M^{K+m})$$
which sends $a\colon \nu(M) \to \varphi^*(\nu(N))$ into
$$\nu(M)\oplus \varphi^*(\tau(N)) \xrightarrow{a\oplus id} \varphi^*(\nu(N)) \oplus \varphi^*(\tau(N)) = \epsilon_M^{K+m}.$$

\begin{lemma} \label{lemma:4.4}
For $K$ sufficiently large, the map $\alpha$ induces a bijection on homotopy classes.
\end{lemma}

\begin{proof}
For $x\in M$, consider the diagram
$$
\xymatrix{GL(\nu_x M,\nu_{\varphi(x)} N) \ar[r]^-{\alpha_x} \ar[d]^-{\beta_x \alpha_x} & GL(\nu_x M \oplus \tau_{\varphi(x)} N,\RR^{K+m}) \ar[dl]_-{\beta_x} \ar[d]_-{\alpha_x \beta_x} \\
GL(\nu_x M\oplus \RR^{K+m},\nu_{\varphi(x)} N\oplus \RR^{K+m}) \ar[r]^-{\alpha_x} & GL(\nu_x M, \tau_{\varphi(x)} N \oplus \RR^{K+m}, \RR^{2(K+m)}),}
$$
where $\beta_x$ adds the identity of $\nu_{\varphi(x)}N$ to $\nu_x M\oplus \tau_{\varphi(x)} N \to \RR^{K+m}$. The space $GL(\nu(M)\oplus \varphi^*(\tau (N)),\epsilon_M^{K+m})$ is non-empty because it contains the map
$$\nu(M) \oplus \varphi^*(\tau(N)) \xrightarrow{id \oplus (D \varphi)^{-1}} \nu(M) \oplus \tau(M) = \epsilon_M^{K+m}.$$
It follows that $GL(\nu(M) \times \RR^{K+m}, \varphi^*(\nu(N))\times \RR^{K+m})$ is non-empty and therefore that $GL(\nu(M), \varphi^*(\nu(N)))$ is non-empty for $K$ sufficiently large. Moreover, $\beta_x \circ \alpha_x$ and $\alpha_x \circ \beta_x$ are homotopy equivalences in a range of dimensions depending on $K$. It follows that
$$\alpha\colon \Bun(\nu(M),\varphi^*(\nu(N)))\to \Bun(\nu(M)\oplus \varphi^*(\tau(N)),\epsilon_M^{K+m})$$
defines a bijection of homotopy classes for large $K$.
\end{proof}

A bundle map $a\colon \nu(M) \to \varphi^*(\nu(N))$ is called a \emph{normal derivative} of $\varphi$ if $\alpha(a)$ is homotopic to $id\oplus (D \varphi)^{-1}\colon \nu(M) \oplus \varphi^*(\tau(N)) \to \epsilon^{K+m}$. We use the notation $D^\nu \varphi$ for such a bundle map, noting that it exists only for $K\gg 0$ and is determined only up to homotopy.

Returning to the normal invariants above, we have
$\zeta|_{\partial X} = \psi^*(\nu(\partial W))$, where $\psi = (\partial f)^{-1}$ and
$$\partial t = \psi^*(D^\nu \varphi)\colon \psi^*(\nu(\partial W)) \to \nu(\partial X).$$
We observe that the framing
\begin{equation} \label{eq:framing}
\partial \theta = \partial t \oplus id \colon \psi^*(\nu(\partial W)) \oplus \tau(\partial X) \to \epsilon^{K+n-1}
\end{equation}
is homotopic to $id\oplus D\psi$.

\subsection{Fundamental homotopy fibrations} \label{sec:fhf}
For a smooth simply connected manifold $M$ with boundary, we consider the following string of grouplike monoids of homotopy automorphisms
$$
\aut_{\partial,\circ}(M) \subset \aut_{\partial,J}(M) \subset \aut_\partial(M).
$$
Here, $\aut_{\partial,\circ}(M)$ denotes the connected component of the identity, and $\aut_{\partial,J}(M)$ the larger monoid of connected components in the image of the forgetful map
$$
J\colon \pi_0 \tDiff_\partial(M) \to \pi_0 \aut_\partial(M).
$$
Let $\tDiff_{\partial,\circ}(M) \subset \tDiff_\partial(M)$ denote the union of those components that map to $\aut_{\partial,\circ}(M)$. Then
\begin{equation*} \label{eq:(22)}
\aut_{\partial,\circ}(M)/\tDiff_{\partial,\circ}(M) \simeq \aut_{\partial,J}(M)/\tDiff_\partial(M) \simeq \Ss_\partial^{G/O}(M)_{(1)}.
\end{equation*}

We have a diagram in which each square is a homotopy pullback,
\begin{equation} \label{eq:(23)}
\xymatrix{\Ss_\partial^{G/O}(M)_{(1)} \ar[r] \ar[d] & {*} \ar[d] \\ B\tDiff_{\partial,\circ}(M) \ar[r] \ar[d] & B\aut_{\partial,\circ}(M) \ar[d] \ar[r] & {*} \ar[d] \\ B\tDiff_\partial(M) \ar[r] & B\aut_{\partial,J}(M) \ar[r] & B\big(\pi_0\aut_{\partial,J}(M)\big).}
\end{equation}
Notice also that $B\aut_{\partial,\circ}(M)$ is a simply connected cover of $B\aut_\partial(M)$.

\begin{proposition} \label{prop:F(M)}
Let $M$ be an $n$-dimensional smooth compact simply connected manifold with boundary $\partial M \cong S^{n-1}$ ($n\geq 5$). There is a homotopy fibration
$$\map_*(S^n,Top/O) \to \Ss_\partial^{G/O}(M) \xrightarrow{q^*\circ \eta} \map_*(M,G/O),$$
where $q\colon M\to M/\partial M$ is the quotient map.
\end{proposition}

\begin{proof}
Let us first recall that the structure space for $S^n$ is homotopy equivalent to $\map_*(S^n,Top/O)$. Indeed, the surgery obstruction map depends only on the underlying topological situation in the sense that there is a homotopy commutative diagram
$$
\xymatrix{\map_*(S^n,G/O) \ar[d] \ar[r]^-\lambda & \LL(S^n) \\ \map_*(S^n,G/Top). \ar[ur]_-{\lambda^{Top}} }
$$
The map $\lambda^{Top}$ is a homotopy equivalence since the its homotopy fiber --- the topological structure space $\Ss_\partial^{G/Top}(D^n)$ --- is contractible (by Alexander's trick). Apply the homotopy fibration
$$Top/O \to G/O \to G/Top$$
to identify the homotopy fiber of
$$\lambda\colon \map_*(S^n,G/O) \to \map_*(S^n,G/Top) \simeq \LL(S^{n})$$
with $\map_*(S^n,Top/O)$.

Next, consider the diagram
\begin{equation} \label{eq:pb}
\xymatrix{\map_*(S^n,Top/O) \ar[d] \ar@{-->}[r] & \Ss_\partial^{G/O}(M) \ar[d]^-\eta \ar[r] & {*} \ar[d] \\
\map_*(S^n,G/O) \ar[d] \ar[r]^-{c^*} & \map_*(M/\partial M,G/O) \ar[d]^-{q^*} \ar[r]^-\lambda & \LL(M) \\
{*} \ar[r] & \map_*(M,G/O).}
\end{equation}
The upper right square is the surgery fibration for $M$ and the bottom left square is induced by the homotopy cofiber sequence
$$M\xrightarrow{q} M/\partial M \xrightarrow{c} S^n,$$
where $q$ is the quotient map and $c$ collapses onto the top cell.

The composite $\lambda\circ c^*\colon \map_*(S^n,G/O)\to \LL(M)\,\, (= \LL(S^n))$ may be identified up to homotopy with the surgery obstruction map for $S^n$, whose homotopy fiber was just recalled to be $\map_*(S^n,Top/O)$. This implies firstly the existence of the dashed map and, secondly, that the top left square is a homotopy pullback. It follows that the rectangle formed by the two squares to the left is a homotopy pullback. This finishes the proof.
\end{proof}

\begin{corollary} \label{cor:ss}
With hypothesis as in Proposition \ref{prop:F(M)}, the map
$$j_*\circ q^*\circ \eta \colon \Ss_\partial^{G/O}(M)_{(1)} \to \map_*(M,BO)_{(0)}$$
is a rational homotopy equivalence. Here the right hand side denotes the connected component containing the constant map and $j\colon G/O \to BO$ is the canonical map.
\end{corollary}

\begin{proof}
By Proposition \ref{prop:F(M)}, the $k$-th homotopy group of the homotopy fiber of the map $q^*\circ \eta\colon \Ss_\partial^{G/O}(M) \to \map_*(M,G/O)$ is given by
$$\pi_k\map_*(S^n,Top/O) \cong \Theta_{k+n},$$
the group of homotopy spheres of dimension $k+n$, which is finite by the celebrated result of Kervaire and Milnor \cite{KM}.
Secondly, since the homotopy groups of $\map_*(M,BG)$ are finite, the homotopy fiber of the map $j_*\colon \map_*(M,G/O)\to \map_*(M,BO)$ has finite homotopy groups. These facts imply the result.
\end{proof}

We end this subsection by showing that the space $B\tDiff_{\partial,\circ}(M)$ is nilpotent, which ensures that it has a well-behaved rationalization. To do this, we use the following lemma.
\begin{lemma} \label{lemma:nilpotent}
Let $F\xrightarrow{i} E\xrightarrow{p} B$ be a fibration of connected spaces. If $F$ is nilpotent and $B$ is simply connected, then $E$ is nilpotent.
\end{lemma}

\begin{proof}
The exact sequence $\pi_1 F \to \pi_1 E \to \pi_1 B = 0$ shows that $\pi_1E$ is a quotient of a nilpotent group, hence nilpotent itself. Next, for $k\geq 2$, note that $\pi_k F \to \pi_k E \to \pi_k B$ is an exact sequence of $\pi_1 E$-modules such that $\gamma \cdot \alpha = i_*(\gamma) \cdot \alpha$ for $\gamma\in \pi_1 F$ and $\alpha\in \pi_k F$, and furthermore $\xi \cdot \beta = p_*(\xi) \cdot \beta$ for $\xi\in \pi_1 E$ and $\beta\in \pi_k B$, see e.g.~\cite[Proposition 1.5.4]{MP}. Since $\pi_k F$ is a nilpotent $\pi_1 F$-module, this implies that $\pi_k F$ is also nilpotent as a $\pi_1 E$-module, and since $B$ is simply connected it implies that $\pi_k B$ is trivial, hence nilpotent, as a $\pi_1 E$-module. It follows from \cite[Proposition I.4.3]{HMR} that $\pi_k E$ is nilpotent as a $\pi_1 E$-module.
\end{proof}

\begin{proposition} \label{prop:nilpotent}
The space $B\tDiff_{\partial,\circ}(M)$ is nilpotent.
\end{proposition}

\begin{proof}
As noted just after Corollary \ref{cor:surgery fibration}, the space $\Ss_\partial^{G/O}(M)_{(1)}$ is nilpotent. We can then apply the previous lemma to the fibration sequence
$$
\Ss_\partial^{G/O}(M)_{(1)} \to B\tDiff_{\partial,\circ}(M) \to B\aut_{\partial,\circ}(M).
$$
\end{proof}

\subsection{A partial linearization of $\tDiff_\partial(M)$} \label{sec:partial linearization}
This section studies the Jacobian of the topological group $\tDiff_\partial(M)$. The end result is a map of monoids
$$D \colon \tDiff_{\partial,\circ}(M) \to \aut_{\partial,\circ}(\tau_M^S)$$
where $\aut_{\partial,\circ}(\tau_M^S)$ is a certain space of automorphisms of the stable tangent bundle of the compact manifold with boundary $M$ and $\tDiff_{\partial,\circ}(M)\subseteq \tDiff_\partial(M)$ is the set of components of diffeomorphisms that are homotopic to the identity. We begin with a description of the target.

To a vector bundle $\xi$ over a finite $CW$-complex $X$, with projection map $E\to X$, we associate the topological monoid $\aut(\xi)$ of diagrams
$$
\xymatrix{E \ar[r]^-{\widehat{f}} \ar[d] & E \ar[d] \\ X \ar[r]^-f & X,}
$$
with $f$ a homotopy equivalence and $\widehat{f}$ a fiberwise isomorphism, topologized as a subset of $\aut(X)\times \aut(E)$.
For subspaces $D\subseteq C\subseteq X$, we have a submonoid $\aut_C^D(\xi)\subseteq \aut(\xi)$ consisting of those $(f,\widehat{f})\in \aut(\xi)$ such that $f$ restricts to the identity on $C$ and $\widehat{f}$ restricts to the identity on $\xi|_D$. When $C=D$ we write $\aut_C(\xi) = \aut_C^C(\xi)$. Let $\aut_{C,\circ}^D(\xi)\subseteq \aut_C^D(\xi)$ be the submonoid of those $(f,\widehat{f})$ for which $f$ is homotopic to the identity map. We will assume that the inclusions $D\to C \to X$ are cofibrations.

We let $GL(\xi) \to X$ be the fiber bundle whose fiber at $x\in X$ is the group $GL(\xi_x)$ of linear isomorphisms of the fiber $\xi_x$, and write $\Gamma_D(X,GL(\xi))$ for the space of sections that map points $x\in D$ to the identity isomorphism of $\xi_x$.

\begin{lemma} \label{lemma:serre fibration}
The map
$$\pi\colon \aut_{C}^D(\xi) \to \aut_{C}(X)$$
is a Serre fibration with fiber $\Gamma_D(X,GL(\xi))$.
\end{lemma}

\begin{proof}
A diagram of the form
$$
\xymatrix{A \ar[r] \ar[d] & \aut_{C}^D(\xi) \ar[d]   \\ A\times I \ar[r] \ar@{-->}[ur] & \aut_{C}(X), }
$$
may be viewed as a bundle map $A\times \xi \to \xi$,
$$\xymatrix{A\times E \ar[r]^-{\widehat{f}_0} \ar[d] & E \ar[d] \\ A\times X \ar[r]^-f & X}$$
together with a homotopy $A\times X\times I \to X$ of the base map. Steenrod's ``first covering homotopy theorem'' \cite[\S11.3]{Steenrod} provides an extension of the homotopy to a bundle map $A \times \xi \times I \to \xi$, which can be taken to be stationary over $D$. This yields the desired lift $A\times I \to \aut_{C}^D(\xi)$.
The fiber $\pi^{-1}(id_X)$ is easily identified with $\Gamma_D(X,GL(\xi))$.
\end{proof}

We proceed to analyze the homotopy type of the classifying space $B\aut_C^D(\xi)$. We will use the following lemma.

\begin{lemma} \label{lemma:homotopy orbit}
Consider a commutative square of the form
$$
\xymatrix{H \ar[r]^-{x_0 \cdot}  \ar[d]^-\varphi & X \ar[d]^-f \\ G \ar[r]^{y_0 \cdot} & Y,}
$$
where $\varphi$ is a morphism of grouplike topological monoids, $X$ is a right $H$-space, $Y$ is a right $G$-space, $x_0\in X$, $y_0\in Y$, and $f$ is a map of $H$-spaces such that $f(x_0) = y_0$. If the square is a homotopy pullback, then the induced map on components,
$$B\left( X , H,* \right)_{x_0} \to B\left( Y, G,* \right)_{y_0},$$
is a weak homotopy equivalence.
\end{lemma}

\begin{proof}
By \cite[Proposition 7.9]{May}, the horizontal maps extend to homotopy fiber sequences
$$
\xymatrix{H \ar[r]^-{x_0 \cdot}  \ar[d]^-\varphi & X \ar[d]^-f \ar[r] & B(X,H,*) \ar[d] \\ G \ar[r]^{y_0 \cdot} & Y \ar[r] & B(Y,G,*).}
$$
The assumption that the left square is a homotopy pullback implies that the induced map on homotopy fibers
$\Omega_{x_0} B(X,H,*) \to \Omega_{y_0} B(Y,G,*)$ is a weak homotopy equivalence. Delooping this yields the desired weak homotopy equivalence.
\end{proof}

Let $k\colon X\to BO(n)$ be the classifying map for $\xi$ and let $\map_D(X,BO(n))$ denote the space of maps from $X$ to $BO(n)$ whose restriction to $D$ agrees with $k|_D$. Precomposition endows it with a right action of the monoid $\aut_C(X)$.

\begin{proposition} \label{prop:fibrations}
For an $n$-dimensional vector bundle $\xi$ over $X$, the classifying space $B\aut_C^D(\xi)$ is weakly homotopy equivalent to the component of
$$B\big (\map_D(X,BO(n)),\aut_C(X),*)$$
determined by the classifying map $k\colon X\to BO(n)$ for $\xi$. In particular, there is a homotopy fibration
\begin{equation} \label{eq:1}
\map_D(X,BO(n))_\xi \to B\aut_{C}^D(\xi) \to B\aut_{C}(X).
\end{equation}
\end{proposition}

\begin{proof}
Let $\gamma_n$ denote the universal bundle over $BO(n)$ and fix a classifying map $\kappa\colon \xi \to \gamma_n$,
$$
\xymatrix{E\ar[d]^-p \ar[r] & E(\gamma_n) \ar[d] \\ X\ar[r]^-k & BO(n).}
$$
We may assume it is a pullback. This implies that the square
$$
\xymatrix{\aut_C^D(\xi) \ar[r]^-{\kappa_*} \ar[d] & Bun_D(\xi,\gamma_n) \ar[d] \\
\aut_C(X) \ar[r]^-{k_*} & \map_D(X,BO(n)),}
$$
is a pullback, where $Bun_D(\xi,\gamma_n)$ denotes the space of bundle maps $\alpha\colon \xi\to \gamma_n$ such that $\alpha|_D = \kappa|_D\colon \xi|_D\to \gamma_n$. The vertical maps are fibrations by Lemma \ref{lemma:serre fibration} (and its obvious extension to $Bun_D(\xi,\gamma_n)$), so the square is a homotopy pullback. By Lemma \ref{lemma:homotopy orbit}, the induced map
$$B\big(Bun_D(\xi,\gamma_n),\aut_C^D(\xi),*\big)_{\kappa} \to B \big(\map_D(X,BO(n)),\aut_C(X),* \big)_k$$
is a weak homotopy equivalence. To finish the proof, note that $Bun_D(\xi,\gamma_n)$ is contractible, since it is the homotopy fiber of the restriction map $Bun(\xi,\gamma_n) \to Bun(\xi|_D,\gamma_n)$, where both the source and the target are contractible (for the last statement, see e.g.~\cite[Lemma 5.1]{GMTW}). This yields a weak homotopy equivalence
$$B\big(Bun_D(\xi,\gamma_n),\aut_C^D(\xi),*\big)_{\kappa} \to B\aut_C^D(\xi).$$
\end{proof}

We also need to consider stable vector bundles, so we now turn to stabilization.
\begin{proposition} \label{prop:connectivity}
The stabilization map
$$\aut_{C}^D(\xi) \to \aut_{C}^D(\xi\times \RR),\quad (f,\widehat{f}) \mapsto (f,\widehat{f}\times id_{\RR}),$$
is $(\dim \xi - \dim X -1)$-connected.
\end{proposition}

\begin{proof}
By Lemma \ref{lemma:serre fibration}, it suffices to see that
$$
\Gamma_D(X,GL(\xi)) \to \Gamma_D(X,GL(\xi\times \RR))
$$
is $(\dim \xi - \dim X -1)$-connected. But this follows from obstruction theory because $GL(\xi_x) \to GL(\xi_x\oplus \RR)$ is $(\dim \xi -1)$-connected.
\end{proof}

Denote by $\aut_{C}^{D}(\xi^S)$ the homotopy colimit of
$$\aut_{C}^D(\xi) \to \aut_{C}^D(\xi \times \RR) \to \aut_{C}^D(\xi\times \RR^2) \to \cdots.$$
Clearly, the homotopy type of $\aut_{C}^{D}(\xi^S)$ depends only on the stable equivalence class of the vector bundle $\xi$. The following description of the homotopy type of $B\aut_C^D(\xi^S)$ is easily deduced from Proposition \ref{prop:fibrations}.

\begin{proposition} \label{prop:stable fibration}
The space $B\aut_C^D(\xi^S)$ is weakly homotopy equivalent to the component of 
$$B\big (\map_D(X,BO),\aut_C(X),*)$$
determined by the classifying map $X\to BO$ of the stable vector bundle $\xi^S$. In particular, there is a homotopy fibration
\begin{equation} \label{eq:2}
\map_D(X,BO)_{\xi^S} \to B\aut_{C}^D(\xi^S) \to B\aut_{C}(X).
\end{equation}\hfill $\square$
\end{proposition}

\begin{corollary}
The space $B\aut_{C,\circ}^D(\xi^S)$ is nilpotent.
\end{corollary}

\begin{proof}
Apply Lemma \ref{lemma:nilpotent} to the fibration sequence
$$\map_D(X,BO)_{\xi^S} \to B\aut_{C,\circ}^D(\xi^S) \to B\aut_{C,\circ}(X),$$
which is obtained from \eqref{eq:2} by pulling back along $B\aut_{C,\circ}(X) \to B\aut_C(X)$.
The fiber is nilpotent since it is an infinite loop space.
\end{proof}

\begin{remark}
Since the components of an infinite loop space are all homotopy equivalent, we have
$$\map_D(X,BO)_{\xi^S} \simeq \map_D(X,BO)_{(0)} \simeq \map_*(X/D,BO)_{(0)},$$
where the subscript $(0)$ indicates the component of the trivial map,
but beware that these spaces need not be equivalent as $\aut_{C,\circ}(X)$-spaces.
\end{remark}

We now turn to the definition of the map from the block diffeomorphism group to the automorphisms of the stable tangent bundle.

We shall use the model
$$\Delta^k = \set{(x_1,\ldots,x_k)\in \RR^k}{0\leq x_1 \leq \cdots \leq x_k \leq 1}$$
for the $k$-simplex rather than the usual model in $\RR^{k+1}$. Its tangent space is then canonically trivialized, $\tau_{\Delta^k} = \Delta^k \times \RR^k$. The simplicial operators
$$d^\mu\colon \Delta^{k-1} \to \Delta^k,\quad s^\mu \colon \Delta^k \to \Delta^{k-1},$$
take the form
\begin{align*}
d^\mu(x_1,\ldots,x_{k-1}) & = \left\{ \begin{array}{lc} (0,x_1,\ldots,x_{k-1}), & \mu = 0, \\ (x_1,\ldots,x_\mu,x_\mu,\ldots,x_{k-1}), & 0<\mu<k,\\ (x_1,\ldots,x_{k-1},1), & \mu = k, \end{array} \right. \\
s^\mu(x_1,\ldots,x_k) & = (x_1,\ldots,\widehat{x}_{\mu+1},\ldots,x_k).
\end{align*}
The induced operators on the tangent spaces are
\begin{align*}
d^\mu \times d^\mu \colon & \Delta^{k-1} \times \RR^{k-1} \to \Delta^k\times \RR^k, \\
s^\mu \times s^\mu \colon & \Delta^k\times \RR^k \to \Delta^{k-1} \times \RR^{k-1}
\end{align*}
where $d^\mu\colon \RR^{k-1} \to \RR^k$ and $s^\mu\colon \RR^k\to \RR^{k-1}$ are given by the same formulas as above.

We remember that the $k$-simplices of $\tDiff_\partial(M)_\bullet$ consist of face preserving diffeomorphisms $(\varphi,\psi) \colon \Delta^k\times M \to \Delta^k \times M$ which in addition preserve a collar of each face. For a face $\theta\colon \Delta^r\to\Delta^k$, let $U_\epsilon(\theta) = \theta(\Delta^r)\times D_\epsilon(\RR^{k-r})$ be an $\epsilon$-normal neighborhood in $\RR^k$ of the face and let $\theta_M = \theta \times M$. The collar condition is that the restriction of $(\varphi,\psi)$ to $U_\epsilon(\theta)\times M$ is equal to $(\theta_M^*(\varphi),\theta_M^*(\psi))\times id_{D_\epsilon(\RR^{k-r})}$ for some small $\epsilon > 0$. Below we shall use the notation
$$
x_1 \sim 0,\quad x_\mu \sim x_{\mu+1},\quad x_k\sim 1
$$
to indicate that $x=(x_1,\ldots,x_k)$ belongs to a small normal neighborhood of $d^\mu(\Delta^{k-1})$. With this notation we have the following for $(\varphi,\psi)\in \tDiff_\partial(M)_k$ and $(x,y)\in \Delta^k\times M$.

\begin{align*}
\mbox{For $x_1\sim 0$:} && \varphi_1(x,y) & = x_1,\quad \varphi_i(x,y) = \varphi_i(0,x_2,\ldots,x_k,y),\quad i>0, \\
&&\psi (x,y) & = \psi(0,x_2,\ldots,x_k,y).\\
\\
\mbox{For $x_k\sim 1$:}&& \varphi_k(x,y) & = x_k,\quad \varphi_i(x,y) = \varphi_i(x_1,\ldots,x_{k-1},1,y), \quad i<k, \\
&&\psi(x,y) & = \psi(x_1,\ldots,x_{k-1},1,y).\\
\\
\mbox{For $x_\mu\sim x_{\mu+1}$:}&&
\varphi(x,y) & = \varphi(t_\mu(x),y) + r_\mu(x), \\
&& \psi(x,y) & = \psi(t_\mu(x),y),
\end{align*}
\begin{align*}
\mbox{where}&& t_\mu(x) & = (x_1,\ldots,x_{\mu-1}, \frac{1}{2}(x_\mu + x_{\mu+1}), \frac{1}{2}(x_\mu + x_{\mu+1}), x_{\mu+2},\ldots,x_k), \\
&& r_\mu(x) & = (0,\ldots,0,\frac{1}{2}(x_\mu-x_{\mu+1}),\frac{1}{2}(x_{\mu+1}-x_\mu),0,\ldots,0).
\end{align*}
Since $(\varphi,\psi)$ is face preserving,
$$
\varphi(t_\mu(x),y) = d^\mu(d_\mu \varphi(\overline{t}_\mu(x),y)), \quad \psi(t_\mu(x),y) = d_\mu\psi(\overline{t}_\mu(x),y),
$$
where $\overline{t}_\mu \colon \RR^k\to \RR^{k-1}$ is given by
$$
\overline{t}_\mu(x_1,\ldots,x_k) = (x_1,\ldots,x_{\mu-1}, \frac{1}{2}(x_\mu+x_{\mu+1}),x_{\mu+2},\ldots,x_k)
$$
and $(d_\mu\varphi,d_\mu\psi)$ is $\mu$-th face of $(\varphi,\psi)$.

We shall compare the Jacobians $D(\varphi,\psi)(d^\mu x,y)$ and $D(d_\mu \varphi,d_\mu \psi)(x,y)$ when $(x,y)\in \Delta^{k-1}\times M$. We need the isomorphism $\phi_\mu\colon \RR\times \RR^{k-1} \to \RR^k$ defined by
\begin{align*}
\phi_\mu(e_1) & = e_{\mu+1}-e_\mu, \\
\phi_\mu(e_i) & = e_{i-1}, \quad 1< i \leq \mu, \\
\phi_\mu(e_{\mu+1}) & = e_{\mu+1} + e_\mu, \\
\phi_\mu(e_i) & = e_i, \quad \mu +1 < i \leq k,
\end{align*}
where $e_1,\ldots,e_k$ is the standard basis for $\RR^k$, with the convention that $e_0 = 0$ and $e_{k+1} = 0$.

The linear inclusions $d^\mu\colon \RR^{k-1} \to \RR^k$ induce $k+1$ embeddings $d_\mu \colon GL(\RR^{k-1}) \to GL(\RR^k)$. For $A\in GL(\RR^{k-1})$, we have that $d_\mu(A)\in GL(\RR^k)$ is equal to the identity on the orthogonal complement of $d^\mu(\RR^{k-1})$, and there is a commutative diagram
$$
\xymatrix{\RR\times \RR^{k-1} \ar[r]^-{id \times A} \ar[d]_-{\phi_\mu} & \RR\times \RR^{k-1} \ar[d]_-{\phi_\mu} \\
\RR^k \ar[r]^-{d_\mu(A)} & \RR^k.}
$$
Differentiating the collar conditions imposed on the $k$-simplices $(\varphi,\psi)$ of $\tDiff_\partial(M)_\bullet$ and listed above we will obtain

\begin{lemma} \label{lemma:jacobian}
For a $k$-simplex $(\varphi,\psi)$ of $\tDiff_\partial(M)_\bullet$ and $(x,y)\in \Delta^{k-1} \times M$, the diagrams
$$
\xymatrix{\RR^k \times \tau_y M \ar[rrr]^-{D(\varphi,\psi)(d^\mu x,y)} &&& \RR^k \times \tau_{\psi(x,y)} M \\
\RR\times \RR^{k-1} \times \tau_y M \ar[u]_-{\phi_\mu\times id} \ar[rrr]^-{id \times D(d_\mu\varphi,d_\mu\psi)(x,y)} &&& \RR\times \RR^{k-1} \times \tau_{\psi (x,y)} M \ar[u]_-{\phi_\mu\times id}}
$$
commute for $0\leq \mu \leq k$.
\end{lemma}

In the proof below we use $D_\Delta$ and $D_M$ to denote the part of the Jacobian $D$ which differentiates with respect to $x\in \Delta^k$ and $y\in M$, respectively. With this notation, the bottom map in the diagram above consists of the following four homomorphisms
\begin{align*}
id\times D_\Delta(d_\mu \varphi)(x,y) \colon & \RR\times \RR^{k-1} \to \RR\times \RR^{k-1},\\
D_M(d_\mu \varphi)(x,y) \colon & \tau_y M \to 0\times \RR^{k-1},\\
id \times D_\Delta(d_\mu \psi)(x,y) \colon & \RR \times \RR^{k-1} \to \RR \times \tau_{\psi(x,y)} M \xrightarrow{pr_2} \tau_{\psi(x,y)} M,\\
D_M(d_\mu \psi)(x,y)\colon & \tau_y M \to \tau_{\psi(x,y)} M.
\end{align*}

\begin{proof}
We leave for the reader to check the easier cases $\mu = 0$ and $\mu = k$. So assume $0<\mu<k$. We differentiate the equation
$$\varphi(x,y) = d^\mu(d_\mu\varphi(\overline{t}_\mu(x),y)) + r_\mu(x),$$
valid for $x_\mu \sim x_{\mu+1}$, with respect to $x\in \Delta^k$ to get
$$
D_\Delta(\varphi) = d^\mu \circ D_\Delta(d_\mu\varphi)\circ D_\Delta(\overline{t}_\mu) + D_\Delta(r_\mu).
$$
Now observe that
$$
D_\Delta(\overline{t}_\mu) \circ \phi_\mu = p_1^\perp \colon \RR^k\to \RR^{k-1}
$$
is the projection onto the last $(k-1)$ coordinates, and that
$$
R_\mu = D_\Delta(r_\mu) \circ \phi_\mu \colon \RR \times \RR^{k-1} \to \RR^k$$
is the linear map that sends $v = (v_1,v_2,\ldots,v_k)$ into $\phi_\mu(v_1 e_1) = -v_1e_\mu + v_1e_{\mu+1}$. Since $\phi_\mu(0,v_2,\ldots,v_k) = d^\mu(v_2,\ldots,v_k)$, we have
$$
d^\mu \circ D_\Delta(d_\mu\varphi)\circ p_1^\perp(v) + R_\mu(v) = \phi_\mu(id_\RR \times D_\Delta(d_\mu \varphi))(v)
$$
so that
$$
D_\Delta(\varphi)\circ \phi_\mu = \phi_\mu \circ (id_\RR \times D_\Delta(d_\mu \varphi)).
$$
Differentiating with respect to $y\in M$, we have
$$D_M\varphi(d^\mu x,y) = d^\mu \circ D_M(d_\mu \varphi)(x,y) = \phi_\mu D_M(d_\mu \varphi)(x,y).$$
This proves the required commutativity for $D(\varphi)$, $D(d_\mu \varphi)$, and leaves us to check commutativity of
$$
\xymatrix{\RR^k \times \tau_y M \ar[rrr]^-{D_\Delta(\psi)(d^\mu x,y)} &&& \tau_{\psi(x,y)} M \\
\RR\times \RR^{k-1} \times \tau_y M \ar[u]_-{\phi_\mu\times id} \ar[rrr]^-{id\times D_\Delta(d_\mu \psi)(x,y)} &&& \RR\times \tau_{\psi(x,y)} M. \ar[u]_-{p_2}}
$$
But this follows upon differentiating the equation
$$
\psi(x,y) = \psi(t_\mu(x),y) = (d_\mu \psi)(\overline{t}_\mu x,y)
$$
with respect to $x\in \Delta^k$ to get
$$
D_\Delta \psi (d^\mu x,y)\circ \phi_\mu = D_\Delta (d_\mu \psi)\circ p_1^\perp.
$$
Finally, $D_M(\psi)(d^\mu x,y) = D_M(d_\mu)(x,y)$ because $\psi(d^\mu x,y) = d_\mu\psi(x,y)$.
This completes the proof.
\end{proof}

We next introduce the $\Delta$-monoid $\widetilde{S}_\bullet \aut(\xi)$ and a map of $\Delta$-monoids
$$
\alpha_\bullet\colon S_\bullet \aut(\xi) \to \widetilde{S}_\bullet \aut(\xi).
$$
\begin{definition}
The $k$-simplices of $\widetilde{S}_\bullet \aut_C(\xi)$ consist of all maps
$$
F\colon \Delta^k \to \aut_C(\RR^k\times \xi)
$$
that satisfy the conditions
$$
(\phi_\mu\times \xi)^{-1} \circ F(d^\mu x)\circ (\phi_\mu \times \xi) = id_\RR \times d_\mu F(x),\quad \mu \in \{0,\ldots,k\},
$$
where $d_\mu F \colon \Delta^{k-1} \to \aut_C(\RR^{k-1}\times \xi)$ is a $(k-1)$-simplex of $\widetilde{S}_\bullet \aut_C(\xi)$.

Variants such as $\widetilde{S}_\bullet \aut_{C,\circ}(\xi)$ or $\widetilde{S}_\bullet\aut_{C}(\xi^S)$ have the obvious meanings.
\end{definition}
The inclusion of $S_\bullet \aut(\xi)$ into $\widetilde{S}_\bullet \aut(\xi)$ is induced in degree $k$ from the map $\aut(\xi) \to \aut(\RR^k\times \xi)$ that sends $(f,\widehat{f})$ to $(f,id_{\RR^k}\times \widehat{f})$. It is a consequence of Lemma \ref{lemma:jacobian} that the Jacobian defines a $\Delta$-map
\begin{equation*} \label{eq:star}
D_\bullet\colon \tDiff_\partial(M)_\bullet \to \widetilde{S}_\bullet \aut_\partial(\tau_M).
\end{equation*}
It sends $(\varphi,\psi)\in\tDiff_\partial(M)_k$ to
$$
\xymatrix{\Delta^k\times \RR^k \times \tau_M \ar[d] \ar[r]^-{D(\varphi,\psi)} & \Delta^k \times \RR^k \times \tau_M \ar[d] \ar[r]^-{p_1^\perp} & \RR^k\times \tau_M \ar[d] \\
\Delta^k \times M \ar[r]^-{(\varphi,\psi)} & \Delta^k\times M \ar[r]^-{p_1^\perp} & M}
$$
where $p_1^\perp$ is the projection onto the last two factors. Observe that $\psi\colon \Delta^k \times M \to M$ adjoins to a map $\Delta^k\to \aut_\partial(M)$.

Let $\tDiff_{\partial,\circ}(M)_\bullet \subseteq \tDiff_\partial(M)_\bullet$ be the union of the components of those block diffeomorphism that are homotopic to the identity; it maps into $S_\bullet \aut_{\partial,\circ}(M)$. In Appendix D we prove
\begin{theorem} \label{thm:S-tilde}
The map
$$
\alpha_\bullet \colon S_\bullet \aut_{\partial,\circ}(\xi^S) \to \widetilde{S}_\bullet \aut_{\partial,\circ}(\xi^S)
$$
defines a homotopy equivalence of Kan $\Delta$-monoids.
\end{theorem}

For $\xi = \tau_M$, we have (weak) homotopy equivalences
$$
|\alpha_\bullet| \colon |S_\bullet \aut_{\partial,\circ}(\tau_M^S)| \to |\widetilde{S}_\bullet \aut_{\partial,\circ}(\tau_M^S)|,
$$
$$
ev\colon |S_\bullet \aut_{\partial,\circ}(\tau_M^S)| \to \aut_{\partial,\circ}(\tau_M^S),
$$
of topological monoids. Combined with the geometric realization of
$$D_\bullet \colon \tDiff_{\partial,\circ}(M)_\bullet \to \widetilde{S}_\bullet \aut_{\partial,\circ}(\tau_M^S),$$
one obtains a zig-zag of monoid maps from $\tDiff_{\partial,\circ}(M)$ to $\aut_{\partial,\circ}(\tau_M^S)$. Applying the classifying space construction one gets a zig-zag of maps from $B\tDiff_{\partial,\circ}(M)$ to $B\aut_{\partial,\circ}(\tau_M^S)$. The map $B|\alpha_\bullet|$ is a homotopy equivalence, since the spaces involved are of the homotopy type of CW complexes, and choosing a homotopy inverse we obtain a well-defined homotopy class of maps
$$L\colon B\tDiff_{\partial,\circ}(M) \to B\aut_{\partial,\circ}(\tau_M^S).$$
We let $D = \Omega L$,
$$D\colon \tDiff_{\partial,\circ}(M) \to \aut_{\partial,\circ}(\tau_M^S).$$
(Alternatively, one could let $L$ and $D$ denote the zig-zag maps, and remember this for the proofs below, which are all statements about homotopy groups).

If $f\colon D^k\times M \to D^k\times M$ with $\partial f = id$ represents an element of $\pi_k(\tDiff_\partial(M),id)$, then the composition
$$
D^k\times (\RR^k \times \tau_M) \xrightarrow{D(f)} D^k \times (\RR^k \times \tau_M) \xrightarrow{proj} \RR^k\times \tau_M
$$
represents its image in $\pi_k(\aut_{\partial,\circ}(\tau_M^S))$.

\subsection{The rational homotopy theory of $B\tDiff_{\partial,\circ}(M)$}
In this section we compare $B\tDiff_{\partial,\circ}(M)$ with $B\aut_{\partial,\circ}^*(\tau_M^S)$, where $\tau_M$ is the tangent bundle of $M$, and we calculate the rational homotopy type of $B\aut_{\partial,\circ}^*(\tau_M^S)$.

In the diagram of homotopy fibrations
$$
\xymatrix{\aut_{\partial,\circ}(M)/\tDiff_{\partial,\circ}(M) \ar[r] \ar[d]^-\sim & B\tDiff_{\partial,\circ}(M) \ar[r] \ar[d] & B\aut_{\partial,\circ}(M) \ar[d] \\
\aut_{\partial,J}(M)/\tDiff_\partial(M) \ar[r] & B\tDiff_\partial(M) \ar[r] & B\aut_{\partial,J}(M)}
$$
the map of the homotopy fibers is a homotopy equivalence, and there are homotopy equivalences
$$
k\colon \aut_{\partial,\circ}(M)/\tDiff_{\partial,\circ}(M) \xrightarrow{\sim} \aut_{\partial,J}(M)/\tDiff_\partial(M) \xrightarrow{\sim} \Ss_\partial^{G/O}(M)_{(1)}
$$
into the identity component of the structure space.

The map $L$ constructed in the previous section fits in a diagram
\begin{equation} \label{eq:(26)}
\xymatrix{\aut_{\partial,\circ}(M)/\tDiff_{\partial,\circ}(M) \ar[r] \ar@{-->}[d]^-\ell & B\tDiff_{\partial,\circ}(M) \ar[r] \ar[d]^-L & B\aut_{\partial,\circ}(M) \ar@{=}[d] \\ \map_{\partial M}(M,BO)_{\tau_M^S} \ar[r] & B\aut_{\partial,\circ}(\tau_M^S) \ar[r] & B\aut_{\partial,\circ}(M).}
\end{equation}
Note that $\map_{\partial M}(M,BO)_{\tau_M^S} \simeq \map_*(M/\partial M,BO)_{(0)}$.
We shall compare the induced map $\ell$ on homotopy fibers with the normal invariant map $\eta$, composed with the map induced by $j\colon G/O\to BO$,
$$
\Ss_\partial^{G/O}(M)_{(1)} \xrightarrow{\eta} \map_*(M/\partial M,G/O)_{(0)} \xrightarrow{j_*} \map_*(M/\partial M,BO)_{(0)}.
$$
The result is
\begin{lemma} \label{lemma:key}
For a simply connected smooth compact manifold $M$ of dimension at least $5$, the diagram
$$
\xymatrix{\aut_{\partial,\circ}(M)/\tDiff_{\partial,\circ}(M) \ar[d]_-\sim^-k  \ar[r]^-\ell & \map_*(M/\partial M,BO)_{(0)} \\ 
\Ss_\partial^{G/O}(M)_{(1)} \ar[r]^-\eta & \map_*(M/\partial M,G/O)_{(0)} \ar[u]^-{j_*}}
$$
commutes up to a sign upon taking homotopy groups, i.e.,
$$(j_* \circ \eta \circ k)_* = (-1) \cdot \ell_* \colon \pi_k(\aut_{\partial,\circ}(M)/\tDiff_{\partial,\circ}(M)) \to \pi_k\big(\map_*(M/\partial M,BO)_{(0)}\big).$$
\end{lemma}

\begin{proof}
The homotopy fibration
$$
\tDiff_{\partial,\circ}(M) \to \aut_{\partial,\circ}(M) \to \aut_{\partial,\circ}(M)/\tDiff_{\partial,\circ}(M)
$$
shows that we have isomorphisms
$$
\pi_k(\aut_{\partial,\circ}(M),\tDiff_{\partial,\circ}(M);id_M) \cong \pi_k(\aut_{\partial,\circ}(M)/\tDiff_{\partial,\circ}(M);*)
$$
so that
$$\pi_k(\Ss_\partial^{G/O}(M)_{(1)};id_M) \cong \pi_k(\aut_{\partial,\circ}(M),\tDiff_{\partial,\circ}(M);id_M).$$
An element
$$
[f,\varphi]\in \pi_k(\aut_{\partial,\circ}(M),\tDiff_{\partial,\circ}(M);id_M)
$$
is represented by the diagram
\begin{equation} \label{eq:source element}
\xymatrix{D^k\times M \ar[r]^-f & D^k\times M \\
S^{k-1}\times M \ar[r]^-\varphi \ar[u] & S^{k-1}\times M \ar[u]\\
{*}\times M \ar[u] \ar@{=}[r] & {*}\times M \ar[u]}
\end{equation}
with the additional conditions that the restriction of $(f,\varphi)$ to $(D^k\times \partial M,S^{k-1}\times \partial M)$ is equal to the identity, and $\varphi$ is homotopic to the identity relative to $*\times M$. The resulting element
$$\eta([f,\partial f]) \in [S^k\wedge M/\partial M,G/O]_*$$
and its companion
$$j_*\eta([f,\partial f]) \in [S^k\wedge M/\partial M,BO]_*$$
were described at the end of \S\ref{sec:surgery}.

In our case,
$$f\colon D^k \times M \to D^k \times M, \quad \partial f = \varphi \cup id \colon \partial(D^k\times M) \to \partial(D^k\times M),$$
and $(f,\varphi)$ restricts to the identity on $(C,\partial C)$ where
$$C=D^k \times \partial M \cup *\times M,\quad \partial C = S^{k-1}\times \partial M\cup * \times M.$$
Let $(g,\psi)$ be homotopy inverse to $(f,\varphi)$ relative to $(C,\partial C)$, and set
$$
\xi = g^*(D^k\times \nu_M) \oplus D^k\times \RR^k \times \tau_M,\quad \gamma = \psi^*(S^{k-1}\times \nu_M)\oplus S^{k-1}\times \RR^k\times \tau_M.
$$
Here $\nu_M$ is the $K$-dimensional normal bundle of an embedding
$$(M,\partial M) \subset (\RR^{K+n},\RR^{K+n-1})$$
with $K\gg 0$. The restriction of $\xi$ to $S^{k-1}\times M$ is equal to $\gamma$ and
$$\xi|_C = C\times \RR^{K+n},\quad \gamma|_{\partial C} = \partial C \times \RR^{K+n}$$
upon using that $\nu_M \oplus \tau_M = M\times \RR^{K+n}$.

Since $(D^k\times M)/C = D^k\wedge M/\partial M$ and $S^{k-1} \times M / \partial C = S^{k-1} \wedge M/\partial M$ we have bundles
$$\overline{\xi} \searrow D^k \wedge M/\partial M,\quad \overline{\gamma} \searrow S^{k-1}\wedge M/\partial M.$$
The framing $\partial \theta$ of $\gamma$ was defined in \eqref{eq:framing} at the end of \S\ref{sec:surgery}. The differential of $\psi$ induces a bundle isomorphism
$$D\psi \colon S^{k-1} \times (\RR^k \times \tau_M) \to \psi^*(S^{k-1} \times \RR^k \times \tau_M)$$
and $\partial \theta$ is the framing of $\gamma$ associated to the isomorphism
$$\psi^*(S^{k-1}\times \nu_M)\oplus \psi^*(S^{k-1} \times \RR^k \times \tau_M) \xrightarrow{id\oplus D\psi} \epsilon_M^{K+n+k},$$
where we have used the identification $\nu_M \oplus \tau_M  = \epsilon_M^{K+n}$. Since $\psi = \varphi^{-1}$ is the identity over $\partial C$, the framing $\partial \theta$ induces a framing $\overline{\partial \theta}$ of $\overline{\gamma}$ and
$$[\overline{\xi}/\overline{\partial \theta}] = j_* \eta (f,\partial f) \in [S^k \wedge M/\partial M,BO].$$

The bundle $\overline{\xi}$ is canonically trivialized as a bundle over the cone $D^k\wedge M/\partial M$ and induces a trivialization
$$h\colon S^{k-1}\wedge M/\partial M \times \RR^{K+n+k} \xrightarrow{\cong} \overline{\gamma}.$$
The composition $\overline{\partial \theta} \circ h$ adjoins to a map
$$S^{k-1}\wedge M/\partial M \to GL(\RR^{K+n+k})$$
that represents $\partial_*^{-1}([\overline{\xi},\overline{\partial \theta}])$. Indeed, this is the well-known relation between maps from a space $X$ into $GL_\infty(\RR)$ and bundles over the suspension of $X$.

We next calculate the composition $(\Omega \ell)_*\circ \partial_*$. By construction, the analog of \eqref{eq:source element} is valid with $(f,\varphi)$ replaced by $(g,\psi)$, so $\psi\simeq id$ (rel.~$\partial C$). This yields the isomorphism
$$h'\colon \psi^*(S^{k-1}\times \RR^k\times \tau_M)\to S^{k-1}\times \RR^k\times \tau_M$$
and leads to the element
$$\sigma = h' \circ D\psi \in \Gamma_{\partial C}(S^{k-1}\times M,S^{k-1}\times \RR^k \times \tau_M).$$
The map
$$\Gamma_{\partial C}(S^{k-1}\times M,S^{k-1}\times \RR^k \times \tau_M) \to \map_*(S^{k-1}\wedge M/\partial M,GL_\infty(\RR))$$
sends $\sigma$ into (the adjoint of) the bundle isomorphism
$$\widehat{\sigma}\colon S^{k-1}\wedge M/\partial M \times \RR^{K+n+k} \to S^{k-1} \wedge M/\partial M \times \RR^{K+n+k}$$
induced from $\sigma$ by adding the identity of $S^{k-1}\times \nu_M$. It follows that the above $\overline{\partial \theta} \circ h$ is homotopic to $\widehat{\sigma}$.

Finally, the map $(\Omega \ell)_* \circ \partial_*$ is represented not by $\widehat{\sigma}$ but by $\widehat{(\sigma^{-1})} = h'\circ D\varphi$, the inverse of $\widehat{\sigma}$. Since inversion $GL_\infty(\RR) \to GL_\infty(\RR)$ induces multiplication by $(-1)$ on the mapping space, this completes the proof.
\end{proof}

\begin{proposition} \label{prop:fiber proposition}
For a simply connected smooth compact manifold $M$ with $\partial M = S^{n-1}$ ($n\geq 5$), the composite map
$$\aut_{\partial,\circ}(M)/\tDiff_{\partial,\circ}(M) \xrightarrow{\ell} \map_*(M/\partial M,BO)_{(0)} \xrightarrow{q^*} \map_*(M,BO)_{(0)}$$
is a rational homotopy equivalence.
\end{proposition}

\begin{proof}
Consider the diagram
$$
\xymatrix{\aut_{\partial,\circ}(M)/\tDiff_{\partial,\circ}(M) \ar[d]_-\sim^-k  \ar[r]^-\ell & \map_*(M/\partial M,BO)_{(0)} \ar[r]^-{q^*} & \map_*(M,BO)_{(0)} \\ 
\Ss_\partial^{G/O}(M)_{(1)} \ar[r]^-\eta & \map_*(M/\partial M,G/O)_{(0)} \ar[u]^-{j_*} \ar[r]^-{q^*} & \map_*(M,G/O)_{(0)} \ar[u]^-{j_*}}
$$
By Lemma \ref{lemma:key}, the diagram anti-commutes after taking homotopy groups. By Corollary \ref{cor:ss} the map $j_*\circ q^* \circ \eta$ is a rational homotopy equivalence. It follows that $q^*\circ \ell$ is a rational homotopy equivalence as well.
\end{proof}

Recall that for a vector bundle $\xi$ over a space $X$ with subspaces $D\subseteq C\subseteq X$, we let $\aut_{C,\circ}^D(\xi) \subseteq \aut(\xi)$ denote the submonoid of those $(f,\widehat{f})$ for which $f\in \aut_{C,\circ}(X)$ and $\widehat{f}$ restricts to the identity map on the fibers over points in $D$. The map $B\tDiff_{\partial,\circ}(M) \to B\aut_{\partial,\circ}(\tau_M^S)$, followed by the map on classifying spaces induced by the inclusion of $\aut_{\partial,\circ}(\tau_M^S) = \aut_{\partial,\circ}^\partial(\tau_M^S)$ into $\aut_{\partial,\circ}^{*}(\tau_M^S)$, induces a map of fibration sequences, similar to \eqref{eq:(26)},

\begin{equation} \label{eq:b}
\xymatrix{
\aut_{\partial,\circ}(M)/\tDiff_{\partial,\circ}(M) \ar[d]^-{q^*\circ \ell} \ar[r] & B\tDiff_{\partial,\circ}(M) \ar[d] \ar[r] & B\aut_{\partial,\circ}(M) \ar@{=}[d] \\
\map_*(M,BO)_{\tau_M^S} \ar[r] & B\aut_{\partial,\circ}^*(\tau_M^S) \ar[r] & B\aut_{\partial,\circ}(M).}
\end{equation}

The induced map on homotopy fibers may be identified with the map $q^*\circ \ell$, which is a rational homotopy equivalence by the previous proposition. The following is a consequence.

\begin{corollary} \label{cor:bdiff model}
For a simply connected smooth compact manifold $M^n$ with $\partial M = S^{n-1}$  $(n\geq 5)$ and tangent bundle $\tau_M$, the differential induces a rational homotopy equivalence
$$
B\tDiff_{\partial,\circ}(M) \rightarrow B\aut_{\partial,\circ}^*(\tau_M^S).
$$
\hfill $\square$
\end{corollary}

\begin{remark}
Let $\nu_M$ be the normal bundle of an embedding of $(M,\partial M)$ into $(\RR^{n+K},\RR^{n+K-1})$ for $K$ large. Then we have that $B\aut_{\partial,\circ}^*(\nu_M^S)$ is weakly equivalent to $B\aut_{\partial,\circ}^*(\tau_M^S)$. To see this, one can use Proposition \ref{prop:stable fibration}, which shows that the two spaces are different, but homotopy equivalent, components of the space $B\big( \map_*(M,BO),\aut_{\partial,\circ}(M),*\big)$. Indeed, since $BO$ is an infinite loop space, inversion defines an $\aut_{\partial,\circ}(M)$-equivariant homotopy automorphism of $\map_*(M,BO)$, which maps the component of a bundle $\xi$ to the component of its complementary bundle. The induced map of $B\big(\map_*(M,BO),\aut_{\partial,\circ}(M),*\big)$ maps the component of $\tau_M$ to the component of $\nu_M$.
\end{remark}

Thus, the analysis of the rational homotopy type of $B\tDiff_{\partial,\circ}(M)$ reduces to homotopy theory of stable vector bundles. We proceed to analyze the rational homotopy type of $B\aut_{\partial,\circ}^*(\tau_M^S)$, following \cite{Berglund2}. Applying Proposition \ref{prop:stable fibration} to $(X,C,D) = (M,\partial M,*)$, $\xi = \tau_M$, we obtain
\begin{corollary} \label{cor:bar model}
There is a weak homotopy equivalence
$$B\aut_{\partial,\circ}^*(\tau_M^S) \simeq B\big( \map_*(M,BO)_{\tau_M^S},\aut_{\partial,\circ}(M),*\big).$$
\end{corollary}
From Corollary \ref{cor:bar model}, the following rational model for $B\aut_{\partial,\circ}^*(\tau_M^S)$ can be derived using the methods of \cite{Berglund2}: The classifying space for stable vector bundles $BO$ has a simple dg Lie algebra model, namely the abelian dg Lie algebra with zero differential
$$P = \pi_*(\Omega B O )\tensor \QQ.$$
It admits a basis $q_1,q_2,\ldots$, where $q_i\in \pi_{4i-1}(\Omega BO)\tensor \QQ = \pi_{4i}(BO)\tensor \QQ$ is dual to the universal Pontryagin class $p_i \in \HH^{4i}(BO;\QQ)$.
Recall from the previous section that the minimal Quillen model of $M$ has the form
$$\LL_M = \big( \LL(V),\delta \big),$$
where $V = s^{-1}\widetilde{\HH}_*(M;\QQ)$.
The action of the dg Lie algebra
$$\Der_\omega^+ \LL_M = \big(\Der_\omega^+ \LL(V),[\delta,-]\big)$$
on $\LL_M$ induces an action on indecomposables
$$\LL_M/[\LL_M,\LL_M] = V = s^{-1} \widetilde{\HH}_*(M;\QQ),$$
which dualizes to give an action on reduced cohomology $\widetilde{\HH}^*(M;\QQ)$.
The Pontryagin classes of the tangent bundle $\tau_M$,
$$p_i(\tau_M)\in \HH^{4i}(M;\QQ),$$
may be assembled to a distinguished element $\tau$ of degree $-1$ in the tensor product $\widetilde{\HH^*}(M;\QQ)\tensor P$;
$$\tau = \sum_i p_i(\tau_M) \tensor q_i.$$

\begin{theorem} \label{thm:block diff model}
Let $M$ be a simply connected smooth compact manifold with boundary $\partial M = S^{n-1}$ and tangent bundle $\tau_M$. Let $\big(\LL(V),\delta\big)$ be a minimal Quillen model for $M$ and let $\omega\in \LL(V)$ represent the inclusion of the boundary.

The classifying space $B\aut_{\partial,\circ}^*(\tau_M^S)$ is rationally homotopy equivalent to the geometric realization of the dg Lie algebra
$$
\big(\widetilde{\HH}^*(M;\QQ)\tensor P\big)_{\geq 0} \rtimes_{\tau} \Der_\omega^+ \LL(V).
$$
Explicitly, the differential is given by
\begin{equation} \label{eq:bdiff differential}
\partial^\tau(x,\theta) = \big(\tau\ldotp \theta, [\delta,\theta] \big),
\end{equation}
where $\tau\ldotp \theta$ denotes the action of $\theta\in \Der_\omega^+ \LL(V)$ on $\tau$ described above.
\end{theorem}

\begin{proof}
Using that $B\aut_{\partial,\circ}^*(\tau_M^S)$ is weakly homotopy equivalent to the bar construction $B\big(\map_*(M,BO)_{(0)},\aut_{\partial,\circ}(M),*\big)$, the result follows from \cite{Berglund2}.
\end{proof}

\begin{corollary}
If $M$ is formal with trivial multiplication on the reduced cohomology ring and if the rational Pontryagin classes of $\tau_M$ are trivial, then $B\aut_{\partial,\circ}^*(\tau_M^S)$ is coformal, with rational homotopy Lie algebra isomorphic to
$$
\big(\widetilde{\HH}^*(M;\QQ)\tensor P\big)_{\geq 0} \rtimes \Der_\omega^+ \LL(V).
$$
\end{corollary}

\begin{proof}
Vanishing of the rational Pontryagin classes is equivalent to $\tau = 0$.
Formality of $M$ together with triviality of the multiplication on $\widetilde{\HH}^*(M;\QQ)$ is equivalent to $\delta = 0$.
This implies that the differential \eqref{eq:bdiff differential} is zero, which in particular means that $B\aut_{\partial,\circ}^*(\tau_M^S)$ is coformal.
\end{proof}

\begin{corollary} \label{cor:triviality}
If $\widetilde{\HH}_*(M;\QQ)$ is concentrated in a single degree, then the fibration
$$
\map_*(M,BO)_{\tau_M^S} \to B\aut_{\partial,\circ}^*(\tau_M^S) \to B\aut_{\partial,\circ}(M)
$$
is rationally trivial. Consequently, there is a rational homotopy equivalence
$$B\aut_{\partial,\circ}^*(\tau_M^S) \simeq_\QQ \map_*(M,BO)_{(0)} \times B\aut_{\partial,\circ}(M).$$
Moreover, $B\aut_{\partial,\circ}^*(\tau_M^S)$ is coformal with rational homotopy Lie algebra isomorphic to
$$
\big(\widetilde{\HH}^*(M;\QQ)\tensor P\big)_{\geq 0} \times \Der_\omega^+ \LL(V).
$$
\end{corollary}

\begin{proof}
If $\widetilde{\HH}_*(M;\QQ)$ is concentrated in a single degree, then $\delta = 0$ is forced by degree reasons. Moreover, the action of $\Der_\omega^+ \LL(V)$ (which is concentrated in positive degrees) on $\widetilde{\HH}_*(M;\QQ)$ is trivial, also for degree reasons. This implies that the differential \eqref{eq:bdiff differential} is zero, and moreover that the semi-direct product is a product.
\end{proof}

Clearly, the previous corollary applies to $(d-1)$-connected manifolds of dimension $2d$ and in particular to the generalized surfaces
$$
M_{g,1} = \#^g (S^d\times S^d)\setminus \interior(D^{2d}).
$$
In the next section, we will focus on this class of manifolds.

\section{Automorphisms of highly connected manifolds} \label{sec:auto}
This section is devoted to the proof of the following general result on the rational homotopy type of classifying spaces of homotopy automorphisms of highly connected manifolds.

\begin{theorem} \label{thm:rht}
Let $M$ be a closed $(d-1)$-connected $2d$-dimensional manifold and let $N$ denote the result of removing an open $2d$-disk from $M$. Let $X$ denote either of the classifying spaces
$$B\aut(M),\quad B\aut_*(M),\quad \textrm{or} \quad B\aut_\partial(N),$$
and $\cX$ the simply connected cover of $X$. Let $H = \HH_d(M;\ZZ)$ with intersection form $\mu$ and quadratic refinement $Jq$ (see \S\ref{sec:wall}). If $d\geq 3$ and $\rank H>2$, then
\begin{enumerate}
\item The fundamental group $\pi_1(X)$ maps surjectively, with finite kernel, onto the automorphism group $\Aut(H,\mu,Jq)$.
\item Quillen's dg Lie algebra $\lambda(\cX)$ is formal.
\item The rational homotopy Lie algebra $\pi_*^\QQ(\cX) = \pi_{*+1}(\cX)\tensor \QQ$, with the Whitehead product, is isomorphic to
$$\OutDer^+ \pi_*^\QQ(M),\quad \Der^+ \pi_*^\QQ(M),\quad \textrm{or} \quad \Der_\omega^+ \pi_*^\QQ(N);$$
the graded Lie algebra of positive degree outer derivations, derivations, or derivations annihilating $\omega$, respectively. The graded Lie algebra $\pi_*^\QQ(N)$ is free on $\rank H$ generators of degree $d-1$, and $\pi_*^\QQ(M)$ is isomorphic to the quotient graded Lie algebra $\pi_*^\QQ(N)/(\omega)$, where $\omega\in \pi_{2d-1}(N)$ is the homotopy class of the attaching map for the top cell in $M$.
\end{enumerate}
\end{theorem}

The result for boundary preserving automorphisms will be a key ingredient in later sections. The results for base-point preserving and free automorphisms will not play a further role in this paper, but they are of independent interest and are included for completeness. We remark that Theorem \ref{thm:rht} may be viewed as an `infinitesimal' version of the Dehn-Nielsen-Baer theorem (see, e.g., \cite[Chapter 8]{FM}).

\subsection{Wall's classification of highly connected manifolds} \label{sec:wall}
Let $M$ be a closed oriented $(d-1)$-connected smooth manifold of dimension $2d$, where $d\geq 3$. The intersection form
$$\mu \colon \HH_d(M) \tensor \HH_d(M) \rightarrow \ZZ,\quad \mu(x,y) = \langle x,y \rangle,$$
endows $\HH_d(M)$ with the structure of an $(-1)^d$-symmetric inner product space over $\ZZ$. Every homology class $x\in \HH_d(M)$ may be represented as the fundamental class of some embedded sphere $S^d\subset M$ by \cite{Haefliger}. The normal bundle of the embedding $S^d \subset M$ is classified by a map $\nu\colon S^d\rightarrow BSO(d)$ and determines a homotopy class $[\nu] \in \pi_{d-1}(SO(d))$. The function
$$q\colon \HH_d(M)\rightarrow \pi_{d-1}(SO(d)),\quad x\mapsto [\nu],$$
is well-defined and satisfies the following equations:
\begin{align} \label{eq:n-space1}
q(x+y) & = q(x) + q(y) + \langle x,y\rangle \partial(\iota_d), \\
HJ q(x) & = \langle x,x\rangle. \label{eq:n-space2}
\end{align}
Here $\partial(\iota_d)\in \pi_{d-1}(SO(d))$ denotes the image of the class of the identity map of $S^d$ under the boundary map of the long exact homotopy sequence associated to the fibration $SO(d)\rightarrow SO(d+1) \rightarrow S^d$. In the second row, $J\colon \pi_{d-1}(SO(d))\rightarrow \pi_{2d-1}(S^d)$ is the $J$-homomorphism and $H\colon \pi_{2d-1}(S^d) \rightarrow \ZZ$ the Hopf invariant. We refer to Wall's work \cite{Wall0, Wall1,Wall2,Wall3} for more details.

By a \emph{(geometric) quadratic module} we will mean the data $(H,\mu,q)$ of an abelian group $H$ together with a $(-1)^d$-symmetric non-degenerate bilinear form
$$\mu \colon H\tensor H \rightarrow \ZZ,\quad \mu(x,y) = \langle x,y\rangle,$$
and a function
$$q\colon H \rightarrow \pi_{d-1}(SO(d)),$$
such that the equations \eqref{eq:n-space1} and \eqref{eq:n-space2} are satisfied. A morphism of quadratic modules is a homomorphism that preserves $\mu$ and $q$. Let $Q(M)$ denote the quadratic module $(\HH_d(M),\mu,q)$ associated to a highly connected manifold $M$.

If the normal bundles $\nu$ of the embedded spheres $S^d \subset M$ are stably trivial, i.e., if the tangent bundle $\tau_M$ restricts to the trivial bundle on the embedded spheres, then the quadratic function $q$ maps into the subgroup of $\pi_{d-1}(SO(d))$ generated by $\partial(\iota_d)$. The $J$-homomorphism maps this subgroup isomorphically onto the subgroup of $\pi_{2d-1}(S^d)$ generated by the Whitehead product $[\iota_d,\iota_d]$. If $d$ is even, then $\partial(\iota_d)$ has infinite order, and in this case the quadratic function is determined by the self-intersection by \eqref{eq:n-space2}. If $d$ is odd and $\ne 1,3,7$, then $\partial(\iota_d)$ is a non-zero element of order $2$. In the cases $d=1,3,7$, we have $\partial(\iota_d) = 0$.

Let $N$ denote the manifold obtained by removing an open $2d$-disk from $M$. Then $N$ is homotopy equivalent to a wedge of spheres $\vee^n S^d$, where $n$ is the rank of $H=\HH_d(M)$, and we may identify its boundary $\partial N$ with $S^{2d-1}$. The homotopy type of $M$ is determined by the homotopy class $\omega\in \pi_{2d-1}(N)$ of the inclusion $S^{2d-1} = \partial N\rightarrow N$, which may be expressed in terms of the associated quadratic module as follows. Let $\alpha_i\colon S^d\rightarrow N$, for $i=1,\ldots,n$, represent a basis for $\pi_d(N)$ and let $e_1,\ldots,e_n$ be the corresponding basis for $H$. Then we have the equality
\begin{equation} \label{eq:attaching map}
\omega = \sum_{i<j} \langle e_i,e_j \rangle [\alpha_i,\alpha_j] + \sum_i \alpha_i\circ Jq(e_i)
\end{equation}
of elements in the homotopy group $\pi_{2d-1}(N)$, see \cite{Wall0}. We note furthermore that the rational homotopy groups $\pi_*^\QQ(N) = \pi_{*+1}(N)\tensor \QQ$, with the Whitehead product, is a free graded Lie algebra on the classes $\alpha_1,\ldots,\alpha_n$.

\begin{remark} \label{rmk:Jq}
The function $Jq\colon \HH_d(M)\rightarrow \pi_{2d-1}(S^d)$ may be defined in purely homotopy theoretical terms. In fact, given a homology class $x\in \HH_d(M)$ with Poincar\'e dual cohomology class $\xi\in \HH^d(M)$, one can check that $Jq(x)\in \pi_{2d-1}(S^d)$ is the obstruction for $\xi$ to be induced by a map $M\rightarrow S^d$. This obstruction has been studied by Kervaire and Milnor in \cite[\S8]{KM}, where it is denoted $\psi(\xi)$. In particular, the function $Jq$ is defined for all $d$. Note however that the function $q$ is defined only for $d\geq 3$, because one needs to be able to represent homology classes by embedded spheres (cf.~\cite{Haefliger}).
\end{remark}

\subsection{Mapping class groups} \label{sec:mcg}
The mapping class groups of highly connected manifolds may be described in terms of the associated quadratic modules, up to extensions. We will recall the calculation for the homotopy and block diffeomorphism mapping class groups of $N$ relative to its boundary, where as above $M$ is a closed $(d-1)$-connected $2d$-dimensional manifold and $N = M \setminus \interior D^{2d}$.

\begin{proposition}[See \cite{BM,Kreck}] \label{prop:mcg}
Let $d\geq 3$. There is a commutative diagram with exact rows
$$
\xymatrix{
0\ar[r] & \widetilde{K} \ar[d] \ar[r] & \pi_0 \tDiff_\partial(N) \ar[d] \ar[r] & \Aut(H,\mu,q) \ar[d] \ar[r] & 0 \\
0\ar[r] & K \ar[r] & \pi_0 \aut_\partial(N) \ar[r] & \Aut(H,\mu,Jq) \ar[r] & 0.
}
$$
The group $K$ is finite. The group $\widetilde{K}$ is finite except when $d\equiv 3 \mbox{(mod $4$)}$, in which case there is an exact sequence
$$\xymatrix{0\ar[r] & \Theta_{2d+1} \ar[r] & \widetilde{K} \ar[r] & H \ar[r] & 0,}$$
where $\Theta_{2d+1}$ denotes the group of $(2d+1)$-dimensional homotopy spheres.
\end{proposition}

This description of $\pi_0\aut_\partial(N)$ is valid also for $d=2$, see Remark \ref{rmk:Jq}.

\begin{remark} \label{rem:arithmetic}
We note that $\Aut(H,\mu,Jq)$ is an arithmetic subgroup (in the sense of \cite{Serre}) of the algebraic group over $\QQ$ of automorphisms of the inner product space $(H^\QQ,\mu)$. The exact sequence of Proposition \ref{prop:mcg} shows that $\pi_0\aut_\partial(N)$ maps onto $\Aut(H,\mu,Jq)$ with finite kernel. This is related to the general result, due to Sullivan \cite[Theorem 10.3]{Sullivan} and Wilkerson \cite[Theorem B(2)]{Wilkerson}, that the homotopy mapping class group of a simply connected finite CW-complex is commensurable with an arithmetic group.
\end{remark}

\begin{example} \label{ex:hyperbolic}
For the manifold $S^d\times S^d$ the normal bundles of the embeddings $S^d\times * \subset S^d \times S^d$ and $*\times S^d\subset S^d\times S^d$ are trivial. Thus, if we let $e$ and $f$ be the classes in $\HH_d(S^d\times S^d)$ represented by these embeddings, then the quadratic module associated to $S^d\times S^d$ is given by $(\ZZ e \oplus \ZZ f,\mu,q)$, where
$$\langle e,e\rangle = 0,\quad \langle e,f\rangle = 1, \quad \langle f,f\rangle = 0,$$
$$q(a e +  b f) = a b \partial(\iota_d).$$
Connected sums of oriented manifolds correspond to orthogonal sums of quadratic modules; for $M$ and $N$ two highly connected manifolds, there is a natural isomorphism of quadratic modules $Q(M\# N) \cong Q(M)\oplus Q(N)$. It follows that the quadratic module associated to the manifold $M_g = \#^g S^d \times S^d$ is the hyperbolic module $(H_g,\mu,q)$: there is a basis $e_1,\ldots,e_g,f_1,\ldots,f_g$ for $H_g$ such that
$$\langle e_i,e_j\rangle = 0,\quad \langle e_i,f_j\rangle = \delta_{ij}, \quad \langle f_i,f_j\rangle = 0,$$
$$q(a_1e_1 + \cdots +a_g e_g + b_1f_1 + \cdots b_g f_g) = \sum_{i=1}^g a_i b_i \partial(\iota_d).$$

It follows that $\Aut(H_g,\mu,q) = \Aut(H_g,\mu,Jq)$ for the hyperbolic module. According to Proposition \ref{prop:mcg}, both groups $\pi_0\tDiff_\partial(M_{g,1})$ and $\pi_0\aut_\partial(M_{g,1})$ map onto $\Aut(H_g,\mu,q)$\footnote{We warn the reader that there is an erroneous claim in \cite{BM} (bottom of p.24 and onwards) that $\pi_0\aut_\partial(M_{g,1})$ maps onto $\Aut(H_g,\mu)$. The error comes from the incorrect inference ``if $q(e_i) = 0$ for all $i$, then $q=0$''. The mistake is harmless; replacing $\Aut(H_g,\mu)$ by $\Aut(H_g,\mu,q)$ in \cite{BM} (where $\mu$ is denoted $q$ and $q$ is denoted $\beta$) the arguments go through.}.
The automorphism group
$$\Gamma_g := \Aut(H_g,\mu,q)$$
admits the following concrete description. If $d$ is even, then $\Gamma_g$ is isomorphic to the automorphism group $\Aut(H_g,\mu)$, i.e., to the orthogonal group $O_{g,g}(\ZZ)$. If $d=1,3,7$, then $\Gamma_g$ is isomorphic to the symplectic group $\Sp_{2g}(\ZZ)$. If $d\ne 1,3,7$ is odd, then $\Gamma_g$ is isomorphic to the subgroup of $\Sp_{2g}(\ZZ)$ consisting of those symplectic matrices
$$\left( \begin{array}{cc} \alpha & \beta \\ \gamma & \delta \end{array} \right)$$
for which the diagonal entries of the $g\times g$-matrices $\gamma^t \alpha$ and $\delta^t \beta$ are even. For this last description, see e.g., \cite[\S3]{Bak}. In the notation of \cite{Bak}, $\Gamma_g$ is isomorphic to the automorphism group of the hyperbolic module in the category $Q^\lambda(A,\Lambda)$, where $A$ is the ring $\ZZ$ with trivial involution, $\lambda = (-1)^d$, and $\Lambda = 0$ if $d$ is even, $\Lambda = \ZZ$ if $d=1,3,7$ and $\Lambda = 2\ZZ$ if $d\ne 1,3,7$ is odd.
\end{example}

In what follows we will describe the rational homotopy type of the simply connected cover of $B\aut_\partial(N)$, viewed as a representation of the mapping class group.

\subsection{Equivariant rational homotopy type}
Let $M$ be a $(d-1)$-connected, $2d$-dimensional manifold and let $N$ be the manifold obtained by removing an open $2d$-disk from $M$. Let $(H,\mu,q)$ be the associated quadratic module and let $H^\QQ = H\tensor \QQ$. We may identify $\pi_d(N)$ with $H$, and the homotopy Lie algebra $\pi_*^\QQ(N) = \pi_{*+1}(N)\tensor \QQ$ with the free graded Lie algebra $\LL(H^\QQ[d-1])$, where the generators are put in degree $d-1$. Note also that $V=s^{-1}\widetilde{\HH}_*(N;\QQ)$ is the same as $H^\QQ[d-1]$.
It follows immediately from Corollary \ref{cor:coformal} that $B\aut_{\partial,\circ}(N)$ is coformal with rational homotopy Lie algebra
$$\pi_*(\aut_\partial(N))\tensor \QQ \cong \Der_\omega^+ \LL(H^\QQ[d-1]).$$
We next wish to identify the action of the mapping class group algebraically.

\begin{proposition} \label{prop:equivariant}
There is a $\pi_0\aut_\partial(N)$-equivariant isomorphism of graded Lie algebras
$$\pi_*(\aut_\partial(N))\tensor \QQ \cong \Der_\omega^+ \LL(H^\QQ[d-1]),$$
where the action on the right hand side is induced from the standard action of $\Aut(H,\mu,Jq)$ on $H$.
\end{proposition}

\begin{proof}
According to \eqref{eq:aut der lie}, there is an isomorphism of graded Lie algebras
\begin{equation} \label{eq:samelson}
\pi_*(\aut_*(N),id_N)\tensor \QQ \cong \Der^+ \LL_N.
\end{equation}

The monoid $\aut_\partial(N)$ is the fiber over the inclusion map $i\colon \partial N\rightarrow N$ of the fibration
$$i^*\colon \aut_*(N)\rightarrow \map_*(\partial N,N).$$
By naturality of \eqref{eq:aut der}, the map in rational homotopy induced by $i^*$ may be identified with the map
\begin{equation} \label{eq:model}
\varphi^*\colon \Der^+ \LL_N \rightarrow \Der_\varphi^+(\LL_{\partial N},\LL_N).
\end{equation}
As in the proof of Theorem \ref{thm:aut model}, the map \eqref{eq:model} is surjective. Hence, the rational homotopy exact sequence of the fibration
$$\aut_\partial(N) \rightarrow \aut_*(N) \rightarrow \map_*(\partial N,N)$$
splits into short exact sequences, and the rational homotopy groups of $\aut_\partial(N)$ may be identified with the kernel of \eqref{eq:model}. Thus, for $*>0$,
\begin{equation} \label{eq:identification}
\pi_*(\aut_\partial(N))\tensor \QQ \cong \Der_\omega^+ \LL(H^\QQ[d-1]).
\end{equation}
An argument is needed to show that this isomorphism commutes with Lie brackets. Since $\aut_\partial(N)\rightarrow \aut_*(N)$ is a map of monoids the map $\pi_*(\aut_\partial(N))\tensor \QQ \rightarrow \pi_*(\aut_*(N))\tensor \QQ$ commutes with Samelson products, and since the map is injective, we may calculate Lie brackets in $\pi_*(\aut_\partial(N))\tensor \QQ$ by passing to $\pi_*(\aut_*(N))\tensor \QQ$, where they are calculated in terms of commutators of derivations \eqref{eq:samelson}, so it does follow that \eqref{eq:identification} preserves Lie brackets.

By the same token, the action of $\pi_0(\aut_\partial(N))$ on $\pi_*(\aut_\partial(N))\tensor \QQ$ may be calculated by passing to the action of $\pi_0(\aut_*(N))$ on $\pi_*(\aut_*(N))\tensor \QQ$. The latter action may in turn be identified by exploiting the naturality of the isomorphism \eqref{eq:aut der}. Indeed, if $f\colon N\rightarrow N$ is a based homotopy self-equivalence, then a Lie model for $f$ is simply given by the isomorphism $\varphi_f\colon \LL_N\rightarrow \LL_N$ that is induced by $f$ in rational homotopy. By our previous considerations, cf.~\S\ref{sec:derivations}, the action of the class $[f]\in\pi_0(\aut_*(N))$ on $\pi_*(\aut_*(N))\tensor \QQ$ is induced by the self-map of $\aut_*(N)$ that sends $g$ to $fgf^{-1}$, where $f^{-1}$ is a choice of homotopy inverse of $f$. From the naturality of the isomorphism \eqref{eq:aut der} it follows that the action of $[f]$ on $\Der^+(\LL_N)$ is given by $\theta \mapsto \varphi_f \circ \theta \circ \varphi_f^{-1}$.
\end{proof}

\begin{corollary} \label{thm:finite type}
The rational cohomology groups of $B\aut_\partial(N)$ are finite dimensional in each degree.
\end{corollary}

\begin{proof}
Let $X = B\aut_\partial(N)$. The graded Lie algebra $\Der_\omega^+ \LL(H^\QQ[d-1])$ is finite dimensional in each degree, so the rational homotopy groups of $\cX$ are finite dimensional in each degree. Hence, the same is true of the rational cohomology groups of $\cX$. As pointed out in Remark \ref{rem:arithmetic}, the group $\pi_1(X)$ maps onto an arithmetic group with finite kernel.
A spectral sequence argument together with a certain finiteness property of the cohomology of arithmetic groups (see Theorem \ref{thm:arithmetic groups}) then shows that the cohomology $\HH^p(\pi_1(X);V)$, with coefficients in any finite dimensional representation $V$, is finite dimensional. Thus, in the universal cover spectral sequence,
$$E_2^{p,q} = \HH^p(\pi_1(X);\HH^q(\cX;\QQ))\Rightarrow \HH^{p+q}(X;\QQ),$$
each term $E_2^{p,q}$ is finite dimensional. It follows that $\HH^k(X;\QQ)$ is finite dimensional for every $k$.
\end{proof}

\subsection{Free and based homotopy automorphisms}
We now turn to the rational homotopy theory of the classifying spaces of the monoids $\aut(M)$ and $\aut_*(M)$ of free and base-point preserving homotopy automorphisms, respectively, for highly connected closed manifolds $M$.

Let $M$ be a closed $(d-1)$-connected, $2d$-dimensional manifold, and let $N$ be the the result of removing an open $2d$-disk from $M$. Recall from \S\ref{sec:wall} the definition of the quadratic module $(H,\mu,q)$ and the homotopy class $\omega\in \pi_{2d-1}(N)$. Let $n$ be the rank of $H$. The rational homotopy groups $\pi_*^\QQ(N) = \pi_{*+1}(N)\tensor \QQ$, with the Whitehead product, may be identified with the free graded Lie algebra $\LL(\alpha_1,\ldots,\alpha_n)$ over $\QQ$ on classes $\alpha_1,\ldots,\alpha_n$ of degree $d-1$. The homotopy class of the inclusion of the boundary is, up to a sign, represented by the element
$$\omega = \frac{1}{2} \sum_i[\alpha_i^\#,\alpha_i],$$
cf.~Theorem \ref{thm:stasheff}. The rational homotopy groups of the closed manifold $M$ may be identified with the quotient graded Lie algebra,
$$\pi_*^\QQ(M) \cong  \LL(\alpha_1,\ldots,\alpha_n)/(\omega).$$

\begin{theorem} \label{thm:fb}
Let $M$ be a closed $(d-1)$-connected $2d$-dimensional manifold, where $d\geq 2$, and consider the classifying spaces
$$X = B\aut(M),\quad X_*= B\aut_*(M).$$
If $\rank H>2$, then

\begin{enumerate}
\item  \label{eq:one} Both groups $\pi_1(X)$ and $\pi_1(X_*)$ surject onto $\Aut(H,\mu,Jq)$ with finite kernel.
\item \label{eq:two} The Quillen dg Lie algebras $\lambda(\cX)$ and $\lambda(\cX_*)$ are formal.
\item \label{eq:three} There are $\pi_1$-equivariant isomorphisms of graded Lie algebras
\begin{align*}
\pi_*^\QQ(\cX) & \cong \OutDer^+\big(\LL(\alpha_1,\ldots,\alpha_n)/(\omega) \big), \\
\pi_*^\QQ(\cX_*) & \cong \Der^+\big(\LL(\alpha_1,\ldots,\alpha_n)/(\omega) \big).
\end{align*}
\end{enumerate}
\end{theorem}

Statement \eqref{eq:one} about the homotopy mapping class groups $\pi_1(X) = \pi_0\aut(M)$ and $\pi_1(X_*) = \pi_0\aut_*(M)$ was established in \cite{BM}, see also \cite{Baues}. Statements \eqref{eq:two} and \eqref{eq:three} are consequences of the following general result.

\begin{theorem} \label{thm:smg}
Let $M$ be a simply connected space of finite $\QQ$-type with homotopy Lie algebra $L = \pi_*^\QQ(M)$. Assume that $M$ is coformal and let $f\colon \LL\to L$ be the minimal model. Suppose that
\begin{enumerate}
\item The graded Lie algebra $L$ has trivial center.
\item The map $f^*\colon \Der L \to \Der_f(\LL,L)$ induces an isomorphism in homology in non-negative degrees.
\end{enumerate}
Then the universal simply connected fibration with fiber $M$,
$$M \to E_M \to B_M,$$
is rationally modeled by the short exact sequence of graded Lie algebras
$$0\to L \to \Der^+ L \to \Der^+ L/ \ad L\to 0.$$
\end{theorem}

\begin{proof}
Consider the pullback of chain complexes
\begin{equation} \label{eq:free pullback}
\xymatrix{\Der f \ar[d]^-{pr_2} \ar[r]^-{pr_1} & \Der \LL \ar[d]^-{f_*} \\
\Der L \ar[r]^-{f^*} & \Der_f(\LL,L),}
\end{equation}
where $\Der f$ is the chain complex of pairs $(\theta,\eta)$ of derivations $\theta\in \Der \LL$ and $\eta \in \Der L$ such that $f^*(\theta) = f_*(\eta)$. The coordinatewise Lie bracket on $\Der f$ makes it into a dg Lie algebra, and the chain maps $pr_1$ and $pr_2$ become morphisms of dg Lie algebras. Since $f\colon \LL\to L$ is a surjective quasi-isomorphism and $\LL$ is cofibrant, Lemma \ref{lemma:qi} implies that the chain map $f_*$ is a surjective quasi-isomorphism. Since \eqref{eq:free pullback} is a pullback, it follows that $pr_1$ is a surjective quasi-isomorphism as well. By hypothesis, $f^*$ induces an isomorphism in homology in non-negative degrees.

As the reader may check, a morphism of dg Lie algebras may be defined by
$$\ad \colon \LL \to \Der f,\quad \ad(\xi) = (\ad_{\xi},\ad_{f(\xi)}).$$
Its mapping cone $\Der f\dquot \ad \LL$ admits a dg Lie algebra structure such that the map $\Der f \to \Der f \dquot \ad \LL$ becomes a morphism of dg Lie algebras. After taking positive truncations, we get a commutative diagram of dg Lie algebras where all vertical morphisms are quasi-isomorphisms
\begin{equation} \label{eq:zig-zag}
\xymatrix{
\LL \ar[rr]^-{\ad} && \Der^+ \LL \ar[r] & \Der^+ \LL \dquot \ad \LL \\
\LL \ar@{=}[u] \ar[d]_-\sim^-f \ar[rr]^-{\iota} && \Der^+ f \ar[u]_{pr_1}^\sim \ar[d]_-\sim^-{pr_2} \ar[r] & \Der^+ f \dquot \ad \LL \ar[u]^-\sim \ar[d]_-\sim \\
L \ar[rr]^-{\ad} && \Der^+ L \ar[r] & \Der^+ L \dquot \ad L.}
\end{equation}
By Corollary \ref{cor:tanre}, the top row is a model for the universal simply connected fibration with fiber $M$.

If the center of $L$ is trivial, then the morphism $\ad \colon L\to \Der^+ L$ is injective and the natural morphism $\Der^+ L \dquot \ad L\to \Der^+ L/ \ad L$ is a surjective quasi-isomorphism. Hence,  in this case the bottom row of \eqref{eq:zig-zag} is weakly equivalent to the short exact sequence
$$0\to L \to \Der^+ L \to \Der^+ L/\ad L \to 0.$$
\end{proof}

To finish the proof of Theorem \ref{thm:fb}, we will verify the hypotheses of Theorem \ref{thm:smg}. This is done in Proposition \ref{prop:center} and Lemma \ref{lemma:f^*} below. For $M$ as in Theorem \ref{thm:fb}, we have $L = \pi_*^\QQ(M) = \LL(\alpha_1,\ldots,\alpha_n)/(\omega)$, and a cofibrant dg Lie algebra model $M$ is given by
$$\LL = \big(\LL(\alpha_1,\ldots,\alpha_n,\rho),\delta\big),\quad \delta(\rho) = \omega,\quad \delta(\alpha_i) = 0.$$
The generators $\alpha_i$ have degree $d-1$ and the generator $\rho$ has degree $2d-1$. The evident morphism of dg Lie algebras $f\colon \LL\rightarrow L$ is a quasi-isomorphism.

\begin{proposition} \label{prop:center}
Let $M$ be a $(d-1)$-connected $2d$-dimensional closed manifold where $d\geq 2$ and let $n=\rank \HH_d(M)$. If $n > 2$ then the homotopy Lie algebra $\pi_*^\QQ(M)$ has trivial center.
\end{proposition}

\begin{proof}
We invoke \cite[Proposition 2]{Bog} which says that a graded Lie algebra $L$ of finite global dimension has non-trivial center only if the Euler characteristic $\chi(L)$ is zero, where
$$\chi(L) = \sum_i (-1)^i \dim_\QQ \Ext_{UL}^i(\QQ,\QQ),$$
and $UL$ denotes the universal enveloping algebra of $L$. For $L = \pi_*^\QQ(M)$, we have that $\Ext_{UL}^i(\QQ,\QQ) \cong \HH^{id}(M;\QQ)$, because $\HH^*(M;\QQ)$ is Koszul dual to $\pi_*^\QQ(M)$ (see \cite{Berglund}). It follows that $L$ has global dimension $2$ and that $\chi(L) = 2-n$, whence $L$ must have trivial center whenever $n > 2$.
\end{proof}

\begin{lemma} \label{lemma:f^*}
The chain map $f^* \colon \Der L \rightarrow \Der_f(\LL,L)$ induces an isomorphism in homology in non-negative degrees.
\end{lemma}

\begin{proof}
Recall that $\LL$ is the free graded Lie algebra on generators $\alpha_1,\ldots,\alpha_n$ in degree $d-1$ and $\rho$ in degree $2d-1$, with differential $d\alpha_i = 0$ and $d\rho  = \omega$. Note that $L$ is concentrated in degrees $r(d-1)$, for $r\geq 1$. The chain complex $\Der_f(\LL,L)$ is spanned by elements of the form $\zeta \frac{\partial}{\partial \alpha_i^\#}$ in degrees congruent to $0$ modulo $(d-1)$ and $\xi \frac{\partial}{\partial \rho}$ in degrees congruent to $-1$ modulo $(d-1)$, where $\zeta,\xi\in L$. The differential is given by $D\big( \xi \frac{\partial}{\partial \rho} \big) = 0$ and 
\begin{equation} \label{eq:differential}
D\big(\zeta \frac{\partial}{\partial \alpha_i^\#}\big) = \pm [\zeta,f(\alpha_i)] \frac{\partial}{\partial \rho}.
\end{equation}
One checks that the image of the map $f^*\colon \Der L \rightarrow \Der_f(\LL,L)$ is precisely the kernel of $D$ in degrees congruent to $0$ modulo $(d-1)$. For every element $\xi\in L_{r(d-1)}$, we have a cycle $\xi \frac{\partial}{\partial \rho}$ in degree $(r-2)(d-1)-1$. If $r\geq 2$, then $\xi$ is decomposable and \eqref{eq:differential} shows that $\xi \frac{\partial}{\partial \rho}$ is in the image of $D$. For $r=1$, it represents a non-trivial homology class outside the image of $f^*$, but this is harmless because it is of negative degree.
\end{proof}

\section{On the structure of derivation Lie algebras}
In this section we will analyze the structure of the graded Lie algebra $\Der_{\omega} \LL(V)$ associated to a graded anti-symmetric inner product space $V$. This will be an essential ingredient both in the proof of homological stability for $B\aut_\partial(M_{g,1})$ and $B\tDiff_\partial(M_{g,1})$ and for the calculation of the stable cohomology.

\subsection{$\SP$-modules} \label{sec:SPD}
Recall that a graded anti-symmetric inner product space of degree $D$ is a finite dimensional graded vector space $V$ together with a non-degenerate bilinear form of degree $-D$,
$$V\tensor V \to \QQ,\quad x\tensor y \mapsto \langle x,y\rangle,$$
such that $\langle x,y \rangle = - (-1)^{|x||y|} \langle y,x \rangle$ for all $x,y\in V$.

A morphism $f\colon V\to W$ of graded anti-symmetric inner product spaces of the same degree is a linear map of degree $0$ such that
$$\langle fx,fy\rangle_W = \langle x,y\rangle_V$$
for all $x,y\in V$. Let $\SP^D$ denote the category of graded anti-symmetric inner product spaces of degree $D$.
The adjoint of a morphism $f\colon V\to W$ is the unique linear map $f^!\colon W\rightarrow V$ such that
$$\langle f^! x, y \rangle_V = \langle x,fy\rangle_W$$
for all $x\in W$, $y\in V$. Clearly $f^! f = 1_V$. In particular, every morphism $f\colon V\rightarrow W$ is injective and there is an isomorphism of inner product spaces
$$W\cong V \oplus V^\perp, \quad x\mapsto (f^!(x), x - f f^!(x)),$$
where
$$V^\perp = \set{x\in W}{\mbox{$\langle x,fy \rangle_W = 0$ for all $y\in V$}} = \ker f^!.$$
(Note however that $f^!$ is not a morphism in $\SP^D$.)

We define an $\SP^D$-module in a category $\VV$ to be a functor $\SP^D\to \VV$.
In what follows we will show that $\Der_\omega \LL(V)$ is the value at $V$ of an $\SP^D$-module in graded Lie algebras.

Recall that $\LL(V)$ denotes the free graded Lie algebra on $V$. For a linear map $f\colon V\rightarrow W$, we let $\LL(f)\colon \LL(V) \rightarrow \LL(W)$ denote the induced morphism of graded Lie algebras. Given a morphism $f\colon V\rightarrow W$ in $\SP^D$, we define a morphism of graded Lie algebras
$$\chi_f\colon \Der \LL(V) \rightarrow \Der \LL(W)$$
as follows. For $\theta\in \Der\LL(V)$, we let $\chi_f(\theta)\in \Der \LL(W)$ be the unique derivation that satisfies
$$\chi_f(\theta)(x) = \LL(f) \theta (f^! x)$$
for all $x\in W$. It is easy to check that $\chi_g \chi_f = \chi_{gf}$ when the composition $gf$ is defined. It is perhaps not evident from the definition that $\chi_f$ is a morphism of Lie algebras, but this will be verified below.

\begin{proposition} \label{prop:der functor}
If $f\colon V\rightarrow W$ is a morphism in $\SP^D$, then $\chi_f\colon \Der \LL(V) \rightarrow \Der \LL(W)$ is an injective morphism of graded Lie algebras.
\end{proposition}

\begin{proof}
Let $\theta,\eta\in\Der \LL(V)$. We have an equality of maps from $\LL(V)$ to $\LL(W)$,
\begin{equation} \label{eq:commute}
\chi_f(\theta) \circ \LL(f) = \LL(f) \circ \theta.
\end{equation}
This follows because both sides are $\LL(f)$-derivations and for every $y\in V$ we have
$$\chi_f(\theta)\LL(f)(y) = \chi_f(\theta)(fy)  = \LL(f) \theta (f^! f y) = \LL(f)\theta(y).$$
Next, to verify the equality
$$[\chi_f(\theta),\chi_f(\eta)] = \chi_f[\theta,\eta],$$
observe that both sides are derivations, so equality may be checked by evaluating at generators $x\in V$. But
\begin{align*}
[\chi_f(\theta),\chi_f(\eta)](x) & = \chi_f(\theta) \chi_f(\eta)(x) - (-1)^{|\chi_f(\theta)||\chi_f(\eta)|} \chi_f(\eta)\chi_f(\theta)(x) \\
& = \chi_f(\theta)\LL(f)\eta(f^! x) -(-1)^{|\theta||\eta|} \chi_f(\eta)\LL(f)\theta(f^! x) \\
& = \LL(f)\theta \eta (f^!x) - (-1)^{|\theta||\eta|} \LL(f) \eta \theta(f^! x) \\
& = \chi_f([\theta,\eta])(x),
\end{align*}
where we have used \eqref{eq:commute} in the middle step.

To check injectivity, define the map
$$\psi_f\colon \Der \LL(W) \rightarrow \Der \LL(V)$$
by requiring
$$\psi_f(\theta)(x) = \LL(f^!) \theta(fx),$$
for $\theta \in \Der \LL(W)$ and $x\in V$. Then one checks that the composite $\psi_f \circ \theta_f$ is the identity map on $\Der \LL(V)$. (Note however that $\psi_f$ is not necessarily a morphism of Lie algebras.)
\end{proof}

Next, recall from \S\ref{sec:ham} that there is a canonical element $\omega = \omega_V \in \LL(V)$ associated to every graded anti-symmetric inner product space $V$. If $\alpha_1,\ldots,\alpha_r$ is a graded basis with dual basis $\alpha_1^\#,\ldots,\alpha_r^\#$, then
$$\omega = \frac{1}{2} \sum_i \big[\alpha_i^\#,\alpha_i\big].$$
In what follows, we will show that the evaluation map,
$$ev_{\omega}\colon \Der \LL(V) \to V,\quad ev_{\omega}(\theta) = \theta(\omega),$$
and the map
$$\theta_{-,-}\colon \LL(V)\tensor V \to \Der \LL(V),\quad \theta_{\xi,x}(y) = \xi \langle x,y\rangle,$$
are natural transformations of $\SP^D$-modules.

\begin{proposition}
Given a morphism $f\colon V\rightarrow W$ in $\SP^D$, the diagram
$$\xymatrix{\Der \LL(V) \ar[r]^-{ev_{\omega_V}} \ar[d]_-{\chi_f} & \LL(V) \ar[d]^-{\LL(f)} \\ \Der \LL(W) \ar[r]^-{ev_{\omega_W}} & \LL(W)}$$
is commutative.
\end{proposition}

\begin{proof}
As we noted above, the map $f\colon V\rightarrow W$ is injective and induces an isomorphism of inner product spaces
$$W \cong V\oplus V^\perp,$$
so we may without loss of generality assume that $W = V\oplus V^\perp$. Then we have that $\omega_W = \omega_V + \omega_{V^\perp}$, where $\omega_V\in \LL(V)$ and $\omega_{V^\perp}\in\LL(V^\perp)$. For a derivation $\theta$ on $\LL(V)$, we may describe $\chi_f(\theta)$ as the unique derivation on $\LL(W)$ that restricts to $\theta$ on $\LL(V)$ and restricts to zero on $\LL(V^\perp)$. Thus,
$$\chi_f(\theta)(\omega_W) = \chi_f(\theta)(\omega_V) + \chi_f(\theta)(\omega_{V^\perp}) = \theta(\omega_V),$$
which proves the claim.
\end{proof}

\begin{proposition}
The map
$$\theta_{-,-}\colon \LL(V)\tensor V \to \Der \LL(V).$$
defines an isomorphism of $\SP^D$-modules (of degree $-D$).
\end{proposition}

\begin{proof}
Since the inner product is non-degenerate and since a derivation is determined by its value on generators, it follows that $\theta_{-,-}$ is an isomorphism. We need to verify that the diagram
$$\xymatrix{\LL(V)\tensor V \ar[d]_-{\LL(f)\tensor f} \ar[r]^-{\theta_{-,-}} & \Der \LL(V) \ar[d]^-{\chi_f} \\
\LL(W)\tensor W \ar[r]^-{\theta_{-,-}} & \Der \LL(W)}$$
is commutative for every morphism $f\colon V\rightarrow W$ in $\SP^D$. Indeed, for $x\in V$, $y\in W$ and $\xi \in \LL(V)$, we have
$$\chi_f(\theta_{\xi,x})(y) = \LL(f)\theta_{\xi,x}(f^!y) = \LL(f)(\xi \langle x, f^! y\rangle),$$
and on the other hand
$$\theta_{\LL(f)\xi,fx}(y) = \LL(f)(\xi) \langle fx,y\rangle.$$
These elements are clearly equal.
\end{proof}

The kernel of the evaluation map is exactly $\Der_{\omega} \LL(V)$. It is a graded Lie subalgebra of $\Der \LL(V)$ and since $ev_\omega$ is a morphism in $\SP^D$, it inherits an $\SP^D$-module structure. By Corollary \ref{cor:imev}, the image of $ev_\omega$ is the space of decomposables $[\LL(V),\LL(V)] = \LL^{\geq 2}(V)$. Hence, there is a short exact sequence of $\SP^D$-modules,
$$0\rightarrow \Der_{\omega} \LL(V) \rightarrow \Der \LL(V) \rightarrow \LL^{\geq 2}(V) \to 0.$$

Clearly, the Lie bracket $\LL(V)\tensor V \rightarrow \LL^{\geq 2}(V)$ defines a morphism of $\SP^D$-modules. Let, for the moment, $\gl(V)$ denote the kernel. It follows from Proposition \ref{prop:bracket} that we have a commutative diagram of $\SP^D$-modules, where the rows are exact and the vertical maps are isomorphisms,
$$\xymatrix{0 \ar[r] & \gl(V) \ar[r] \ar@{-->}[d]_-\cong & \LL(V)\tensor V \ar[r]^-{[-,-]} \ar[d]_-{\theta_{-,-}}^-\cong & \LL^{\geq 2}(V) \ar[r] \ar@{=}[d] & 0\\
0\ar[r] & \Der_{\omega} \LL(V) \ar[r] & \Der \LL(V) \ar[r]^-{ev_{\omega}} & \LL^{\geq 2}(V) \ar[r] & 0.}
$$
In fact, the top row is functorial not only with respect to morphisms in $\SP^D$, but for all linear maps. In what follows, we will identify the entries in the top row as the values at $V$ of certain Schur functors.

\subsection{On the Lie operad}
Let $\Lie = \{\Lie(n)\}_{n\geq 0}$ denote the Lie operad. As a vector space, $\Lie(n)$ is the subspace of the free Lie algebra $\LL(x_1,\ldots,x_n)$ spanned by all Lie monomials in $x_1,\ldots,x_n$ containing every generator exactly once, cf.~\cite[(1.3.9)]{GiKa}. There is a left action of the symmetric group $\Sigma_n$ given by permuting the generators;
$$\sigma f(x_1,\ldots,x_n) = f(x_{\sigma_1},\ldots,x_{\sigma_n}).$$
Define the Lie monomials $\ell_n = \ell(x_1,\ldots,x_n)$ and $r_n  = r(x_1,\ldots,x_n)$ inductively by $\ell(x_1) = r(x_1) = x_1$, and
\begin{align*}
\ell(x_1,\ldots,x_n) & = \big[x_1,\ell(x_2,\ldots,x_n)\big],\\
r(x_1,\ldots,x_n) & = \big[r(x_1,\ldots,x_{n-1}),x_n\big].
\end{align*}
It is an exercise to show that $\ell_n$ generates $\Lie(n)$ as a $\Sigma_n$-module. In fact, identifying $\Sigma_{n-1}$ with the subgroup of $\Sigma_n$ consisting of those permutations that leave $n$ fixed, it is not difficult to show that the elements
$$\sigma\ell_n =  \big[x_{\sigma_1},[\ldots [x_{\sigma_{n-2}}, [x_{\sigma_{n-1}},x_n]]\ldots ] \big],\quad \sigma \in \Sigma_{n-1},$$
form a vector space basis for $\Lie(n)$, see \cite[\S5.6.2]{Reutenauer}. In particular, $\dim \Lie(n) = (n-1)!$. Similarly, $r_n$ generates $\Lie(n)$ as a $\Sigma_n$-module.

The Lie operad is a cyclic operad, which means that the action of $\Sigma_{n-1}$ on $\Lie(n-1)$ extends to an action of $\Sigma_{n}$ in a way compatible with the operad structure, cf.~\cite[(3.9)(b)]{GK1}. Let $\Lie((n))$ denote $\Lie(n-1)$ viewed as a $\Sigma_n$-module. For $p\in \Lie((n))$ and $\sigma\in \Sigma_n$, we write $p\sigma = \sigma^{-1} p$. Let $t_n$ denote the cyclic permutation $(1 2 \ldots n) \in \Sigma_n$. We will abbreviate $pt_n$ to $p t$ when the index $n$ is clear from the context. Since $\Sigma_n$ is generated by $\Sigma_{n-1}$ and $t_n$, the action can be computed using \cite[Theorem (2.2)]{GK1} or \cite[Proposition 42]{Markl}: For $p\in \Lie((m))$, $q\in \Lie((n))$ and $1\leq i \leq m-1$ we have
\begin{equation} \label{eq:cyclic operad}
\big( p \circ_i q \big) t = \left\{ \begin{array}{ll} qt \circ_{n-1} pt, & i = 1, \\ pt \circ_{i-1} q, & 2\leq i \leq m-1. \end{array} \right.
\end{equation}
The action has a compelling graphical description. If we represent a Lie monomial in $\Lie((n))$ by a planar binary rooted tree with leaves labeled by $1,\ldots,n-1$ and the root labeled by $n$, then a permutation $\sigma\in \Sigma_n$ acts by relabeling and then reinterpreting the tree as a Lie monomial using the leaf labeled by $n$ as the root. For example, here is the computation of the action of $t_5 = (1 2 3 4 5)\in \Sigma_5$ on $r_4 = [[[x_1,x_2],x_3],x_4] \in \Lie((5))$:
$$
\rtree \quad \stackrel{t_5}{\mapsto}\quad  \trtree = \ntree
$$
$$
[[[x_1,x_2],x_3],x_4]  \quad \stackrel{t_5}{\mapsto}\quad  [x_1,[[x_2,x_3],x_4]].
$$
More generally, one can work out that the action of $t_n^{-i}$ on $r_{n-1}\in \Lie((n))$ is given by
\begin{equation} \label{eq:t-action}
t_n^{-i} r_{n-1} = \left\{ \begin{array}{lc} \ell_{n-1}, & i=1, \\ \big[\ell(x_1,\ldots,x_{n-i}),r(x_{n-i+1},\ldots,x_{n-1})\big], & 1< i < n, \\ r_{n-1}, & i = n. \end{array} \right.
\end{equation} 

\begin{proposition} \label{prop:lie sequence}
For every $n\geq 2$, the sequence
\begin{equation} \label{eq:lie sequence}
0 \to \Lie((n)) \xrightarrow{\mu} \QQ\Sigma_{n}\tensor_{\Sigma_{n-1}} \Lie(n-1) \xrightarrow{\epsilon} \Lie(n) \to 0,
\end{equation}
$$\mu(\xi) = \sum_{i=1}^n t_n^i \tensor t_n^{-i}\xi,\quad \epsilon(\sigma\tensor \zeta) = \sigma[\zeta,x_n],$$
is a short exact sequence of $\Sigma_n$-modules.
\end{proposition}

\begin{remark}
The existence of an exact sequence of the form \eqref{eq:lie sequence} has been noted before, see \cite[Corollary 2.7]{Whitehouse}, but the explicit expressions for the maps and the following direct proof of exactness seem to be new.
\end{remark}

\begin{proof}
It is immediate that $\epsilon$ is $\Sigma_n$-equivariant. To check that $\mu$ is $\Sigma_n$-equivariant, it suffices to check that $\mu(t_n\xi) = t_n\mu(\xi)$ and $\mu(\rho \xi) = \rho\mu(\xi)$ for all $\xi \in \Lie(n-1)$ and all $\rho\in \Sigma_{n-1}$. The former is clear. For the latter, we use the fact that the permutation $t_n^{-i}\rho t_n^{\rho^{-1}(i)}$ fixes $n$, so that
\begin{align*}
\mu(\rho \xi) & = \sum_{i=1}^n t_n^i \tensor t_n^{-i}(\rho \xi) = \sum_{i=1}^n t_n^i \tensor (t_n^{-i}\rho t_n^{\rho^{-1}(i)}) t_n^{-\rho^{-1}(i)}  \xi \\
& = \sum_{i=1}^n t_n^i (t_n^{-i}\rho t^{\rho^{-1}(i)}) \tensor t_n^{-\rho^{-1}(i)}  \xi = \rho \sum_{i=1}^n t_n^{\rho^{-1}(i)} \tensor t_n^{-\rho^{-1}(i)}  \xi \\
& = \rho \mu(\xi),
\end{align*}
where the last equality follows by changing the summation index; $i' = \rho^{-1}(i)$.

Since $\epsilon(1\tensor r_{n-1})= r_n$, and since $\Lie(n)$ is generated by $r_n$ as a $\Sigma_n$-module, it is clear that $\epsilon$ is surjective. To see that $\mu$ is injective, note that there is an isomorphism of right $\Sigma_{n-1}$-modules,
$$\QQ\Sigma_n \cong \QQ C_n \tensor \QQ\Sigma_{n-1},$$
where $C_n$ denotes the cyclic subgroup of $\Sigma_n$ generated by $t_n$.
Hence, there is an isomorphism $\QQ\Sigma_n\tensor_{\Sigma_{n-1}} \Lie(n-1) \cong \QQ C_n \tensor \Lie(n-1)$. We can define a $\QQ$-linear splitting of $\mu$ by sending $t_n^i\tensor \xi$ to $0$ if $i\ne 0$ and to $\xi$ if $i=0$.

To show that $\epsilon \mu = 0$, it suffices to verify that $\epsilon\mu(r_{n-1}) = 0$, because $\Lie(n-1)$ is generated by $r_{n-1}$ as a $\Sigma_{n-1}$-module and $\mu$ and $\epsilon$ are equivariant. 
Using the formula for the action \eqref{eq:t-action}, we get that
\begin{align*}
\epsilon&\mu(r_{n-1})  = \sum_{i=1}^n t_n^i \big[t_n^{-i} r_{n-1},x_n\big] \\
& = t_n\big[\ell_{n-1},x_n] + \sum_{i=2}^{n-1} t_n^i \big[ [\ell(x_1,\ldots, x_{n-i}), r(x_{n-i+1},\ldots,x_{n-1})],x_n \big] + [r_{n-1},x_n] \\
& = \big[\ell(x_2,\ldots,x_n),x_1] + \sum_{i=2}^{n-1} \big[ [\ell(x_{i+1},\ldots, x_n), r(x_1,\ldots,x_{i-1})],x_i \big] + [r_{n-1},x_n].
\end{align*}
Using the Jacobi relation and anti-symmetry, we may rewrite the middle sum as:
\begin{align*}
\sum_{i=2}^{n-1} \big[\ell(x_{i+1},\ldots, x_n),[r(x_1,\ldots,x_{i-1}),x_i]\big] - \big[[x_i,\ell(x_{i+1},\ldots, x_n)], r(x_1,\ldots,x_{i-1}) \big] \\
= \sum_{i=2}^{n-1} \big[\ell(x_{i+1},\ldots, x_n),r(x_1,\ldots,x_{i-1},x_i)\big] - \big[\ell(x_i,x_{i+1},\ldots, x_n),r(x_1,\ldots,x_{i-1})\big].
\end{align*}
The last expression is a telescoping sum, whose surviving terms are $[x_n,r_{n-1}]$ and $-\big[\ell(x_2,\ldots,x_n),x_1\big]$. These cancel the outer terms in the expression for $\epsilon\mu(r_{n-1})$ we found above. Thus, $\epsilon\mu(r_{n-1}) = 0$.

Finally, exactness of \eqref{eq:lie sequence} follows from a dimension count: So far, we know that \eqref{eq:lie sequence} is a chain complex and that the homology vanishes except possibly at the middle term. Since $\dim \Lie(n) = (n-1)!$ and $\dim \QQ\Sigma_n\tensor_{\Sigma_{n-1}} \Lie(n-1) = n(n-2)!$, the Euler characteristic of the complex is $(n-2)! - n(n-2)! + (n-1)! = 0$, which implies that the middle homology must vanish.
\end{proof}

Recall that $\SP^D$ denotes the category of graded anti-symmetric inner product spaces of degree $D$ (see \S\ref{sec:SPD}).
The cyclic operad $\Lie$ determines an $\SP^D$-module $V\mapsto \Lie((V))$, defined by
$$\Lie((V)) = s^{-D}\bigoplus_{n\geq 2} \Lie((n))\tensor_{\Sigma_n} V^{\tensor n}.$$

\begin{proposition} \label{prop:Lie iso}
There is an isomorphism of $\SP^D$-modules,
$$\Der_\omega \LL(V) \cong \Lie((V)).$$
Explicitly, the derivation $\eta_{\xi,h}$ corresponding to $\xi\tensor h_1\tensor \ldots \tensor h_n$ sends $x\in V$ to
$$\eta_{\xi,h}(x) = \sum_{i=1}^n (-1)^{\epsilon_i^h} \big(t_n^{-i} \xi \big)(h_{i+1},\ldots,h_n,h_1,\ldots,h_{i-1}) \langle h_i,x \rangle,$$
where $\epsilon_i^h = \big( |h_1|+\cdots +|h_i|\big) \big(|h_{i+1}| + \cdots + |h_n|\big)$.
\end{proposition}

\begin{proof}
The claim is proved by considering the following commutative diagram for $n\geq 2$, where the rows are short exact sequences and all vertical maps are isomorphisms (and we remember that the map $\theta_{-,-}$ is of degree $-D$):
{\small $$
\xymatrix{
\Lie((n))\tensor_{\Sigma_n} V^{\tensor n} \ar@{-->}[d]_-\cong \ar[r]^-{\mu} & \big( \QQ\Sigma_n\tensor_{\Sigma_{n-1}} \Lie(n-1) \big) \tensor_{\Sigma_n} V^{\tensor n} \ar[d]_-\cong^-{\alpha} \ar[r]^-{\epsilon} & \Lie(n)\tensor_{\Sigma_n} V^{\tensor n} \ar[d]_-\cong^-{\beta} \\ \gl^n(V) \ar@{-->}[d]_-\cong \ar[r] & \LL^{n-1}(V)\tensor V \ar[d]_-\cong^-{\theta_{-,-}} \ar[r]^-{[-,-]} & \LL^n(V) \ar@{=}[d] \\
\Der_{\omega}^{n-2} \LL(V) \ar[r] & \Der^{n-2} \LL(V) \ar[r]^-{ev_{\omega}} & \LL^n(V).}
$$}
The top row is obtained by applying the functor $-\tensor_{\Sigma_n} V^{\tensor n}$ to the short exact sequence in Proposition \ref{prop:lie sequence}. Explicitly, the maps are given by
\begin{align*}
\mu(\xi \tensor h_1\tensor \ldots \tensor h_n) & = \sum_{i=1}^n \big( t^i \tensor t^{-i} \xi \big) \tensor h_1\tensor \ldots \tensor h_n, \\
\epsilon(\rho \tensor \zeta \tensor h_1\tensor \ldots \tensor h_n) & = \big(\rho\big[\zeta,x_n]\big) \tensor h_1\tensor \ldots \tensor h_n, \\
\alpha(\rho \tensor \zeta \tensor h_1\tensor \ldots \tensor h_n) & = \pm\zeta(h_{\rho_1},\ldots,h_{\rho_{n-1}})\tensor h_{\rho_n}, \\
\beta(\zeta\tensor h_1\tensor \ldots \tensor h_n) & = \zeta(h_1,\ldots,h_n),
\end{align*}
The sign in the formula for $\alpha$ is dictated by the standard sign convention, according to which a sign is introduced every time two elements of odd degree are permuted. Everything is natural in $V\in \SP^D$, so there results an isomorphism of $\SP^D$-modules $\Der_\omega\LL(V) \cong \Lie((V))$.
\end{proof}

The isomorphism of Proposition \ref{prop:Lie iso} can be used to endow $\Lie((V))$ with the structure of a graded Lie algebra. This can be made explicit as follows.

Let $p\in\Lie((m))$, $q\in\Lie((n))$. Define for $1\leq i \leq m$ and $1\leq j\leq n$,
\begin{equation} \label{eq:circ ij}
p{}_i\circ_j q = p \circ_i q t^j,
\end{equation}
where $p \circ_m q$ is defined to be $\big(p t \circ_{m-1} q\big)t^{-1}$.

Let $V$ be a graded anti-symmetric inner product space of degree $D$. For $1\leq i\leq m$ and $1\leq j\leq n$, define the \emph{contraction} ${}_i\circ_j\colon V^{\tensor m}\tensor V^{\tensor n} \to V^{\tensor m+n-2}$ on generators $v=v_1\tensor \cdots \tensor v_m$ and $w=w_1\tensor \cdots \tensor w_n$ by
{\small 
$$v{}_i\circ_j w = (-1)^{c_{i,j}^{v,w}} v_1\tensor \cdots \tensor v_{i-1} \tensor w_{j+1} \tensor \cdots \tensor w_n \tensor w_1 \tensor \cdots \tensor w_{j-1} \tensor v_{i+1} \tensor \cdots \tensor v_m \langle w_j, v_i \rangle,$$}
where the sign is given by
$$c_{i,j}^{v,w} = \big( |v_i| + \cdots + |v_m| - D\big) \big( |w| - D \big) + \big(|w_1| + \cdots + |w_j| \big) \big(|w_{j+1}| + \cdots + |w_n| \big) + 1.$$
\begin{proposition}
The equality
\begin{equation*} 
\big[ \eta_{p,v}, \eta_{q,w} \big] = \sum_{i=1}^m \sum_{j=1}^n \eta_{p{}_i\circ_j q, v{}_i \circ_j w}
\end{equation*}
holds in the graded Lie algebra of derivations on $\LL V$.
\end{proposition}

\begin{proof}
Inductive application of the formula \eqref{eq:cyclic operad} yields
\begin{equation} \label{eq:cyclic operad identities}
\big( p \circ_i q \big) t^j = \left\{ \begin{array}{ll} pt^j \circ_{i-j} q, & 1\leq j \leq i-1, \\ qt^{j-i+1} \circ_{n+i-j-1} pt^i, & i\leq j \leq n+i-2, \\ pt^{j-n+2} \circ_{m+n+i-j-2} q, & n+i-1 \leq j \leq m+n+i-3. \end{array} \right.
\end{equation}
The proof is a long but straightforward calculation that uses the rules \eqref{eq:cyclic operad identities} and the basic fact that
$$\eta(q(w_1,\ldots,w_n)) = \sum_{j=1}^{n-1} \pm q(w_1,\ldots,w_{j-1},\eta(w_j),w_{j+1},\ldots,w_{n-1})$$
for every derivation $\eta$ on $\LL V$. We omit the details.
\end{proof}

It follows that an explicit description of the Lie bracket on $\Lie((V))$ is
\begin{equation} \label{eq:lie bracket}
\big[\xi\tensor h,\zeta \tensor g\big] = \sum_{i,j} \xi {}_i\circ_j \zeta \tensor h{}_i\circ_j g.
\end{equation}
Thus, $\Der_\omega \LL(V)$ and $\Lie((V))$ are, naturally in $V\in \SP^D$, isomorphic as graded Lie algebras.
More generally, one can prove that the formula \eqref{eq:lie bracket} defines a graded Lie algebra structure on $\CC((V))$ for any cyclic operad $\CC$, where ${}_i\circ_j$ is defined as in \eqref{eq:circ ij}.

\section{Homological stability}
This section contains the proof of rational homological stability for the classifying spaces $B\aut_\partial(M_{g,1})$ and $B\tDiff_\partial(M_{g,1})$. The proof consists in a reduction to a homological stability result for certain arithmetic groups with twisted coefficients; we begin by reviewing this.

\subsection{Polynomial functors and homological stability}
We adopt a naive approach to polynomial functors. By a \emph{polynomial functor of degree $\leq \ell$} we will mean a functor from abelian groups to itself isomorphic to a functor of the form
$$P(H) = \bigoplus_{k=0}^\ell P(k) \tensor_{\Sigma_k} H^{\tensor k},$$
for some sequence of abelian groups $P(k)$ with an action of the symmetric group $\Sigma_k$. Recall that $\Gamma_g$ denotes the automorphism group of the hyperbolic quadratic module $(H_g,\mu,q)$, see Example \ref{ex:hyperbolic}.

\begin{theorem}[Charney {\cite[Theorem 4.3]{Charney}}] \label{thm:charney}
If $P$ is a polynomial functor of degree $\leq \ell$, then the stabilization map
$$\HH_q(\Gamma_g;P(H_g))\rightarrow \HH_q(\Gamma_{g+1};P(H_{g+1}))$$
is an isomorphism for $g>2q+\ell+4$ and a surjection for $g = 2q+\ell+4$.
\end{theorem}

\begin{proof}
In the notation of \cite{Charney}, the group $\Gamma_g$ is isomorphic to the automorphism group of the hyperbolic module in the category $Q^\lambda(A,\Lambda)$, where $A$ is the ring $\ZZ$ with trivial involution, $\lambda = (-1)^d$, and $\Lambda = 0$ if $d$ is even, $\Lambda = \ZZ$ if $d=1,3,7$, and $\Lambda = 2\ZZ$ if $d\ne 1,3,7$ is odd. It is straightforward to verify that $P(H_g)$ is a `central coefficient system of degree $\leq \ell$' whenever $P$ is a polynomial functor of degree $\leq \ell$.
\end{proof}

Let $V_g$ denote the graded anti-symmetric inner product space $s^{-1}\widetilde{\HH}_*(M_{g,1};\QQ) \cong s^{d-1} H_g\tensor \QQ$ and consider the graded Lie algebra
$$\gl_g = \Der_{\omega_g}^+ \LL(V_g)$$
with its natural $\Gamma_g$-action. Let $\sigma = \chi_f\colon \gl_g\rightarrow \gl_{g+1}$ be the morphism of graded Lie algebras induced by the inclusion $V_g\rightarrow V_{g+1}$.

\begin{lemma} \label{lemma:polynomial functor}
Let $d\geq 2$. The component of the Chevalley-Eilenberg chains in total degree $p$, $C_p^{CE}(\gl_g)$, may be identified with the value at $H_g$ of a polynomial functor of degree $\leq \lfloor \frac{3p}{d} \rfloor$.
\end{lemma}

\begin{proof}
As a graded vector space, the Chevalley-Eilenberg chains on a graded Lie algebra $L$ may be described as the value at $L$ of a Schur functor,
$$C_*^{CE}(L) = \bigoplus_{k\geq 0} \Lambda(k) \tensor_{\Sigma_k} L^{\tensor k},$$
where $\Lambda(k)$ is the sign representation of $\Sigma_k$ concentrated in degree $k$. It follows from Proposition \ref{prop:Lie iso} that $C_*^{CE}(\gl_g)$ may be identified with the value at $H_g$ of the Schur functor associated to the symmetric sequence
$$\CC = \Lambda\circ \widetilde{\Lie} \circ I^{d-1},$$
where $I^{d-1}$ is the symmetric sequence with $I^{d-1}(k) = 0$ for $k\ne 1$ and $I^{d-1}(1)$ the trivial representation concentrated in degree $d-1$, and $\widetilde{Lie}$ is the symmetric sequence with $\widetilde{\Lie}(k) = \Lie((k))$ concentrated in degree $2-2d$ for $k\geq 3$ and $\widetilde{\Lie}(k) = 0$ for $k\leq 2$. A calculation with composition products reveals that $\CC(k)$ is concentrated in degrees $\geq \frac{kd}{3}$.

Thus, there is an isomorphism
$$C_p^{CE}(\gl_g) \cong \bigoplus_{k\geq 0} \CC(k)_p \tensor_{\Sigma_k} H_g^{\tensor k},$$
where $\CC(k)_p = 0$ unless $p\geq \frac{kd}{3}$, i.e., $k\leq \frac{3p}{d}$. Hence, $C_p^{CE}(\gl_g)$ may be identified with the value at $H_g$ of a polynomial functor of degree $\leq \lfloor \frac{3p}{d} \rfloor$.
\end{proof}

The following proposition is an immediate consequence of the previous lemma and Charney's theorem.

\begin{proposition} \label{prop:e1 stability}
Let $d\geq 2$ and fix a non-negative integer $p$. The map
$$\sigma_* \colon \HH_q(\Gamma_g;C_p^{CE}(\gl_g)) \rightarrow \HH_q(\Gamma_{g+1};C_p^{CE}(\gl_{g+1}))$$
is an isomorphism for $g>2q+\lfloor \frac{3p}{d} \rfloor + 4$ and a surjection for $g=2q+\lfloor \frac{3p}{d} \rfloor + 4$.
\end{proposition} \hfill $\square$

\begin{theorem} \label{thm:hh stability}
Let $d\geq 2$. The map in hyperhomology
$$\sigma_* \colon \HHH_k(\Gamma_g;C_*^{CE}(\gl_g)) \rightarrow \HHH_k(\Gamma_{g+1};C_*^{CE}(\gl_{g+1}))$$
is an isomorphism for $g>2k+4$ and surjective for $g=2k+4$.
\end{theorem}

\begin{proof}
Consider the first page of the first hyperhomology spectral sequence
$${}^I E_{p,q}^1(g) = \HH_q(\Gamma_g;C_p^{CE}(\gl_g)) \Rightarrow \HHH_{p+q}(\Gamma_g;C_*^{CE}(\gl_g)).$$
The map ${}^I E_{p,q}^1(g) \rightarrow {}^I E_{p,q}^1(g+1)$ is an isomorphism for $g>2q+2p+4$ and a surjection for $g=2q+2p+4$ by Proposition \ref{prop:e1 stability}, because $\frac{3p}{d}\leq 2p$ when $d\geq 2$. The claim then follows from the comparison theorem for spectral sequences.
\end{proof}

So far in this section we might as well have worked over $\ZZ$. However, in the following theorem it will be essential to work over $\QQ$. A vanishing theorem of Borel implies that the stable cohomology may be expressed in terms of cohomology with trivial coefficients and invariants. Denote $H_g^\QQ = H_g\tensor \QQ$.

\begin{theorem} \label{thm:borel vanishing}
If $P$ is a polynomial functor of degree $\leq \ell$, then the natural map
$$\HH^k(\Gamma_g;\QQ)\tensor P(H_g^\QQ)^{\Gamma_g} \rightarrow \HH^k(\Gamma_g;P(H_g^\QQ)) $$
is an isomorphism for $g>2k + \ell +4$.
\end{theorem}

\begin{proof}
This follows by combining Charney's theorem (Theorem \ref{thm:charney}) with Borel's vanishing theorem \cite[Theorem 4.4]{Borel2}. The group $\Gamma_g$ is an arithmetic subgroup of the algebraic group $\Sp_g$ or $O_{g,g}$, depending on whether $d$ is odd or even. Call this algebraic group $G_g$. If $P$ is a polynomial functor, then $P(H_g^\QQ)$ is an algebraic (rational) representation of the algebraic group $G_g$, and we may decompose it as a direct sum,
$$P(H_g^\QQ) = P(H_g^\QQ)^{G_g} \oplus E_1^g\oplus \cdots \oplus E_{r_g}^g,$$
where $E_1^g,\ldots,E_{r_g}^g$ are the non-trivial irreducible subrepresentations. It is easy to check that because $P$ is polynomial of degree $\leq \ell$, the coefficients of the highest weight of $E_i^g$ are bounded above by $\ell$, for all $g$ and all $i$. As explained in \cite[\S 4.6]{Borel2}, this implies that for every $k$, there is an $n(k)$ such that
$$\HH^k(\Gamma_g;E_i^g) = 0,\quad \mbox{for all $g\geq n(k)$ and all $i$.}$$
It follows that the map induced by the inclusion of $P(H_g^\QQ)^{G_g}$ into $P(H_g^\QQ)$,
\begin{equation*}
\HH^k(\Gamma_g;\QQ) \tensor P(H_g^\QQ)^{G_g}\cong \HH^k(\Gamma_g;P(H_g^\QQ)^{G_g}) \rightarrow \HH^k(\Gamma_g;P(H_g^\QQ)),
\end{equation*}
is an isomorphism for all $g\geq n(k)$. Thus, for $k$ fixed, the vertical maps in the diagram
$$
\xymatrix{\HH^k(\Gamma_g;\QQ) \tensor P(H_g^\QQ)^{G_g} \ar[d] \ar[r] & \HH^k(\Gamma_{g+1};\QQ)\tensor P(H_{g+1}^\QQ)^{G_{g+1}} \ar[d] \ar[r] & \cdots \\
\HH^k(\Gamma_g;P(H_g^\QQ)) \ar[r] & \HH^k(\Gamma_{g+1};P(H_{g+1}^\QQ)) \ar[r] & \cdots }
$$
become isomorphisms after continuing far enough to the right. It follows from Theorem \ref{thm:charney} that both the top and the bottom horizontal maps are isomorphisms for $g>2k+\ell+4$ (note that $P(H_g^\QQ)^{G_g} = P(H_g^\QQ)^{\Gamma_g} = \HH^0(\Gamma_g;P(H_g^\QQ))$). This implies that the vertical maps are isomorphisms already for $g>2k+\ell+4$, no matter what $n(k)$ is. Finally, we should point out that $V^{\Gamma_g} = V^{G_g}$ for any algebraic representation $V$ because of density of $\Gamma_g$ in $G_g$ (see e.g. \cite{Borel0}).
\end{proof}

\subsection{Homological stability for homotopy automorphisms} \label{sec:stabilization maps}
Let $M_{g,r}$ denote the result of removing the interiors of $r$ disjointly embedded $2d$-disks from the manifold $M_g = \#^g S^d\times S^d$. The manifold $M_{g+1,1}$ may be realized as the union of $M_{g,1}$ and $M_{1,2}$ along a common boundary component. An automorphism of $M_{g,1}$ that fixes the boundary point-wise may be extended to an automorphism of $M_{g+1,1} = M_{g,1} \cup M_{1,2}$ by letting it act as the identity on $M_{1,2}$. This determines a map of monoids $\aut_\partial(M_{g,1}) \rightarrow \aut_\partial(M_{g+1,1})$, and hence an induced map on classifying spaces
\begin{equation} \label{eq:stabilization map}
\sigma\colon B\aut_\partial(M_{g,1}) \rightarrow B\aut_\partial(M_{g+1,1}),
\end{equation}
which we will refer to as the `stabilization map'. In this section we will prove the following theorem:
\begin{theorem} \label{thm:baut stability}
Let $d\geq 2$. The stabilization map induces an isomorphism
$$\sigma_* \colon \HH_k(B\aut_\partial(M_{g,1});\QQ) \rightarrow \HH_k(B\aut_\partial(M_{g+1,1});\QQ)$$
for $g>2k+4$ and a surjection for $g=2k+4$.
\end{theorem}
Throughout the section we will use the notation
\begin{align*}
X_g & = B\aut_\partial(M_{g,1}), \\
H_g & = \HH_d(M_g;\ZZ), \\
V_g & = s^{d-1}H_g\tensor \QQ, \\
\Gamma_g & = \Aut(H_g,\mu,q), \\
\gl_g & = \Der_{\omega_g}^+ \LL(V_g).
\end{align*}

First, we need to understand the behavior of the stabilization map in homotopy and homology. Proposition \ref{prop:equivariant} yields a $\pi_1(X_g)$-equivariant isomorphism of graded Lie algebras
$$\pi_*^\QQ(\cX_g) \cong \Der_{\omega_g}^+ \LL(V_g).$$
We may choose a basis $\alpha_1,\beta_1,\ldots,\alpha_g,\beta_g$ for $V_g$ in which
$$\omega_g = [\alpha_1,\beta_1] + \cdots + [\alpha_g,\beta_g].$$
The intersection form makes $V_g = H_g^\QQ[d-1]$ into a graded anti-symmetric inner product space and the `stabilization' morphism $\Der_{\omega_g}^+ \LL(V_g) \rightarrow \Der_{\omega_{g+1}}^+ \LL(V_{g+1})$ is induced by the obvious inclusion $V_g\to V_{g+1}$, cf.~ Proposition \ref{prop:der functor}. Explicitly, it is given by extending derivations by zero on the new generators $\alpha_{g+1},\beta_{g+1}$.

\begin{proposition} \label{prop:stabilization}
The isomorphism
$$\pi_*^\QQ(\cX_g) \cong \Der_{\omega_g}^+ \LL(V_g)$$
is compatible with the stabilization maps.
\end{proposition}

\begin{proof}
If $f$ is a self-equivalence of $M_{g,1}$, then $\sigma(f)$ is the self-map of $M_{g+1,1}$ that restricts to $f$ on $M_{g,1}$ and to the identity on $M_{1,2}$, when we realize $M_{g+1,1}$ as the union of $M_{g,1}$ and $M_{1,2}$ along a common boundary component. In other words, the diagram
$$\xymatrix{ & \map_*(M_{g,1},M_{g+1,1}) \\
\aut_\partial(M_{g,1}) \ar[r]^-\sigma \ar[ur]^-{i_*} \ar[d] & \aut_\partial(M_{g+1,1}) \ar[d]_-{j^*} \ar[u]_-{i^*} \\
{*} \ar[r]^-j & \map_*(M_{1,2},M_{g+1,1})}$$
is commutative. The manifold $M_{1,2}$ is homotopy equivalent to a wedge of spheres $S^d\vee S^d \vee S^{2d-1}$, and a Lie model for it is given by the free graded Lie algebra $\LL(\rho,\alpha,\beta)$ with trivial differential, where the generators $\alpha,\beta$ have degree $d-1$, and $\rho$ has degree $2d-2$. The inclusions $i$ and $j$ of $M_{g,1}$ and $M_{1,2}$ into $M_{g+1,1}$ are modeled by the dg Lie algebra morphisms
$$\varphi \colon \LL(V_g)\rightarrow \LL(V_{g+1}),\quad \psi\colon \LL(\rho,\alpha,\beta) \rightarrow \LL(V_{g+1}),$$
respectively, where $\psi(\rho) = \omega$, $\psi(\alpha)  = \alpha_{g+1}$ and $\psi(\beta) = \beta_{g+1}$, and $\varphi$ is induced by the standard inclusion. From our earlier calculation and the naturality of \eqref{eq:aut der}, it follows that the diagram
$$\xymatrix{ & \Der_\varphi^+(\LL(V_g),\LL(V_{g+1})) \\
\Der_{\omega_g}^+(\LL(V_g)) \ar[r]^-{\sigma_*} \ar[ur]^-{\varphi_*} \ar[d] & \Der_{\omega_{g+1}}^+(\LL(V_{g+1})) \ar[d]_-{\psi^*} \ar[u]_-{\varphi^*} \\
0 \ar[r] & \Der_\psi^+(\LL(\rho,\alpha_{g+1},\beta_{g+1}),\LL(V_{g+1}))}$$
is commutative. This pins down $\sigma_*(\theta)$ as the unique derivation on $\LL(V_{g+1})$ that extends $\theta$ and vanishes on $\alpha_{g+1}$ and $\beta_{g+1}$.
\end{proof}

The universal cover spectral sequence,
$$E_{p,q}^2 = \HH_p(\pi_1(X);\HH_q(\cX)) \Rightarrow \HH_{p+q}(X),$$
is natural in $X$. To prove Theorem \ref{thm:baut stability} it is therefore sufficient to show that
$$\sigma \colon \HH_p(\pi_1(X_g);\HH_q(\cX_g;\QQ)) \rightarrow \HH_p(\pi_1(X_{g+1});\HH_q(\cX_{g+1};\QQ))$$
is an isomorphism if $g> 2p+2q+4$ and a surjection for $g = 2p+2q+4$. This will follow from Proposition \ref{prop:identify} and Proposition \ref{prop:e2 stability} below.

Recall that we call a group $\pi$ rationally perfect if $\HH^1(\pi;V) = 0$ for all finite dimensional $\QQ$-vector spaces $V$ with a $\pi$-action (cf.~Definition \ref{def:rationally perfect appendix}).

\begin{proposition} \label{prop:rp}
The group $\pi_1(X_g)$ is rationally perfect for $g\geq 2$.
\end{proposition}

\begin{proof}
By Proposition \ref{prop:mcg}, there is a short exact sequence of groups,
$$1\rightarrow K_g \rightarrow \pi_1(X_g)\rightarrow \Gamma_g \rightarrow 1,$$
where the kernel $K_g$ is finite, whence rationally perfect. The group $\Gamma_g$ is an arithmetic subgroup of the algebraic group $\Sp_g$ or $O_{g,g}$, depending on whether $d$ is odd or even. In either case, the algebraic group is almost simple and its $\QQ$-rank is $g$. Hence, it follows from Theorem \ref{thm:arithmetic groups} that $\Gamma_g$ is rationally perfect. An application of the Hochschild-Serre spectral sequence then shows that $\pi_1(X_g)$ is rationally perfect.
\end{proof}

We note the following consequence for future reference.
\begin{proposition} \label{prop:uceq}
There is a $\pi_1(X_g)$-equivariant isomorphism
$$\HH_*(\cX_g;\QQ) \cong \HH_*^{CE}(\gl_g),$$
compatible with the stabilization maps.
\end{proposition}

\begin{proof}
Combine Proposition \ref{prop:equivariant}, Proposition \ref{prop:rp} and Proposition \ref{prop:perfect}. Compatibility with the stabilization maps follows from Proposition \ref{prop:stabilization} and naturality of the Quillen spectral sequence.
\end{proof}

\begin{proposition} \label{prop:identify}
For $d\geq 2$, $g\geq 2$, and all $p,q$, there is an isomorphism
$$\HH_p(\pi_1(X_g);\HH_q(\cX_g;\QQ)) \cong \HH_p(\Gamma_g;\HH_q^{CE}(\gl_g)),$$
compatible with the stabilization maps.
\end{proposition}

\begin{proof}
The previous proposition implies that
$$\HH_p(\pi_1(X_g);\HH_q(\cX_g;\QQ)) \cong \HH_p(\pi_1(X_g);\HH_q^{CE}(\gl_g)).$$
By Proposition \ref{prop:mcg}, the kernel of the homomorphism $\pi_1(X_g)\rightarrow \Gamma_g$ is a finite group that acts trivially on $\gl_g$ and hence on $\HH_q^{CE}(\gl_g)$. Since we work with rational coefficients, this implies that there is an isomorphism
$$\HH_p(\pi_1(X_g);\HH_q^{CE}(\gl_g)) \cong \HH_p(\Gamma_g;\HH_q^{CE}(\gl_g)),$$
as claimed.
\end{proof}

\begin{proposition} \label{prop:e2 stability}
Let $d\geq 2$. The stabilization map
$$\HH_p(\Gamma_g;\HH_q^{CE}(\gl_g)) \rightarrow \HH_p(\Gamma_{g+1},\HH_q^{CE}(\gl_{g+1}))$$
is an isomorphism for $g> 2p + 2q + 4$ and a surjection for $g= 2p +2q +4$.
\end{proposition}

\begin{proof}
As noted above, the group $\Gamma_g$ is rationally perfect for $g\geq 2$. The chain complex of $\QQ[\Gamma_g]$-modules $C_*^{CE}(\gl_g)$ is finite dimensional over $\QQ$ in each degree, and is therefore split by Proposition \ref{prop:split}. By Lemma \ref{lemma:split} we get a homotopy commutative diagram of chain complexes of $\QQ[\Gamma_g]$-modules
$$\xymatrix{C_*^{CE}(\gl_g) \ar[r]^-\sigma \ar[d]_-\simeq & C_*^{CE}(\gl_{g+1}) \ar[d]^-\simeq \\ \HH_*^{CE}(\gl_g) \ar[r]^-\sigma & \HH_*^{CE}(\gl_{g+1}) }$$
where the vertical maps are chain homotopy equivalences. This implies that the diagram
$$\xymatrix{\HHH_k(\Gamma_g;C_*^{CE}(\gl_g)) \ar[d]_-\cong \ar[r]^-\sigma & \HHH_k(\Gamma_{g+1};C_*^{CE}(\gl_{g+1})) \ar[d]^-\cong \\ \HHH_k(\Gamma_g;\HH_*^{CE}(\gl_g)) \ar[r]^-\sigma & \HHH_k(\Gamma_{g+1};\HH_*^{CE}(\gl_{g+1}))}$$
is commutative, and that the vertical maps are isomorphisms. By Theorem \ref{thm:hh stability}, the top map is an isomorphism for $g> 2k+4$ and a surjection for $g\geq 2k+4$, so the same is true of the bottom map. For any group $\Gamma$ and any graded $\Gamma$-module $H_*$, regarded as a chain complex with zero differential, there is a decomposition of hyperhomology,
$$\HHH_k(\Gamma;H_*) \cong \bigoplus_{p+q = k} \HH_p(\Gamma;H_q),$$
which is natural in $\Gamma$ and $H_*$. It follows that the constituents of the bottom map,
$$\sigma_{p,q}\colon \HH_p(\Gamma_g;\HH_q^{CE}(\gl_g)) \rightarrow \HH_p(\Gamma_{g+1};\HH_q^{CE}(\gl_{g+1})),$$
are isomorphisms for $g>2p+2q+4$ and surjections for $g=2p+2q+4$.
\end{proof}

This completes the proof of Theorem \ref{thm:baut stability}.

\subsection{Homological stability for block diffeomorphisms} \label{sec:block}
The goal of this section is to prove the following theorem.
\begin{theorem} \label{thm:bdiff stability}
Let $d\geq 3$. The stabilization map
$$\sigma_*\colon \HH_k(B\tDiff_\partial(M_{g,1});\QQ) \rightarrow \HH_k(B\tDiff_\partial(M_{g+1,1});\QQ)$$
is an isomorphism for $g>2k+4$ and a surjection for $g = 2k+4$.
\end{theorem}

The proof of the theorem is based on an analysis of the diagram \eqref{eq:(23)} in \S\ref{sec:fhf} for $M = M_{g,1}$. We will use the following abbreviations.
\begin{align*}
X_g & = B\aut_\partial(M_{g,1}), & X_{g,J} & = B\aut_{\partial,J}(M_{g,1}), & X_{g,\circ} & = B\aut_{\partial,\circ}(M_{g,1}), \\
Y_g & = B\tDiff_\partial(M_{g,1}), & Y_{g,\circ} & = B\tDiff_{\partial,\circ}(M_{g,1}), & F_g   = & \aut_{\partial,\circ}(M_{g,1})/\tDiff_{\partial,\circ}(M_{g,1}).
\end{align*}
For $M=M_{g,1}$, the diagram \eqref{eq:(23)} can then be rewritten as
\begin{equation} \label{eq:J-pullback}
\xymatrix{F_g \ar[d] \ar[r] & {*} \ar[d] \\ Y_{g,\circ} \ar[d] \ar[r] & X_{g,\circ} \ar[d] \ar[r] & {*} \ar[d] \\
Y_g \ar[r] & X_{g,J} \ar[r] & B\pi_1(X_{g,J}),
}
\end{equation}
where each square is a homotopy pullback. We let $\widehat{\Gamma}_g = \pi_1(X_{g,J})$. Recall that, by construction, it is the image of $J\colon \tDiff_\partial(M_{g,1})\to \pi_0\aut_\partial(M_{g,1})$. The fundamental group of $Y_{g,\circ}$ is the kernel of $J$.
We will study the spectral sequences of the homotopy fiber sequences
\begin{equation} \label{eq:fib1}
F_g \to Y_{g,\circ} \to X_{g,\circ},
\end{equation}
\begin{equation} \label{eq:fib2}
Y_{g,\circ} \to Y_g \to B\pi_1(X_{g,J}).
\end{equation}

First, we need a result on the fundamental group of $X_{g,J}$.

\begin{proposition} \label{prop:X_g^J}
The kernel of $\pi_1(X_{g,J}) \to \Gamma_g$ is finite. The group $\pi_1(X_{g,J})$ is rationally perfect for $g\geq 2$.
\end{proposition}

\begin{proof}
From \S\ref{sec:mcg} we have a commutative diagram of groups with exact rows,
$$
\xymatrix{
1 \ar[r] & \widetilde{K}_g \ar[d] \ar[r] & \pi_1 Y_g \ar[d]^-J \ar[r] & \Gamma_g \ar@{=}[d] \ar[r] & 1 \\
1 \ar[r] & K_g \ar[r] & \pi_1 X_g \ar[r] & \Gamma_g \ar[r] & 1,
}
$$
By construction, $\pi_1(X_{g,J}) = \im J$. It follows that there is an exact sequence
$$1\to K_{g,J} \to \pi_1(X_{g,J}) \to \Gamma_g \to 1,$$
where $K_{g,J}$ injects into $K_g$. Since $K_g$ is finite, so is $K_{g,J}$. The rest of the proof is similar to that of Proposition \ref{prop:rp}.
\end{proof}

\begin{proposition} \label{prop:splitting}
The fibration
$$F_g \to Y_{g,\circ} \to X_{g,\circ}$$
is rationally totally nonhomologous to zero, i.e., there is an isomorphism
\begin{equation} \label{eq:tensor}
\HH_*(Y_{g,\circ};\QQ)\cong \HH_*(X_{g,\circ};\QQ) \tensor \HH_*(F_g;\QQ).
\end{equation}
Moreover, the isomorphism may be taken to be $\widehat{\Gamma}_g$-equivariant.
\end{proposition}

\begin{proof}
The fibration at hand is the upper row in \eqref{eq:b} for $M = M_{g,1}$.
The vertical maps in \eqref{eq:b} are rational homology isomorphisms, and the lower fibration is rationally trivial by Corollary \ref{cor:triviality}, since the reduced homology of $M_{g,1}$ is concentrated in a single degree. This proves the first claim.

The group $\widehat{\Gamma}_g = \pi_1(X_{g,J}) = \pi_0\aut_{\partial,J}(M_{g,1})$ acts on the spaces in the fibration sequence \eqref{eq:fib1} in the sense that
there are homomorphisms $\widehat{\Gamma}_g \to \pi_0\aut(Z)$ for $Z\in\{F_g,Y_{g,\circ},X_{g,\circ}\}$ and the maps in \eqref{eq:fib1} preserve this (homotopy) action. The actions can be seen as holonomy actions of the various fibrations in \eqref{eq:J-pullback}, but it is better to go back to the definitions in \S\ref{sec:surgery} of the involved spaces. Elements of $\widehat{\Gamma}_g$ are represented by diffeomorphisms $\varphi\colon M_{g,1}\to M_{g,1}$ with $\partial \varphi = id$. The action on the $\Delta$-monoids $\tDiff_{\partial,\circ}(M_{g,1})_\bullet$ and $\taut_{\partial,\circ}(M_{g,1})_\bullet$ is given by conjugating a $k$-simplex $f\colon \Delta^k\times M_{g,1} \to \Delta^k\times M_{g,1}$ with $\Delta^k\times \varphi$. The induced action on $\aut_{\partial,\circ}(M_{g,1})/\tDiff_{\partial,\circ}(M_{g,1})$ induces an action on the structure $\Delta$-set compatible with the homotopy equivalence
$$F_g = \aut_{\partial,\circ}(M_{g,1})/\tDiff_{\partial,\circ}(M_{g,1}) \xrightarrow{\sim} \big(\Ss_\partial^{G/O}(M_{g,1})_\bullet\big)_{(1)}.$$
An element $[\varphi]\in \widehat{\Gamma}_g$ acts on a $k$-simplex $W\xrightarrow{f} \Delta^k\times M$ of the structure $\Delta$-set by the composition $(\Delta^k \times \varphi) \circ f$. Use of geometric realization and the classifying space construction yields the required $\widehat{\Gamma}_g$-action on the relevant spaces. Since $\widehat{\Gamma}_g$ is rationally perfect, the isomorphism \eqref{eq:tensor} can be taken to be $\widehat{\Gamma}_g$-equivariant. 
\end{proof}

Proposition \ref{prop:uceq} identifies $\HH_*(X_{g,\circ};\QQ) \cong \HH_*^{CE}(\gl_g)$ as $\widehat{\Gamma}_g$-modules.
Next, we describe the rational homology of $F_g$ as a $\widehat{\Gamma}_g$-module. Recall that $H_g^\QQ$ denotes $\HH_d(M_{g,1};\QQ)$.

\begin{proposition} \label{prop:F_g}
There is an isomorphism of $\widehat{\Gamma}_g$-modules
$$\HH_*(F_g;\QQ) \cong \Lambda(\Pi\tensor H_g^\QQ),$$
where
$$\Pi = \QQ \set{\pi_i}{|\pi_i| = 4i-d >0} \quad \big(= (s^{-d}\pi_*(G/O)\tensor \QQ)_{> 0} \big),$$
and $\widehat{\Gamma}_g$ acts on the right-hand side via the standard action of $\Gamma_g$ on $H_g^\QQ$.
\end{proposition}

\begin{proof}
There is a $\widehat{\Gamma}_g$-action on the $\Delta$-set of normal invariants, where a diffeomorphism $\varphi\colon M_{g,1} \to M_{g,1}$ with $\partial \varphi = id$ acts by composing a $k$-simplex of $\mathcal{N}_\partial^{G/O}(M_{g,1})_\bullet$ with
$$
\xymatrix{\zeta^K \ar[r]^-{\varphi_*} \ar[d] & (\Delta^k \times \varphi^{-1})(\zeta^K) \ar[d] \\
\Delta^k \times M_{g,1} \ar[r]^-\varphi & \Delta^k \times M_{g,1},}
$$
and the normal invariant
$$
\eta_\bullet\colon \Ss_\partial^{G/O}(M_{g,1})_\bullet \to \mathcal{N}_\partial^{G/O}(M_{g,1})_\bullet
$$
becomes equivariant.

It follows from Lemma 3.3 of \cite{BM} that the homotopy equivalence
$$\mathcal{N}_\partial^{G/O}(M_{g,1})_\bullet\xrightarrow{\sim} S_\bullet \map_*(M_{g,1}/\partial M_{g,1},G/O)$$
of Theorem \ref{thm:kan-eq} is $\widehat{\Gamma}_g$-equivariant on homotopy groups, where $[\varphi]\in \widehat{\Gamma}_g$ acts on the target via the diffeomorphism $\varphi^{-1}\colon M_{g,1}\to M_{g,1}$.

Since $\partial M_{g,1} \to M_{g,1}$ is a sum of Whitehead products, its suspension is homotopically trivial and
$$\map_*(M_{g,1}/\partial M_{g,1} ,G/O) \simeq \map_*(M_{g,1},G/O) \times \Omega^{2d} G/O,$$
because $G/O\simeq \Omega B (G/O)$. The action on $\Omega^{2d} G/O$ is trivial, so all in all
$$
q^*\circ j_* \circ \eta \colon \Ss_\partial^{G/O}(M_{g,1})_{(1)} \to \map_*(M_{g,1},BO)_{(0)}
$$
is $\widehat{\Gamma}_g$-equivariant on homotopy groups with the action on the target induced by $\varphi^{-1}\colon M_{g,1}\to M_{g,1}$.

We have the $\widehat{\Gamma}_g$-isomorphism
$$\pi_*(\map_*(M_{g,1},BO)_{(0)})\tensor \QQ \cong \HH_d(M_{g,1};\QQ)\tensor \Pi$$
and since $\map_*(M_{g,1},BO)_{(0)}$ is an infinite loop space it has trivial rational $k$-invariants, so
$$\HH_*(\map_*(M_{g,1},BO)_{(0)};\QQ) \cong \Lambda(H_g^\QQ\tensor\Pi)$$
as $\widehat{\Gamma}_g$-modules.

Since $F_g\to\Ss_\partial^{G/O}(M_{g,1})_{(1)}$ is a $\widehat{\Gamma}_g$-equivariant homotopy equivalence, and the composite
$$
F_g \xrightarrow{\sim} \Ss_\partial^{G/O}(M_{g,1})_{(1)} \xrightarrow{q^*\circ j_*\circ \eta}  \map_*(M_{g,1},BO)_{(0)}
$$
is a rational homotopy equivalence by Proposition \ref{prop:fiber proposition}, which is $\widehat{\Gamma}_g$-equivariant on homotopy groups, it follows that $F_g$ also has trivial rational $k$-invariants and that we have isomorphisms of $\widehat{\Gamma}_g$-modules $\pi_*(F_g)\tensor\QQ \cong \Pi \tensor H_g^\QQ$ and $\HH_*(F_g;\QQ) \cong \Lambda (\Pi \tensor H_g^\QQ)$ as claimed.
\end{proof}

The previous proposition may be interpreted as asserting an isomorphism of (abelian) graded Lie algebras,
$\al_g = \pi_{*+1}(F_g) \tensor \QQ \cong s^{-1} \Pi\tensor H_g^\QQ.$
The homology is then given by
$$\HH_*(F_g;\QQ) = \HH_*^{CE}(\al_g).$$
We may write $C_*^{CE}(\al_g)$ as the value at $H_g$ of a Schur functor as follows:
\begin{lemma} \label{lemma:polynomial2}
There is an isomorphism of $\pi_1(X_{g,J})$-modules
$$C_*^{CE}(\al_g) \cong \bigoplus_{k\geq 0} \DD(k)\tensor_{\Sigma_k} H_g^{\tensor k}$$
compatible with the stabilization maps, where $\DD(k)$ is the $\Sigma_k$-module
$$\DD(k) = \Pi^{\tensor k}.$$
In particular, $\DD(k)$ is concentrated in degrees $\geq k$ for every $k$.\hfill $\square$
\end{lemma}

By Proposition \ref{prop:uceq} and Proposition \ref{prop:F_g} we may rewrite the right-hand side of \eqref{eq:tensor} in terms of Chevalley-Eilenberg homology:
$$\HH_*(X_{g,\circ};\QQ)\tensor \HH_*(F_g;\QQ) \cong \HH_*^{CE}(\gl_g)\tensor \HH_*^{CE}(\al_g) \cong \HH_*^{CE}(\gl_g\oplus \al_g).$$
In particular, the action of $\widehat{\Gamma}_g = \pi_1(X_{g,J})$ on $\HH^*(Y_{g,\circ};\QQ)$ factors over $\Gamma_g$, since this is true of the right-hand side. Since the kernel of $\pi_1(X_{g,J}) \to \Gamma_g$ is finite, we may then write the $E_2$-term of the spectral sequence of \eqref{eq:fib2} as follows:
$$E_{p,q}^2 = \HH_p(\pi_1(X_{g,J});\HH_q(Y_{g,\circ};\QQ)) \cong \HH_p(\Gamma_g;\HH_q^{CE}(\gl_g\oplus \al_g)).$$
Hence, the proof of Theorem \ref{thm:bdiff stability} will be complete once we verify the following proposition.

\begin{proposition}
Let $d\geq 3$. The stabilization map
$$\HH_p(\tGamma_g;\HH_q^{CE}(\gl_g\oplus \al_g)) \rightarrow \HH_p(\tGamma_{g+1};\HH_q^{CE}(\gl_{g+1}\oplus \al_{g+1}))$$
is an isomorphism for $g>2p+2q+4$ and a surjection for $g\geq 2p+2q +4$.
\end{proposition}

\begin{proof}
The proof proceeds exactly as the proof of Proposition \ref{prop:e2 stability}, after noting that $C_*^{CE}(\gl_g \oplus\al_g)$ is a split complex of $\QQ[\tGamma_g]$-modules, whose module of $q$-chains is the value of a polynomial functor of degree $2q$ on the standard $\QQ[\tGamma_g]$-module $H_g^\QQ$. More precisely, by combining Lemma \ref{lemma:polynomial functor} and Lemma \ref{lemma:polynomial2} we have that
$$C_*^{CE}(\gl_g \oplus \al_g) = C_*^{CE}(\gl_g) \tensor C_*^{CE}(\al_g) \cong \bigoplus_{k\geq 0} (\CC\tensor \DD)(k) \tensor_{\Sigma_k} H_g^{\tensor k},$$
where
$$(\CC\tensor \DD)(k) = \bigoplus_{i+j=k} \Ind_{\Sigma_i\times \Sigma_j}^{\Sigma_k} \CC(i)\tensor \DD(j).$$
Since $\CC(i)$ is concentrated in degrees $\geq \frac{id}3$ and $\DD(j)$ is concentrated in degrees $\geq j$, it follows that
$(\CC\tensor \DD)(k)$ is concentrated in degrees $\geq k/2$. This implies that the functor $(\CC\tensor \DD)_q(-)$ is polynomial of degree $\leq 2q$ for every $q$.
\end{proof}

\section{Stable cohomology}
The goal of this section is to calculate the stable rational cohomology of the classifying spaces $B\aut_\partial(M_{g,1})$ and $B\tDiff_\partial(M_{g,1})$.
We begin by reviewing the calculation of the stable cohomology of $B\Diff_\partial(M_{g,1})$ in terms of $\kappa$-classes, due to Galatius and Randal-Williams, and Borel's results on the stable cohomology of arithmetic groups.

\subsection{$\kappa$-classes and the stable cohomology of the diffeomorphism group} \label{sec:diff}
Let $M$ be a closed oriented $2d$-dimensional manifold and let $\Diff(M)$ be the group of orientation preserving diffeomorphisms of $M$. The space $B\Diff(M)$ is a classifying space for smooth oriented fiber bundles with fiber diffeomorphic to $M$, or `$M$-bundles' for short. To every characteristic class of oriented vector bundles
$$c\in \HH^k(BSO(2d))$$
there is an associated characteristic class of $M$-bundles
$$\kappa_c\in \HH^{k-2d}(B\Diff(M)),$$
characterized by the following. Given a smooth oriented fiber bundle
$$M\rightarrow E\stackrel{\pi}{\rightarrow} X,$$
the vertical tangent bundle $T_\pi E$ is an oriented $2d$-dimensional vector bundle over $E$, and we may consider its characteristic class $c(T_\pi E)\in \HH^k(E)$. By applying the Gysin homomorphism $\pi_!\colon \HH^k(E) \rightarrow \HH^{k-2d}(X)$, we obtain a class
$$\pi_!(c(T_\pi E))\in \HH^{k-2d}(X).$$
By definition, $\kappa_c(\pi) = \pi_!(c(T_\pi E))$.
Recall that the rational cohomology of $BSO(2d)$ is a polynomial ring
$$\HH^*(BSO(2d);\QQ) = \QQ[p_1,\ldots,p_{d-1},e]$$
in the Pontryagin classes $p_i$ and the Euler class $e$.

For $M = M_g$, the pullback of $\kappa_c$ along the map $B\Diff_\partial(M_{g,1}) \rightarrow B\Diff(M_g)$ gives us a class in $\HH^{k-2d}(B\Diff_\partial(M_{g,1}))$ that we will also denote $\kappa_c$. The stabilization map $B\Diff_\partial(M_{g,1}) \to B\Diff_\partial(M_{g+1,1})$ induces an isomorphism on $\HH_k(-)$ for $g>2k+4$ \cite{GRW2}. The stable cohomology is given by the following.

\begin{theorem}[Madsen-Weiss $2d=2$ \cite{MW}, Galatius-Randal-Williams $2d>4$ \cite{GRW1}] \label{thm:GRW}
For $2d\ne 4$, the stable cohomology of the diffeomorphism group of $M_{g,1}$ is given by
$$\HH^*(B\Diff_\partial(M_{\infty,1});\QQ) \cong \QQ[\kappa_c|c\in B],$$
where $B$ is the set of monomials $c$ in the Pontryagin classes $p_{d-1},p_{d-2},\ldots,p_{\lceil\frac{d+1}{4}\rceil}$ and the Euler class $e$, of total degree $|c|>2d$.
\end{theorem}

\subsection{Borel's calculation of the stable cohomology of arithmetic groups}
The rational cohomology of $BU$ is a polynomial algebra in the Chern classes
$$\HH^*(BU;\QQ) = \QQ[c_1,c_2,\ldots].$$
The Hopf algebra structure is given by $\Delta(c_n) = \sum_{p+q=n} c_p\tensor c_q$. Let $\sigma_1,\ldots, \sigma_n$ denote the elementary symmetric polynomials in the indeterminates $t_1,\ldots,t_n$. Then there is a unique polynomial $p_n$ such that
$$p_n(\sigma_1,\ldots,\sigma_n) = t_1^n + \cdots + t_n^n,$$
Define the `Newton classes' by
$$s_n = p_n(c_1,\ldots,c_n)\in \HH^{2n}(BU;\QQ).$$
These are primitive generators for the rational cohomology of $BU$.

According to Borel \cite{Borel1}, the rational cohomology of the infinite symplectic group $\Sp(\ZZ)$ is the primitively generated Hopf algebra
$$\HH^*(B\Sp(\ZZ);\QQ) \cong \QQ[x_1,x_2,\ldots].$$
The primitive generators $x_i$ are of degree $4i-2$, and may be chosen to be the pullbacks of the odd classes $s_{2i-1}\in \HH^{4i-2}(BU;\QQ)$ along the map
$$B\Sp(\ZZ)\rightarrow B\Sp(\RR) \stackrel{\sim}{\leftarrow} BU.$$

\subsection{Relation between Borel classes and $\kappa$-classes}
There is another way of producing characteristic classes of smooth fiber bundles, following Atiyah \cite[\S4]{Atiyah}. Again, let $M$ be a smooth oriented $2d$-dimensional manifold. Assume that $d$ is odd. Then $\HH^d(M;\RR)$ is of even dimension, say $2g$.

Let $E\stackrel{\pi}{\rightarrow} X$ be an $M$-bundle. There is a real $2g$-dimensional vector bundle $\xi$ over $X$ (the Hodge bundle), with structure group $\Sp_{2g}(\RR)$, whose fiber over $x$ is the cohomology group
$$\xi_x = \HH^d(\pi^{-1}(x);\RR).$$
The structure group can be reduced to the maximal compact subgroup $U(g)\subset \Sp_{2g}(\RR)$, so we obtain a $g$-dimensional complex vector bundle $\eta$ over $X$. We may consider the `Newton classes'
$$s_i(\eta)\in \HH^{2i}(X).$$
The even classes vanish, $s_{2i}(\eta) = 0$. The odd classes agree with the pullbacks of the Borel classes
\begin{equation} \label{eq:Borel classes}
s_{2i-1}(\eta) = \rho^*(x_i)
\end{equation}
along the map $\rho\colon X\rightarrow B\Sp_{2g}(\ZZ)$.

Now, one may ask if there are any relations between the classes $\kappa_c(\pi)$ and $s_i(\eta)$. This problem was addressed and solved by Morita \cite[\S2]{Morita} in the case of surface bundles. A similar treatment is possible in our situation. According to \cite[(4.3)]{Atiyah}, we have the relation
\begin{equation} \label{eq:index}
ch(\eta^* - \eta) = \pi_!(\widetilde{L}(T_\pi E))
\end{equation}
in the cohomology of $X$. In the left hand side, $\eta^*$ denotes the conjugate bundle, the formal difference $\eta^* - \eta$ is taken in $K(X)$, and $ch$ is the Chern character $ch\colon K(X)\rightarrow \HH^*(X;\QQ)$,
$$ch(\eta) = g + \sum_{k\geq 1} \frac{s_k(\eta)}{k!}.$$
Since $s_k(\eta^*) = (-1)^k s_k(\eta)$, we may write the left hand side of \eqref{eq:index} as
$$ch(\eta^* - \eta) = -2\sum_{\text{$k$ odd}} \frac{s_k(\eta)}{k!}.$$
Turning to the right hand side, if $\xi$ is a real vector bundle over $E$ of dimension $2d$, then
$$\widetilde{L}(\xi) = \widetilde{L}(p_1,\ldots,p_d)$$
is the formal power series in the Pontryagin classes of $\xi$ determined by
$$\widetilde{L}(\sigma_1,\ldots,\sigma_d) = f(t_1)\cdots f(t_d),$$
where $\sigma_i$ is the elementary symmetric polynomial in $t_1^2,\ldots,t_d^2$ of degree $i$, and $f(t)$ is the formal power series
\begin{equation*}
f(t) = \frac{t}{\tanh(t/2)} = 2\left(1+ \sum_{k\geq 1} (-1)^{k-1} \frac{B_k}{(2k)!} t^{2k}\right).
\end{equation*}
Here $B_k$ are the Bernoulli numbers. Explicitly, the homogeneous term in $\widetilde{L}(\xi)$ of degree $n$ is given by
$$\widetilde{L}_n = \widetilde{L}_n(p_1,\ldots,p_n) = 2^d \sum_{I\vdash n} \lambda_I s_I(p_1,\ldots,p_n).$$
Here, the sum is over all partitions $I = (i_1,\ldots,i_r)$ of $n$, and $s_I$ denotes the corresponding polynomial in the elementary symmetric polynomials (see e.g.~\cite[p.188]{MS}). The coefficients are $\lambda_I = \lambda_{i_1} \cdots \lambda_{i_r}$ where
$$\lambda_k = (-1)^{k-1} \frac{B_k}{(2k)!}.$$
We are assuming that $d$ is odd, say $d = 2s+1$. By comparing homogeneous terms in \eqref{eq:index} and using \eqref{eq:Borel classes}, we obtain the relation
\begin{equation}
-2\frac{\rho^*(x_i)}{(2i-1)!} = \kappa_{\widetilde{L}_{i+s}} \in \HH^{4i-2}(X;\QQ)
\end{equation}
for every $i$. It is easily seen that the class $\widetilde{L}_{i+s}$ is proportional to the Hirzebruch $L$-class $L_{i+s}$ (cf.~\cite[p.224]{MS}). In fact,  $2^{2i-1}\widetilde{L}_{i+s} = L_{i+s}$.
Passing to the universal bundle, we may conclude that the map 
\begin{equation} \label{eq:map}
\HH^*(B\Sp(\ZZ);\QQ) \rightarrow \HH^*(B\Diff_\partial(M_{\infty,1});\QQ)
\end{equation}
sends the Borel class $x_i$ to $-\frac{(2i-1)!}{2^{2i}} \kappa_{L_{i+s}}$ for every $i\geq 1$.

By Theorem \ref{thm:GRW}, the set $\set{\kappa_c}{c\in B}$ generates the cohomology freely, where $B$ is the set of monomials $c$ in the Pontryagin classes $p_{d-1},p_{d-2},\ldots,p_{\lceil\frac{d+1}{4}\rceil}$ and the Euler class $e$ of total degree $|c|>2d$. In particular, these classes are linearly independent. By \cite{BB}, every monomial $c$ in the Pontryagin classes of degree $4i+4s$ appears with a non-zero coefficient in $L_{i+s}$. It is easily seen that $L_{i+s}$ will contain such monomials $c$ belonging to $B$: for example, if we write $i = qs+r$ where $q,r$ are non-negative integers with $r < s$, then $c= p_s^q p_{s+r}$ belongs to $B$ and appears with a non-zero coefficient in $L_{i+s}$. It follows that $\kappa_{L_{i+s}}$ is non-zero in $\HH^*(B\Diff_\partial(M_{\infty,1});\QQ)$.
Thus, we have proved the following theorem, for $d\geq 3$ odd.
\begin{theorem} \label{thm:diff injective}
Let $d\geq 3$. The map $B\Diff_\partial(M_{\infty,1}) \rightarrow B\Gamma_\infty$ is injective on indecomposables in rational cohomology.
\end{theorem}

The argument in the case when $d$ is even is similar. The symplectic group $\Sp_{2g}(\RR)$ is replaced by the orthogonal group $O_{g,g}(\RR)$, and the complex vector bundle $\eta$ is replaced by two real vector bundles $W^+$ and $W^-$, as in \cite{Atiyah}. According to \cite{Borel1}, the stable cohomology of the arithmetic group $O_{g,g}(\ZZ)$ is a polynomial algebra in generators $x_i$ of degree $4i$, for $i=1,2,\ldots$. The class $x_i$ may be chosen as the pullback of $ph_i(W^+)-ph_i(W^-)$ under the map
$$BO_{g,g}(\ZZ) \rightarrow BO_{g,g}(\RR) \rightarrow BO(g)\times BO(g),$$
where the last map is a homotopy inverse of the map $BO(g)\times BO(g) \rightarrow BO_{g,g}$ coming from $O(g)\times O(g)$ being the maximal compact subgroup of $O_{g,g}$. Then \cite[(4.2)]{Atiyah} expresses the Pontryagin character $ph(W^+-W^-)$ (i.e., the Chern character of the complexification) in terms of $\kappa$-classes. As before, this can be used to show that the classes $x_i$ are mapped to non-zero linear combinations of $\kappa$-classes.

\subsection{The stable cohomology of homotopy automorphisms}
Let
$$X_g = B\aut_\partial(M_{g,1}),$$
and consider the homotopy colimit
$$X_\infty = \hocolim(X_1 \rightarrow X_2 \rightarrow X_3 \rightarrow \cdots),$$
taken over the stabilization maps described in \S\ref{sec:stabilization maps}. By Theorem \ref{thm:baut stability}, the canonical map $X_g\rightarrow X_\infty$ induces an isomorphism in rational cohomology
$$\HH^k(X_\infty;\QQ)\rightarrow \HH^k(X_g;\QQ)$$
for $g>2k+4$. The goal of this section is to prove the following theorem.

\begin{theorem} \label{thm:main}
Let $d\geq 3$. There is an isomorphism of graded rings
\begin{equation} \label{eq:sc}
\HH^*(X_{\infty};\QQ) \cong \HH^*(\Gamma_\infty;\QQ) \tensor \HH_{CE}^*(\gl_\infty)^{\Gamma_\infty}.
\end{equation}
\end{theorem}

In the above theorem, $\Gamma_\infty = \colim_g \Gamma_g$ and $\gl_\infty = \colim_g \gl_g$. Recall that $\Gamma_g$ denotes the automorphism group of the quadratic module $(H_g,\mu,q)$ and that $\gl_g$ denotes the graded Lie algebra $\Der_{\omega_g}^+\LL(V_g)$, where $V_g = s^{d-1}H_g \tensor \QQ$.

As discussed earlier, Borel's results \cite{Borel1} yield a computation of the left factor $\HH^*(\Gamma_\infty;\QQ)$. In Section \ref{sec:graph}, we will show how to express the right factor $\HH_{CE}^*(\gl_\infty)^{\Gamma_\infty}$ in terms of graph homology.

The proof of Theorem \ref{thm:main} has several ingredients. Homological stability together with Borel's vanishing result (as manifested in Theorem \ref{thm:borel vanishing}) will allow us to conclude that the universal cover spectral sequence for $X_\infty$ satisfies $E_2^{p,q} \cong E_2^{p,0}\tensor E_2^{0,q}$. Then we will prove that the spectral sequence collapses at $E_2$. We do this by showing that the rational cohomology ring of $X_\infty$ is free and that the map $X_\infty \rightarrow B\Gamma_\infty$ is injective on indecomposables in rational cohomology. Recall that $\widetilde{X}_g = X_{g,\circ} = B\aut_{\partial,\circ}(M_{g,1})$.

\begin{theorem} \label{thm:separate}
The natural map
$$\HH^p(\Gamma_g;\QQ)\tensor \HH_{CE}^q(\gl_g)^{\Gamma_g}\rightarrow \HH^p(\pi_1(X_g);\HH^q(X_{g,\circ};\QQ)),$$
is an isomorphism in the stable range $g>2p+2q+4$. 
\end{theorem}

\begin{proof}
As in Proposition \ref{prop:identify}, there is an isomorphism
\begin{equation} \label{eq:pi-gamma}
\HH^p(\pi_1(X_g);\HH^q(X_{g,\circ};\QQ)) \cong \HH^p(\Gamma_g;\HH_{CE}^q(\gl_g)),
\end{equation}
compatible with the stabilization maps. Since $\Gamma_g$ is rationally perfect, the functor $\HH^p(\Gamma_g;-)$ is exact on the category of finite dimensional $\QQ[\Gamma_g]$-modules. In particular, since the Chevalley-Eilenberg complex $C_{CE}^*(\gl_g)$ is finite dimensional in each degree, we may identify the right hand side of \eqref{eq:pi-gamma} with the $q$-th cohomology of the cochain complex $\HH^p(\Gamma_g;C_{CE}^*(\gl_g))$. By Lemma \ref{lemma:polynomial functor} we may identify $C_{CE}^q(\gl_g)$ with the value at $H_g^\QQ$ of a polynomial functor of degree $\leq 2q$. Hence, by Theorem \ref{thm:borel vanishing} the natural map
$$\HH^p(\Gamma_g;\QQ)\tensor C_{CE}^q(\gl_g)^{\Gamma_g}\rightarrow \HH^p(\Gamma_g;C_{CE}^q(\gl_g))$$
is an isomorphism for $g>2p+2q+4$. The claim follows by passing to cohomology in the $q$ direction and using \eqref{eq:pi-gamma}.
\end{proof}

\begin{theorem} \label{thm:loop space}
The cohomology algebra $\HH^*(X_\infty;\QQ)$ is free graded commutative with finitely many generators in each degree.
\end{theorem}

\begin{proof}
It follows from the homological stability theorem that the natural map $\HH^k(X_\infty;\QQ) \rightarrow \HH^k(X_g;\QQ)$ is an isomorphism for $g>2k+4$. The latter group is finite dimensional by Theorem \ref{thm:finite type}. This proves the claim about finite type. To show that the cohomology algebra is free, we employ an argument similar to that of \cite{Miller}. Let $\mathcal{D}_{2d}$ denote the little $2d$-disks operad. There are maps
\begin{equation} \label{eq:operad}
\mathcal{D}_{2d}(r)\times X_{g_1} \times \cdots \times X_{g_r}\rightarrow X_{g},\quad g = g_1+\cdots +g_r.
\end{equation}
Indeed, given a configuration of $r$ little $2d$-disks in a fixed $2d$-disk, we may remove their interiors and glue in the manifolds $M_{g_1,1},\ldots,M_{g_r,1}$ in their place. The result is homeomorphic to $M_{g,1}$. Given homotopy automorphisms of $M_{g_i,1}$ that restrict to the identity on the boundary, we can extend them to a homotopy automorphism of $M_{g,1}$ by letting it be the identity outside the interiors of the removed disks. This construction respects compositions, so it passes to classifying spaces, giving \eqref{eq:operad}. The maps \eqref{eq:operad} endow the disjoint union
$$X = \coprod_{g\geq 0} X_g$$
with the structure of an $E_{2d}$-space. By the recognition principle for iterated loop spaces (cf.~\cite{May1,May2}), the space $X$ admits a group completion $G X$ which is a $2d$-fold loop space, and it follows from the group completion theorem that there is a homology isomorphism $X_\infty\rightarrow GX_0$, where $GX_0$ denotes a connected component of $GX$ (cf.~\cite[\S3.2]{Adams}). Since $GX_0$ is a $2d$-fold loop space, its rational cohomology is a graded commutative and cocommutative Hopf algebra, so the same is true of the cohomology of $X_\infty$. By \cite{MM}, this implies that the cohomology algebra is free graded commutative.
\end{proof}

\begin{theorem} \label{thm:injective}
The map $X_\infty \rightarrow B\Gamma_\infty$ induces an injective homomorphism on indecomposables in rational cohomology.
\end{theorem}

\begin{proof}
The claim follows immediately from Theorem \ref{thm:diff injective}, which states that the composite map
$$B\Diff_\partial(M_{\infty,1}) \rightarrow B\aut_\partial(M_{\infty,1}) \rightarrow B\Gamma_\infty$$
induces an injective map on indecomposables in rational cohomology.
\end{proof}

\begin{proof}[Proof of Theorem \ref{thm:main}]
We have $\HH^1(X_g;\QQ) = \HH^1(B\pi_1(X_g);\QQ) = 0$ for $g\geq 2$, because the group $\pi_1(X_g)$ is rationally perfect (see Proposition \ref{prop:rp}). Let $(X_g)_\QQ^+$ and $B\pi_1(X_g)_\QQ^+$ denote the rational plus constructions (see Appendix C) and let $T_g$ be the homotopy fiber of the map $(X_g)_\QQ^+\rightarrow B\pi_1(X_g)_\QQ^+$. We obtain a map of fibrations,
$$
\xymatrix{
\uc{X}_g \ar[r] \ar[d] & X_g \ar[r] \ar[d] & B\pi_1(X_g) \ar[d] \\
T_g \ar[r] & (X_g)_\QQ^+ \ar[r] & B\pi_1(X_g)_\QQ^+,}
$$
and we may consider the induced map of cohomology spectral sequences $E\rightarrow \overline{E}$ with $\QQ$-coefficients;
$$E_2^{p,q} = \HH^p(B\pi_1(X_g)_\QQ^+;\HH^q(T_g)),\quad \overline{E}_2^{p,q} = \HH^p(B\pi_1(X_g);\HH^q(\uc{X}_g)).$$

By construction, the maps $E_2^{p,0} \rightarrow \overline{E}_2^{p,0}$ and $E_\infty^{p,q} \rightarrow \overline{E}_\infty^{p,q}$ are isomorphisms for all $p,q$. We have that $E_2^{p,q} \cong E_2^{p,0} \tensor E_2^{0,q}$ because the spaces involved are simply connected. In the spectral sequence $\overline{E}$, we have cohomology with twisted coefficients, but it follows from Theorem \ref{thm:separate} that $\overline{E}_2^{p,q} \cong \overline{E}_2^{p,0}\tensor \overline{E}_2^{0,q}$ for all $p,q$ in the stable range. The map $E\rightarrow \overline{E}$ respects these isomorphisms, because they may be realized by taking cup products. Thus, we are in position to apply Zeeman's comparison theorem for spectral sequences; we may conclude that
$$E_2^{0,q}\rightarrow \overline{E}_2^{0,q}$$
is an isomorphism for all $q$ in the stable range. There results an isomorphism of graded algebras
\begin{equation} \label{eq:fiber1}
\HH^*(T_\infty;\QQ) \cong \HH^*(\cX_\infty;\QQ)^{\Gamma_\infty}.
\end{equation}
The stable cohomology of $B\pi_1(X_g)_\QQ^+$ agrees with the stable rational cohomology of the group $\Gamma_g$, because $\pi_1(X_g)$ surjects onto $\Gamma_g$ with finite kernel. Borel's calculation of the stable cohomology of arithmetic groups \cite{Borel1} tells us that the cohomology ring $\HH^*(\Gamma_\infty;\QQ)$ is free graded commutative. This, together with Theorem \ref{thm:loop space} and Theorem \ref{thm:injective}, shows that the hypotheses of Lemma \ref{lemma:collapse} below are fulfilled, which yields an isomorphism of graded algebras
$$\HH^*((X_\infty)_\QQ^+) = \HH^*(B\pi_1(X_\infty)_\QQ^+)\tensor \HH^*(T_\infty;\QQ).$$
The proof is finished by combining this with the isomorphism \eqref{eq:fiber1}.
\end{proof}

For an arbitrary fibration $F\rightarrow E\rightarrow B$, injectivity of the map $\HH^*(B)\rightarrow \HH^*(E)$ is not enough to ensure collapse of the associated spectral sequence (see e.g. the discussion in \cite[p.148--149]{McCleary}). It is for this reason we need to know that the cohomology ring of $X_\infty$ is free. We use the following lemma. The proof is straightforward and left to the reader.

\begin{lemma} \label{lemma:collapse}
Let $F\rightarrow E\rightarrow B$ be a fibration of simply connected spaces of finite $\QQ$-type. If the cohomology rings $\HH^*(E;\QQ)$ and $\HH^*(B;\QQ)$ are free graded commutative and if $\HH^*(B;\QQ)\rightarrow \HH^*(E;\QQ)$ is injective on indecomposables, then there is an isomorphism of graded algebras
$$\HH^*(E;\QQ) \cong \HH^*(F;\QQ) \tensor \HH^*(B;\QQ).$$
\end{lemma}

We remark that the isomorphism $\HH^*(E;\QQ) \cong \HH^*(F;\QQ) \tensor \HH^*(B;\QQ)$ of Lemma \ref{lemma:collapse} is not canonical but depends on the choice of splitting of the short exact sequence of indecomposables
$$0\to Q\HH^*(B;\QQ) \to Q\HH^*(E;\QQ) \to Q\HH^*(F;\QQ) \to 0.$$
Similar remarks apply to the isomorphisms in Theorem \ref{thm:main} above and in Theorem \ref{thm:block diff cohomology} below since Lemma \ref{lemma:collapse} is used in the proof of these results.

\subsection{The stable cohomology of the block diffeomorphism group}
Let
$$Y_g =B\tDiff_\partial(M_{g,1}),$$
and consider the homotopy colimit over the stabilization maps
$$Y_\infty = \hocolim( Y_1 \to Y_2 \to \cdots).$$
By Theorem \ref{thm:bdiff stability}, the canonical map $Y_g \to Y_\infty$ induces an isomorphism in rational cohomology $\HH^k(Y_\infty;\QQ) \cong \HH^k(Y_g;\QQ)$ for $g>2k+4$.

\begin{theorem} \label{thm:block diff cohomology}
For $d\geq 3$, there is an isomorphism of graded rings
$$\HH^*(Y_\infty;\QQ) \cong \HH^*(B\Gamma_\infty;\QQ)\tensor \HH_{CE}^*(\gl_\infty\oplus \al_\infty)^{\Gamma_\infty}.$$
The two graded Lie algebras to the right are the colimits as $g\to \infty$ of
$$\gl_g = \Der_\omega^+ \LL(V_g),$$
$$\al_g = s^{-1}\Pi\tensor H_g^\QQ,$$
where $\al_g$ is abelian. They are equipped with the evident action of $\Gamma_g$.
\end{theorem}

\begin{proof}
Recall the notation from \S\ref{sec:block}. As follows from the results in that section, we have a $\pi_1(X_{g,J})$-equivariant isomorphism
\begin{equation*}
\HH^*(Y_{g,\circ};\QQ) \cong \HH^*(X_{g,\circ};\QQ)\tensor \HH^*(F_g;\QQ),
\end{equation*}
and we may rewrite the right-hand side in terms of Chevalley-Eilenberg cohomology:
$$\HH^*(X_{g,\circ};\QQ)\tensor \HH^*(F_g;\QQ) \cong \HH_{CE}^*(\gl_g)\tensor \HH_{CE}^*(\al_g) \cong \HH_{CE}^*(\gl_g\oplus \al_g).$$
Since the kernel of $\pi_1(X_{g,J}) \to \Gamma_g$ is finite, we may then write the $E_2$-term of the spectral sequence of \eqref{eq:fib2} as follows:
$$E_2^{p,q} = \HH^p(\pi_1(X_{g,J});\HH^q(Y_{g,\circ};\QQ)) \cong \HH^p(\Gamma_g;\HH_{CE}^q(\gl_g\oplus \al_g)).$$
In the stable range, $g>2p+2q+4$, we have that
\begin{equation*}
\HH^p(\Gamma_g;\HH_{CE}^q(\gl_g\oplus \al_g)) \cong \HH^p(\Gamma_g) \tensor \HH_{CE}^q(\gl_g\oplus \al_g)^{\Gamma_g}.
\end{equation*}
This follows from the argument that proves Theorem \ref{thm:separate} upon noticing that the Chevalley-Eilenberg cochains $C_{CE}^q(\gl_g\oplus \al_g)$ is a polynomial functor of degree $\geq 2q$. The rest of the argument is virtually identical to the proof of Theorem \ref{thm:main}, using the fibration diagram
$$
\xymatrix{Y_{g,\circ} \ar[r] \ar[d] & Y_g \ar[r] \ar[d] & B\pi_1(X_{g,J}) \ar[d] \\
T_g \ar[r] & (Y_g)_\QQ^+ \ar[r] & B\pi_1(X_{g,J})_\QQ^+.
}
$$
The fact that $\HH^1(Y_g;\QQ) = 0$, which is necessary for the construction of $(Y_g)_\QQ^+$, can be verified by using the spectral sequence of the fibration $F_g \to Y_g \to X_{g,J}$. Indeed, first note that $\HH^1(X_{g,J};\HH^0(F_g;\QQ)) = \HH^1(\pi_1(X_{g,J});\QQ) = 0$ since $\widehat {\Gamma}_g=\pi_1(X_{g,J})$ is rationally perfect (Proposition \ref{prop:X_g^J}). Secondly, Proposition \ref{prop:F_g} implies that $\HH^1(F_g;\QQ) = (H_g^\QQ)^\vee$ for $d\equiv 3$ (mod $4$) and $\HH^1(F_g;\QQ) =0$ for $d\not\equiv 3$ (mod $4$), as $\widehat{\Gamma}_g$-modules, from which it follows that $\HH^0(X_{g,J};\HH^1(F_g;\QQ)) = \HH^1(F_g;\QQ)^{\widehat{\Gamma}_g} = 0$.
\end{proof}

\section{Graph complexes} \label{sec:graph}
In the previous section, we arrived at the following description of the stable cohomology of the classifying spaces $X_g = B\aut_\partial(M_{g,1})$ and $Y_g = B\tDiff_\partial(M_{g,1})$:
\begin{align*}
\HH^*(X_{\infty};\QQ) & \cong \HH^*(\Gamma_\infty;\QQ) \tensor \HH_{CE}^*(\gl_\infty)^{\Gamma_\infty}, \\
\HH^*(Y_\infty;\QQ) & \cong \HH^*(\Gamma_\infty;\QQ) \tensor \HH_{CE}^*(\gl_\infty\oplus \al_\infty)^{\Gamma_\infty},
\end{align*}
cf.~ Theorem \ref{thm:main} and Theorem \ref{thm:block diff cohomology}.
As discussed earlier, the first factor $\HH^*(\Gamma_\infty;\QQ)$ is isomorphic to a polynomial algebra $\QQ[x_1,x_2,\ldots]$ on classes $x_i$ of degree $4i-2$ if $d$ is odd and $4i$ if $d$ is even.

In this section, we will examine the second factors. We will show how to express the invariant Lie algebra cohomology in terms of graph complexes. For the proof, it will be convenient to work dually with homology and coinvariants. Since the Chevalley-Eilenberg complex $C_*^{CE}(\gl_g)$ is a chain complex of finite dimensional algebraic representations, the coinvariants $\HH_*^{CE}(\gl_g)_{\Gamma_g}$ may be computed as the homology of the chain complex $C_*^{CE}(\gl_g)_{\Gamma_g}$. Indeed, as observed e.g.~in the proof of Proposition \ref{prop:e2 stability}, if $g\geq 2$ then $C_*^{CE}(\gl_g)$ is chain homotopy equivalent to $\HH_*^{CE}(\gl_g)$ as a complex of $\QQ[\Gamma_g]$-modules, and any additive functor, such as $(-)_{\Gamma_g}$, preserves chain homotopy equivalences. Similarly, $\HH_*^{CE}(\gl_g \oplus \al_g)_{\Gamma_g}$ may be computed as the homology of the chain complex $C_*^{CE}(\gl_g\oplus \al_g)_{\Gamma_g}$.

Recall the notation
$$\gl_g = \Der_\omega^+\LL(V_g),$$
$$\al_g = s^{-1} \Pi\tensor H_g.$$
Let $\GG$ denote the graph complex associated to the Lie operad, as described in the introduction, and let
$$\GG^d = \Sigma^{d-1}\GG.$$
The following is the main result of the section.
\begin{theorem} \label{thm:main graph}
There are isomorphisms of chain complexes,
$$C_*^{CE}(\gl_\infty)_{\Gamma_\infty} \cong \Lambda \GG^d(0),$$
$$C_*^{CE}(\gl_\infty\oplus \al_\infty)_{\Gamma_\infty} \cong \Lambda \GG^d[\Pi].$$
\end{theorem}

\begin{remark}
For $d$ odd, the first statement is essentially equivalent to a theorem of Kontsevich \cite{Kontsevich1,Kontsevich2}, a proof of which has been detailed in \cite{CV}. The proof offered here is new. It has the advantage that it rediscovers Kontsevich's graph complex, no prior construction of the graph complex is necessary.
For $d$ even, it is not a priori clear --- not to the authors at any rate --- that one would expect the same result; for one thing, the Lie algebra $\gl_g$ is different from Kontsevich's Lie algebra when $d$ is even since, e.g., $[\alpha,\alpha] \ne 0$ for an odd generator $\alpha$ of a free graded Lie algebra. Curiously, this difference is canceled in the course of the proof due to the difference between symplectic invariant theory and orthogonal invariant theory.
\end{remark}

\subsection{$\Sigma$-modules}
Let $\Sigma$ denote the groupoid of finite sets and bijections. A (left) \emph{$\Sigma$}-module in a category $\VV$ is a functor $\CC\colon \Sigma \to \VV$. A right $\Sigma$-module is a functor $\DD\colon \Sigma^{op}\to \VV$. Every right $\Sigma$-module $\DD$ can be converted into a left $\Sigma$-module $\DD^{op}$ by letting $\DD^{op}(S) = \DD(S)$ for a finite set $S$ and
$$\DD^{op}(\sigma) = \DD(\sigma^{-1})\colon \DD(S)\to \DD(T)$$
for a bijection $\sigma\colon S\to T$; we will do this tacitly in what follows.
We write $\CC(n)$ for $\CC(\{1,2,\ldots,n\})$ for $n\geq 1$ and $\CC(0) = \CC(\emptyset)$.

Now assume that the target category $\VV$ is symmetric monoidal and has all colimits. Given an object $V$ in $\VV$ and a finite set $S$, let
$$S\tensor V = \bigoplus_{s\in S} V.$$
We may also form the $S$-indexed tensor product
$$V^{\tensor S} = \bigotimes_{s\in S}V.$$
For $V$ fixed, $-\tensor V$ may be regarded as a left $\Sigma$-module and $V^{\tensor -}$ as a right $\Sigma$-module.

Let $(\Sigma\downarrow \Sigma)$ denote the category whose objects are functions $f\colon S\to T$ between finite sets and whose morphisms are commutative squares
$$
\xymatrix{S\ar[r]_-\cong^-\sigma \ar[d]^-f & S' \ar[d]^-{f'} \\ T \ar[r]_-\cong^-\tau & T',}
$$
where the horizontal maps are bijections. For a fixed finite set $S$, let $(S\downarrow \Sigma)$ denote the subcategory of $(\Sigma\downarrow \Sigma)$ where $\sigma=id_S$. In other words, the objects of $(S\downarrow \Sigma)$ are functions between finite sets $f\colon S\to T$ and the morphisms are commuting triangles
$$\xymatrix{S \ar[rr]^-f \ar[dr]_-{f'} && T \ar[dl]_-\cong^-\tau \\ & T'}$$
where $\tau$ is a bijection.

Every $\Sigma$-module $\CC$ gives rise to a functor $(\Sigma\downarrow \Sigma) \to \VV$, defined on objects by
$$\big(S\xrightarrow{f} T\big) \mapsto \CC(f) := \bigotimes_{t\in T} \CC(f^{-1}(t)).$$

Recall the composition product of $\Sigma$-modules (monoids over which are operads): the composition of two $\Sigma$-modules $\CC$ and $\DD$ is the $\Sigma$-module $\CC\circ \DD$, whose value on a finite set $S$ is given by
$$(\CC \circ \DD)(S) = \colim_{f\colon S\rightarrow T} \CC(T) \tensor \DD(f),$$
where the colimit is over the category $(S\downarrow \Sigma)$.

The levelwise tensor product $\CC\ltensor \DD$ is defined by
$$(\CC \ltensor \DD)(S) = \CC(S)\tensor \DD(S),$$
where $\Sigma$ acts diagonally.

The \emph{Schur functor} associated to a $\Sigma$-module $\CC$ is the functor $\CC[-]\colon \VV \to \VV$ defined by
$$\CC[V] = \colim_{S\in \Sigma} \CC(S) \tensor V^{\tensor S} \cong \bigoplus_{n\geq 0} \CC(n)\tensor_{\Sigma_n} V^{\tensor n}.$$
The main feature of the composition product is the existence of a natural isomorphism
$$\CC\big[\DD[V]\big] \cong (\CC\circ \DD)[V].$$

\subsection{Invariant theory and matchings}

\begin{definition}
A \emph{matching} on a set $S$ is a set $M$ of disjoint $2$-element subsets whose union is all of $S$. Let $\match{S}$ denote the set of all matchings on the set $S$.

If $\sigma\colon S\to T$ is a bijection, then for every matching $M\in\match{S}$ there is an induced matching $\sigma_*(M)\in\match{T}$ given by
$$\sigma_*(M) = \set{\{\sigma(x),\sigma(y)\}}{\{x,y\}\in M}.$$
In this way, $\matchfun$ may be viewed as a covariant functor $\Sigma\to \Set$.
\end{definition}
\begin{remark}
Note that $\match{S} = \emptyset$ if the number of elements $|S|$ of $S$ is odd. If $|S|$ is even, say $|S| = 2k$, then $|\match{S}| = (2k-1)!! = 1\cdot 3\cdot 5\cdot \ldots \cdot (2k-1)$.
\end{remark}

If $X$ is a set we let $\QQ X$ denote the graded vector space with basis $X$ concentrated in degree $0$ and we let $X^\vee$ denote the dual of $\QQ X$. If $V$ is a graded vector space, then we let $X\tensor V$ denote $\QQ X\tensor V$. Let $\sgn_{n}$ denote the sign representation of $\Sigma_n$, i.e., $\sgn_n=\QQ$ with action of $\sigma\in\Sigma_n$ given by multiplication by the sign $\sgn(\sigma)$ of $\sigma$. If $V$ is a graded vector space, $x_1,\ldots,x_n\in V$, $x=x_1\tensor\cdots\tensor x_n\in V^{\tensor n}$, and $\sigma\in \Sigma_n$, then we let $\sgn(\sigma,x)$ denote the sign for which
$$(x_1\tensor \cdots \tensor x_n)\sigma = \sgn(\sigma,x) x_{\sigma_1} \tensor \cdots \tensor x_{\sigma_{n}},$$
with respect to the standard right action of $\Sigma_n$ on the graded vector space $V^{\tensor n}$.

\begin{theorem} \label{thm:invariant theory}
Let $V$ be a graded anti-symmetric inner product space of degree $2d-2$, concentrated in degree $d-1$. Consider the pairing
$$\langle -, - \rangle \colon \match{2k}\tensor V^{\tensor 2k} \to \sgn_{2k}$$
defined by
$$\langle M, x_1\tensor \cdots \tensor x_{2k} \rangle = \sgn(\sigma)\sgn(\sigma,x)\langle x_{\sigma_1},x_{\sigma_2}\rangle \cdots \langle x_{\sigma_{2k-1}},x_{\sigma_{2k}}\rangle,$$
for $M =\{\{\sigma_1,\sigma_2\},\ldots,\{\sigma_{2k-1},\sigma_{2k}\}\}\in \match{2k}$.
This pairing gives rise to morphisms of $\Sigma_{2k}$-modules of degree $-2k(d-1)$,
\begin{align*}
\varphi\colon \match{2k}\tensor \sgn_{2k} \to \Hom_{\SP(V)}(V^{\tensor 2k},\QQ), & \quad \varphi(M)(x) = \langle M,x\rangle, \\
\psi\colon (V^{\tensor 2k})_{\SP(V)} \to \match{2k}\tensor \sgn_{2k}, & \quad \psi([x]) = \sum_{M\in \match{2k}} \langle M,x\rangle M.
\end{align*}
The morphism $\varphi$ is surjective and the morphism $\psi$ is injective. Both $\varphi$ and $\psi$ are isomorphisms if $\dim V\geq 2k$.
\end{theorem}

\begin{proof}
First of all, note that the pairing is well-defined because $(V,\langle-,-\rangle)$ is graded anti-symmetric and of even degree. As the reader may check, the pairing is $\Sigma_n$-equivariant is the sense that
$$\langle \tau_*(M),\tau x \rangle = \sgn(\tau) \langle M, x\rangle,$$
for all $\tau\in \Sigma_{2k}$.

Suppressing the grading, $V$ is in effect a symplectic vector space for $d$ odd, and a symmetric inner product space for $d$ even. The statements about $\varphi$ are essentially a summary of the first and second fundamental theorems for the symplectic and orthogonal groups, see \cite[\S9.5]{Loday}. Note that $\QQ\match{2k}$ is isomorphic to the $\Sigma_{2k}$-representation denoted $A_k$ in \cite{Loday}. The sign representation factor is not present in the fundamental theorem for the orthogonal group (\cite[Theorem 9.5.2]{Loday}), but it reappears due to the fact that, when $d$ is even, elements of $V$ are of odd degree $d-1$, which means that signs appear when tensor factors are permuted.

Turning to $\psi$, note that $(V^{\tensor 2k})_{\SP(V)}$ is dual to $\Hom_{\SP(V)}(V^{\tensor 2k},\QQ)$, up to a degree shift by $2k(2d-2)$. Note also that the $\Sigma_{2k}$-module $\QQ\match{2k}$ is self-dual. Indeed, a $\Sigma_{2k}$-equivariant isomorphism $\theta\colon \QQ\match{2k}^\vee \to \QQ\match{2k}$ is given by
$$\theta(f) = \sum_{M\in \match{2k}} f(M) M.$$
The map $\psi$ is the composite
{\small $$(V^{\tensor 2k})_{\SP(V)} \xrightarrow{\eta} ((V^{\tensor 2k})_{\SP(V)})^\vee)^\vee \xrightarrow{\varphi^\vee} (\match{2k}\tensor\sgn_{2k})^\vee \cong \match{2k}^\vee \tensor \sgn_{2k} \xrightarrow{\theta\tensor 1} \match{2k}\tensor \sgn_{2k},$$}
where $\eta$ is the canonical isomorphism from a finite dimensional graded vector space to its double dual.
\end{proof}

\begin{corollary} \label{cor:elementary}
Suppose that $v=\dim V\geq 2k$. Let $e_1,e_2,\ldots,e_v$ be a basis for $V$, and let 
$$e = e_1\tensor e_1^\# \tensor e_2 \tensor e_2^\# \tensor \ldots \tensor e_k\tensor e_k^\#\in V^{\tensor 2k}.$$
Then the coinvariants $(V^{\tensor 2k})_{\SP(V)}$ has basis
$$[e\sigma^{-1}],\quad \sigma \in C_{2k},$$
where $C_{2k}\subseteq \Sigma_{2k}$ is the set of permutations $\sigma$ such that
$\sigma_1 < \sigma_3 < \cdots < \sigma_{2k-1}$ and $\sigma_{2i-1} <\sigma_{2i}$ for all $i$.
\end{corollary}

\begin{proof}
Let $E$ denote the matching $\{\{1,2\},\{3,4\},\ldots,\{2k-1,2k\}\}\in \match{2k}$. There is a bijection $C_{2k} \to \match{2k}$ given by
$$\sigma \mapsto \sigma_*(E) = \{\{\sigma_1,\sigma_2\},\{\sigma_3,\sigma_4\},\ldots \{\sigma_{2k-1},\sigma_{2k}\}\}.$$
Hence, the $\Sigma_{2k}$-module $\match{2k}\tensor \sgn_{2k}$ has basis $\sigma_*(E)$ for $\sigma\in C_{2k}$.
The isomorphism $\psi\colon (V^{\tensor 2k})_{\SP(V)}\to \match{2k}\tensor \sgn_{2k}$ is such that
$$\psi[e] = E,$$
whence $\psi[e \sigma^{-1}] = \sgn(\sigma) \sigma_*(E)$ because $\psi$ is $\Sigma_{2k}$-equivariant.
It follows that $[e\sigma^{-1}]$, for $\sigma\in C_{2k}$ is a basis for $(V^{\tensor 2k})_{\SP(V)}$, as claimed.
\end{proof}

\begin{remark} \label{remark:elementary}
We note for future reference that, in particular, $(V^{\tensor 2k})_{\SP(V)}$ is spanned by \emph{elementary tensors}, meaning elements of the form $[\big(e_1\tensor e_1^\# \tensor e_2 \tensor e_2^\# \tensor \ldots \tensor e_k\tensor e_k^\#\big)\sigma]$ for some $\sigma\in \Sigma_{2k}$.
\end{remark}

Let $\VV$ denote the category of graded vector spaces over $\QQ$ and let $\SP = \SP_{d-1}^{2d-2}$ denote the category of graded anti-symmetric inner product spaces of degree $2d-2$, concentrated in degree $d-1$. For an $\SP$-module (i.e., a functor) $M\colon \SP \to \VV$, we define the `$\SP$-coinvariants' by
$$M_\SP = \colim_{V\in \SP} M(V).$$
Informally, the coinvariants $M_\SP$ may be though of the value of $M(V)_{\SP(V)}$ as the dimension of $V$ tends to infinity. There is an $\SP$-$\Sigma$-bimodule given by
$$\SP\times \Sigma \to \VV, \quad (V,S) \mapsto V^{\tensor S}.$$
Using $\SP$-coinvariants, Theorem \ref{thm:invariant theory} admits the following elegant formulation.

\begin{corollary} \label{cor:invariant theory}
There is an isomorphism of $\Sigma$-modules
$$\big(V^{\tensor S})_{\SP} \cong s^{|S|(d-1)}\match{S}\tensor \sgn_S.$$
\end{corollary}

\subsection{The graph complex}

\begin{definition}
A \emph{graph} $G = \big(F\xrightarrow{f} V,E\big)$ consists of a set $F$ of \emph{flags} or \emph{half-edges}, a set $V$ of \emph{vertices}, a function $f\colon F\to V$, and a matching $E$ on $F$, elements of which are thought of as the \emph{edges} of the graph.

An \emph{isomorphism of graphs $G\to G'$} is a pair of bijections $\sigma\colon F\to F'$, $\tau\colon V\to V'$, that commute with the structure maps,
$$
\xymatrix{F\ar[r]_-\cong^-\sigma \ar[d]^-f & F' \ar[d]^-{f'} \\ V \ar[r]_-\cong^-\tau & V',}
$$
and preserve edges in the sense that $\sigma_*(E) = E'$. Let $\grph$ denote the groupoid of graphs and their isomorphisms.
\end{definition}
The \emph{valence} of a vertex $v\in V$ is the cardinality of the set $f^{-1}(v)$.
The Euler characteristic of a graph is defined by $\chi(G) = |V|-|E|$.

Every $\Sigma$-module $\CC$ gives rise to a functor $\CC\colon \grph \to \VV$, whose value at a graph $G = (F \xrightarrow{f} V, E)$ is
$$\CC(G) = \CC(f) = \bigotimes_{v\in V} \CC(f^{-1}(v)),$$
cf.~\cite[(2.12)]{GK2}.

\begin{definition}
For a $\Sigma$-module $\CC$, define the space of $\CC$-decorated graphs by
\begin{align*}
\GG^d\CC & = \colim_{G\in \grph} s^{|V|-\chi(G)(2d-2)} \sgn_V\tensor \sgn_F\tensor \CC(G) \\
& \cong \bigoplus_{[G]} s^{|V|-\chi(G)(2d-2)} \big(\sgn_V\tensor \sgn_F\tensor \CC(G)\big)_{\Aut(G)},
\end{align*}
where the colimit is over the groupoid of graphs $\grph$, and the sum is over all isomorphism classes $[G]$ of graphs.
\end{definition}

The reader may compare this with \cite[(2.18)]{GK2}. Thus, $\GG^d \CC$ is spanned by oriented graphs whose vertices are decorated by elements of $\CC$. To specify a $\CC$-decorated graph one needs to supply the data of
\begin{itemize} 
\item a graph $G = \big(F\xrightarrow{f} V,E\big)$,
\item an orientation of the vertices and an orientation of the flags, and
\item for each vertex $v\in V$, an element $\xi_v\in \CC(f^{-1}(v))$.
\end{itemize}
The homological degree of a decorated graph is $|V|-\chi(G)(2d-2)$ (plus the homological degrees of the decorations $\xi_v$ in case $\CC$ is equipped with a grading). For $d=1$, the space $\GG \CC$ is isomorphic to the space of `$\CC$-graphs' \cite{CV}. For $d > 1$, $\GG^d \CC$ is simply a regraded version of $\GG^1 \CC$. If $\CC$ is a cyclic operad, then there is a differential $\partial\colon \GG^d \CC_k \to \GG^d \CC_{k-1}$, defined as a sum over edge contractions, see e.g.~\cite{CV} for a detailed description.

\subsubsection{Colimits over Grothendieck constructions}
Let $I$ be a category and $F\colon I\to \Cat$ be a functor from $I$ to the category of small categories. The \emph{Grothendieck construction} $I\int F$ is the category whose objects are pairs $(i,x)$, where $i$ is an object of $I$ and $x$ is an object of $F(i)$. A morphism $(i,x) \to (j,y)$ in $I\int F$ is a pair $(f,g)$ where $f\colon i\to j$ is a morphism in $I$ and $g\colon F(f)(x) \to y$ is a morphism in $F(j)$. There are evident functors
$$\mbox{$F(i) \xrightarrow{\iota_i} I\int F \xrightarrow{\pi} I.$}$$

Consider a functor $D\colon I\int F \to \VV$ to some category $\VV$ with all colimits. For a fixed $i$, we get a functor $D\iota_i = D(i,-) \colon F(i) \to \VV$ and we may form its colimit
$$\colim_{x\in F(i)} D(i,x).$$
As $i$ varies, these colimits assemble into a functor from $I$ to $\VV$, and we may form its colimit
$$\colim_{i\in I} \colim_{x\in F(i)} D(i,x).$$
On the other hand, we may form the of colimit $D$ over $I\int F$. It is an exercise to check that the results are canonically isomorphic. For reference, we state this as a proposition.

\begin{proposition} \label{prop:grothendick colimit}
For every diagram $D\colon I\int F \to \VV$, indexed by the Grothendieck construction of a functor $F\colon I\to \Cat$, there is a canonical isomorphism
$$\colim_{(i,x)\in I\int F} D(i,x) \cong \colim_{i\in I} \colim_{x\in F(i)} D(i,x).$$
\end{proposition}\hfill $\square$

We will apply this observation twice in the proof of Theorem \ref{thm:graph complex} below.

There is a functor $(-\downarrow \Sigma)\colon \Sigma \to \Cat$ sending a finite set $S$ to the category $(S\downarrow \Sigma)$. The Grothendieck construction $\Sigma \int (-\downarrow \Sigma)$ and the comma category $(\Sigma \downarrow \Sigma)$ are isomorphic as categories over $\Sigma$.

Next, we observe that the groupoid $\grph$ is equal to the Grothendieck construction $(\Sigma\downarrow \Sigma) \int \matchfun$, where $\matchfun\colon (\Sigma\downarrow \Sigma) \to \Cat$ is the functor that sends an object $f\colon S\to T$ to the set $\match{S}$ of matchings on the source, viewed as a category with only identity morphisms.

Recall that $\SP = \SP_{d-1}^{2d-2}$ denotes the category of graded anti-symmetric inner product spaces of degree $2d-2$ concentrated in degree $d-1$.
Taking cue from Proposition \ref{prop:Lie iso}, we associate to every $\Sigma$-module $\CC$ an $\SP$-module $V\mapsto \CC((V))$, where
$$\CC((V)) = s^{2-2d} \bigoplus_{n\geq 0} \CC(n)\tensor_{\Sigma_n} V^{\tensor n}.$$

\begin{theorem} \label{thm:graph complex}
There is a canonical isomorphism
$$\Lambda s \CC((V))_\SP \cong \GG^d \CC.$$
\end{theorem}

\begin{proof}
The functor $V\mapsto \Lambda s V$ may be identified with the Schur functor associated to the $\Sigma$-module $\Lambda s$ with $\Lambda s(T) = s^{|T|}\sgn_T$ for a finite set $T$. In particular,
\begin{align*}
\Lambda s \CC((V)) & \cong (\Lambda s \circ \CC)[V] \\
& = \colim_{S\in \Sigma} \big(\Lambda s \circ \CC\big)(S) \tensor V^{\tensor S} \\
& = \colim_{S\in \Sigma} \big(\colim_{f\colon S\to T} \Lambda s(T) \tensor \CC((f)) \big) \tensor V^{\tensor S} \\
& \cong \colim_{f\colon S\to T\in (\Sigma\downarrow \Sigma)} \Lambda s (T) \tensor \CC((f)) \tensor V^{\tensor S},
\end{align*}
where we have used Proposition \ref{prop:grothendick colimit} in the last step. Since colimits commute with colimits and tensor products we get
\begin{align*}
\Lambda s\CC[V]_{\SP} & \cong \colim_{f\colon S\to T\in (\Sigma\downarrow \Sigma)} \Lambda s(T) \tensor \CC((f)) \tensor \big(V^{\tensor S}\big)_{\SP} \\
& \cong \colim_{f\colon S\to T\in (\Sigma\downarrow \Sigma)} \Lambda s(T) \tensor \CC((f)) \tensor s^{|S|(d-1)} \match{S} \tensor \sgn_S,
\end{align*}
where we use Corollary \ref{cor:invariant theory} in the last step. Viewing the set $\match{S}$ as a category with only identity morphisms, we may rewrite the above expression as
$$\colim_{f\colon S\to T \in (\Sigma \downarrow \Sigma)} \colim_{M\in \match{S}} s^{|S|(d-1)} \Lambda s(T)\tensor \CC((f))\tensor \sgn_S.$$
As noted above, the groupoid $\grph$ is isomorphic to the Grothendieck construction $(\Sigma\downarrow \Sigma) \int \matchfun$. Remembering that $\Lambda s(T) = s^{|T|}\sgn_T$, another application of Proposition \ref{prop:grothendick colimit} (and a change of notation $V=T$, $F=S$) then shows that the above colimit is isomorphic to
$$\colim_{G\in \grph} s^{(3-2d)|V| + |F|(d-1)}\sgn_V \tensor \sgn_F\tensor \CC(G),$$
which is equal to $\GG\CC$ by definition (note that $(3-2d)|V| + |F|(d-1) = |V|-(2d-2)\chi(G)$).
\end{proof}

\begin{remark}
Note how the groupoid of graphs emerges through successive assembly of colimits in the proof of the previous theorem, in effect allowing us to rediscover both the groupoid of graphs and the space of $\CC$-graphs. Similarly, in the next theorem we will rediscover the graph complex differential.
\end{remark}

If $\CC$ is a cyclic operad, then the formula \eqref{eq:lie bracket} endows $\CC((V))$ with the structure of a graded Lie algebra, and we may form the Chevalley-Eilenberg complex $C_*^{CE}\big(\CC((V))\big) = \big( \Lambda s \CC((V)),d_{CE} \big)$.

\begin{theorem} \label{thm:operad}
If $\CC$ is a cyclic operad, then the isomorphism in Theorem \ref{thm:graph complex} commutes with differentials, yielding an isomorphism of chain complexes
$$C_*^{CE}\big(\CC((V))\big)_{\SP} \cong \big(\GG^d \CC,\partial\big).$$
\end{theorem}

\begin{proof}
It follows from Corollary \ref{cor:elementary} that the graded vector space $\Lambda^n s\CC((V))_{\SP(V)}$ is spanned by $\SP(V)$-orbits of elements of the form
\begin{equation} \label{eq:element}
\xi_1 \tensor h_1 \wedge \cdots \wedge \xi_{n}\tensor h_{n},\quad \xi_j\in \CC(i_j),\quad h_j\in V^{\tensor i_j},
\end{equation}
such that $h_1\tensor \cdots \tensor h_{n} \in V^{\tensor 2k}$ is an elementary tensor, i.e., of the form 
$$h_1\tensor \cdots \tensor h_{n} = \big(e_1\tensor e_1^\# \tensor e_2 \tensor e_2^\# \tensor \ldots \tensor e_k\tensor e_k^\#\big)\sigma^{-1},$$
for some $\sigma\in C_{2k}$.

Tracing through the isomorphism $\Lambda s \CC((V))_{\SP(V)} \cong \GG^d \CC$ of Theorem \ref{thm:graph complex}, the class of the element \eqref{eq:element} may be represented by the decorated graph, whose underlying graph $G=\big(F \xrightarrow{f} V,E\big)$ has vertices $V=\{h_1,h_2,\ldots,h_n\}$, flags $F=\{e_1,e_1^\#,\ldots,e_k,e_k^\#\}$, edges $E = \{\{e_1,e_1^\#\},\ldots,\{e_k,e_k^\#\}\}$, and where $f\colon F\to V$ is given by $f(e) = h_j$ if $e$ appears as a tensor factor in $h_j$. The vertex $h_i$ is decorated by the element $\xi_i$, the orientation of the vertices is $h_1\wedge \ldots \wedge h_n$ and the orientation of the flags is determined by $\sigma$.

Towards describing the Chevalley-Eilenberg differential applied to \eqref{eq:element}, recall that the formula \eqref{eq:lie bracket} defines the Lie bracket of two elements $\xi_s\tensor h_s$ and $\xi_t \tensor h_t$ in $\CC((V))$:
$$\big[\xi_s\tensor h_s,\xi_t \tensor h_t\big] = \sum_{i,j} (\xi_s) {}_i\circ_j (\xi_t) \tensor (h_s) {}_i\circ_j (h_t).$$
Since $h_1\tensor \cdots \tensor h_n$ is an elementary tensor, the only possibility for $(h_s) {}_i\circ_j (h_t)$ to be non-zero is if the $i$th tensor factor of $h_s$ is $e_r$ and the $j$th factor of $h_t$ is $e_r^\#$, or vice versa, for some $r$. In other words, $(h_s) {}_i\circ_j (h_t)$ is non-zero only if the vertex $h_s$ is connected to $h_t$ by an edge $\{e_r,e_r^\#\}$ in the graph $G$. Bearing this observation in mind, the expression for the Chevalley-Eilenberg differential,
\begin{align*}
d_{CE}(\xi_1 \tensor h_1 & \wedge \cdots \wedge \xi_{n}\tensor h_{n}) = \\
& \sum_{s<t} \pm \big[\xi_s\tensor h_s,\xi_t \tensor h_t\big] \wedge \xi_1 \tensor h_1 \wedge \cdots \widehat{\xi_s\tensor h_s} \cdots \widehat{\xi_t \tensor h_t} \cdots \wedge \xi_{n}\tensor h_{n},
\end{align*}
may be simplified: Since $h_1\tensor \cdots \tensor h_n$ is an elementary tensor we may instead sum over all edges in $G$. Indeed, for each edge $\{e_r,e_r^\#\}$ there are unique $s,t,i,j$ such that the $i$th factor of $h_{s}$ is $e_r$ and the $j$th factor of $h_{t}$ is $e_r^\#$. So the above expression may be written as
$$\sum_{r=1}^k \pm (\xi_s) {}_i\circ_j (\xi_t) \tensor (h_s) {}_i\circ_j (h_t) \wedge \xi_1 \tensor h_1 \wedge \cdots \widehat{\xi_s\tensor h_s} \cdots \widehat{\xi_t \tensor h_t} \cdots \wedge \xi_{n}\tensor h_{n}.$$
The $r$th summand in the above expression corresponds to the decorated graph obtained by contracting the edge $\{e_r,e_r^\#\}$ in the decorated graph we started with. So we have rediscovered Kontsevich's graph complex differential.
\end{proof}

The proof of Theorem \ref{thm:main graph} consists in two applications of Theorem \ref{thm:operad}.
By Proposition \ref{prop:Lie iso} the graded Lie algebra $\gl_g= \Der_{\omega_g}^+ \LL(V_g)$ is isomorphic to the Lie algebra associated to the cyclic operad $\Lie$. Thus,
$$C_*^{CE}(\gl_\infty)_{\Gamma_\infty} \cong C_*^{CE}(\Lie((V)))_{\SP} \cong \GG^d \Lie.$$
The graph complex in the introduction involves only connected graphs, whereas disconnected graphs are allowed in the definition of $\GG^d \Lie$. By interpreting a disconnected graph as a formal product of its connected components one sees that
$$\GG^d\Lie \cong \Lambda \GG^d(0).$$
Secondly, the graded Lie algebra $\gl_g\oplus \al_g$ is isomorphic to the Lie algebra associated to the cyclic operad $\Lie_\Pi$ with
$$\Lie_\Pi((n)) = \left\{ \begin{array}{ll} \Lie((n)), & n\geq 3, \\ s^{d-2}\Pi, & n = 1. \end{array} \right.$$
The `hairy graph complex' $\Lambda \GG^d[\Pi]$ is easily seen to be isomorphic to $\GG^d\Lie_\Pi$.

\appendix

\tocless\section{Cohomology of arithmetic groups}

The automorphism groups $\Aut(H,\mu,q)$ and $\Aut(H,\mu,Jq)$ associated to a quadratic module $(H,\mu,q)$ are arithmetic. We will summarize the results on the cohomology of arithmetic groups that we need. We refer to Serre's survey article \cite{Serre}, and the references therein, for more details.

\begin{theorem} \label{thm:arithmetic groups}
Let $G$ be an algebraic group defined over $\QQ$, let $\Gamma$ be an arithmetic subgroup of $G_\QQ$, and let $V$ be a finite dimensional $\QQ$-vector space with an action of $\Gamma$. Then
\begin{enumerate}
\item The cohomology groups $\HH^k(\Gamma;V)$ are finite dimensional.
\item If $G$ is simple and of $\QQ$-rank at least $2$, then the first cohomology group vanishes, $\HH^1(\Gamma;V) = 0$.

\end{enumerate}
\end{theorem}

\begin{proof}
If $\Gamma$ is torsion-free the first claim follows from the fact that the trivial $\ZZ[\Gamma]$-module $\ZZ$ admits a finite length resolution by finitely generated free $\ZZ[\Gamma]$-modules. For general $\Gamma$, there exists a torsion-free subgroup $\Gamma'\subseteq \Gamma$ of finite index, and the claim follows because $\HH^k(\Gamma;V)$ may be identified with the set of $\Gamma$-invariants in $\HH^k(\Gamma';V)$ by a transfer argument (see e.g., \cite[III.(10.4)]{Brown}).

If $G$ is simple and of $\QQ$-rank at least $2$, every finite dimensional representation $V$ of $\Gamma$ is almost algebraic (see \cite[1.3(9)]{Serre}). This means that there is a finite index subgroup $\Gamma'\subseteq \Gamma$ such that the restriction of $V$ to $\Gamma'$ is the restriction of an algebraic representation of the algebraic group $G$. This implies that the first cohomology group $\HH^1(\Gamma;V)$ vanishes, as in \cite[Corollary 16.4]{BMS}.
\end{proof}

\tocless\section{Some elementary homological algebra} \label{sec:appendixb}
We will consider $\ZZ$-graded chain complexes over an associative ring $R$, e.g., $R = \QQ[\pi]$ for a group $\pi$.

A chain complex $C_*$ is called split if there are maps $s_n\colon C_n\rightarrow C_{n+1}$ such that $dsd = d$. Equivalently, there is a chain homotopy equivalence between $C_*$ and the homology $H_*(C)$, viewed as a chain complex with trivial differential.

\begin{lemma} \label{lemma:split}
If $C_*$ is a split chain complex then there is a chain homotopy equivalence
$$\xymatrix{C_* \ar[r]_-{p_C}^-\simeq & \HH_*(C)}$$
such that $p_C(z) = [z]$ if $z$ is a cycle.

If $f\colon C_*\rightarrow D_*$ is a chain map between split chain complexes (not necessarily compatible with the splittings), then the diagram
$$\xymatrix{C_* \ar[r]^-f \ar[d]_-{p_C} & D_* \ar[d]^-{p_D} \\ \HH_*(C) \ar[r]^-{\HH_*(f)} & \HH_*(D)}$$
commutes up to chain homotopy.
\end{lemma}

\begin{proof}
Let $s\colon C_*\rightarrow C_{*+1}$ satisfy $dsd = d$. The reader may check that the formulas
$$\xymatrix{*[r]{C_*} \ar@<1ex>[r]^-{p} \ar@(ul,dl)[]_{h} & \HH_*(C) \ar@<1ex>[l]^-{\nabla}}$$
$$p(x) = [x-sd(x)],\quad \nabla[z] = z-ds(z),\quad h = s-s^2d$$
give well-defined maps that satisfy
$$p\nabla = 1,\quad 1-\nabla p = dh+hd.$$
Clearly, $p(z) = [z]$ if $z$ is a cycle.

Next, consider a chain map $f\colon C_*\rightarrow D_*$ between split chain complexes. Since $p_C(z) = [z]$ for cycles $z$, we have that $\nabla_C[z] = z - dh(z)$. Therefore,
\begin{align*}
p_D f \nabla_C [z] & = p_D f(z - dh(z)) \\
& = p_D(f(z)) \\
& = [f(z)],
\end{align*}
showing that $p_D f_* \nabla_C = \HH_*(f)$. Hence,
$$\HH_*(f) p_C = p_Df_*\nabla_C p_C \simeq p_D f.$$
\end{proof}

\begin{lemma} \label{lemma:split chain}
A chain complex $C_*$ is split if and only if the short exact sequences
$$0\rightarrow Z_n \rightarrow C_n\stackrel{d_n}{\rightarrow} B_{n-1} \rightarrow 0$$
$$0\rightarrow B_n\rightarrow Z_n \rightarrow H_n \rightarrow 0$$
are split exact for all $n$. Here, $Z_n = \ker(d_n)$, $B_{n-1} =\im(d_n)$, and $H_n = \HH_n(C_*)$.
\end{lemma}

\begin{definition} \label{def:rationally perfect appendix}
We will say that a group $\pi$ is \emph{rationally perfect} if $\HH^1(\pi;V) = 0$ for all finite dimensional vector spaces $V$ over $\QQ$ with an action of $\pi$.
\end{definition}

\begin{lemma}
A group $\pi$ is rationally perfect if and only if $\Ext_{\QQ[\pi]}^1(W,V) = 0$ for all finite dimensional vector spaces $V$ and $W$ over $\QQ$ with an action of $\pi$.
\end{lemma}

\begin{proof}
Use the relation $\Ext_{\QQ[\pi]}^1(W,V) \cong \HH^1(\pi;\Hom_{\QQ}(W,V))$.
\end{proof}

\begin{proposition} \label{prop:split}
Let $\pi$ be a rationally perfect group. If $C_*$ is a chain complex of $\QQ[\pi]$-module such that $C_n$ is finite dimensional over $\QQ$ for every $n$, then $C_*$ is split.
\end{proposition}

\begin{proof}
If $C_n$ is finite dimensional over $\QQ$ for all $n$, then so are $Z_n$, $B_n$ and $H_n$. Since $\pi$ is rationally perfect, the Ext-groups $\Ext_{\QQ[\pi]}^1(H_n,B_n)$ and $\Ext_{\QQ[\pi]}^1(B_{n-1},Z_n)$ vanish for all $n$, which forces $C_*$ to split by Lemma \ref{lemma:split chain}.
\end{proof}

\tocless\section{A $\QQ$-local plus construction} \label{sec:plus}
Let $X$ be a connected space of finite $\QQ$-type such that $\HH^1(X;\QQ) = 0$. Then $X$ admits a minimal Sullivan model $\mathscr{M}_X$ of finite type with generators in degree $2$ and above, see e.g., \cite[Proposition 12.2]{FHT-RHT}. The spatial realization $|\mathscr{M}_X|$ is then a simply connected $\QQ$-local space of finite $\QQ$-type. Moreover, the canonical map $X\rightarrow | \mathscr{M}_X|$ is a rational cohomology isomorphism. One may view $|\mathscr{M}_X|$ as a $\QQ$-local version of the plus construction, and we will denote it by $X_\QQ^+$. In fact, if the fundamental group of $X$ is perfect, i.e., $\HH_1(X;\ZZ) = 0$, then $X_\QQ^+$ is a $\QQ$-localization of the ordinary plus construction. Note however that the rational plus construction has a wider range of applicability as it does not require the fundamental group to be perfect.

\tocless\section{Proof of Theorem \ref{thm:S-tilde}} \label{sec:S-tilde}
For a vector bundle $\xi$ over a finite CW-complex $X$ with a closed subspace $C\subseteq X$, we defined in \S\ref{sec:partial linearization} the $\Delta$-monoid $\widetilde{S}_\bullet \aut_C(\xi)$ and its stable version $\widetilde{S}_\bullet \aut_C(\xi^S)$. There is an obvious forgetful map
$$
\widetilde{\pi} \colon \widetilde{S}_\bullet \aut_C(\xi^S) \to S_\bullet \aut_C(X).
$$

\begin{lemma} \label{lemma:kan condition}
The map
$$
\widetilde{\pi} \colon \widetilde{S}_\bullet \aut_{C,\circ}(\xi^S) \to S_\bullet \aut_{C,\circ}(X)
$$
satisfies the Kan condition.
\end{lemma}

\begin{proof}
Given a diagram
$$
\xymatrix{\Lambda^k \ar[d] \ar[r]^-{F} & \aut_{C,\circ}(\RR^{k-1}\times \xi) \ar[d] \\
\Delta^k \ar[r]^-f & \aut_{C,\circ}(X)}
$$
that displays $(f,F)$ as a horn in $\widetilde{S}_{k-1}\aut_{C,\circ}(\xi)$, we seek to extend it to an element $(f,\widetilde{F})$ of $\widetilde{S}_k\aut_{C,\circ}(\xi)$ provided $\dim \xi - \dim X - 1 > k$. Since we stabilize, we may assume $\dim \xi - \dim X \gg k$. We treat the case of the $k$-th horn, i.e., $\Lambda^k$ is the union of the faces $d_i\Delta^k = d^i(\Delta^{k-1})$ for $i=0,1,\ldots,k-1$. The cases of the other horns are similar.
Let $F_i$ be the restriction of $F$ to $d_i\Delta^k$ so that
$$
d_\nu F_\mu = d_{\mu-1} F_\nu,\quad 0\leq \nu < \mu \leq k-1.
$$
In \S\ref{sec:partial linearization} we introduced the isomorphism $\phi_\mu \colon \RR\times \RR^{k-1} \to \RR^k$ and we let 
$$
(\phi_\mu)_\# \colon \aut_{C,\circ}(\RR^{k-1} \times \xi) \to \aut_{C,\circ}(\RR^k\times \xi)
$$
be the map
$$
(\phi_\mu)_\#(F) = \phi_\mu(id_\RR \times F) \phi_\mu^{-1}.
$$
Then $F_\mu' = (\phi_\mu)_\#(F_\mu)$ is a map from $d_\mu \Delta^k$ into $\aut_{C,\circ}(\RR^k\times \xi)$ and these morphisms fit together to define a diagram
$$
\xymatrix{\Lambda^k \ar[d] \ar[r]^-{F'} & \aut_{C,\circ}(\RR^k\times \xi) \ar[d] \\
\Delta^k \ar[r]^-f & \aut_{C,\circ}(X).}
$$
Indeed, $d_\nu\Delta^k \cap d_\mu \Delta^k = d_\nu d_\mu \Delta^k = d_{\mu-1} d_\nu \Delta^k$ when $\nu < \mu$, and since
$$
(\phi_\nu)_\# (\phi_\mu)_\# = (\phi_{\mu-1})_\# (\phi_\nu)_\# \colon \aut_{C,\circ}(\RR^{k-2} \times \xi) \to \aut_{C,\circ}(\RR^k \times \xi)
$$
it follows that $F_\nu'$ and $F_\mu'$ agree on $d_\nu\Delta^k \cap d_\mu \Delta^k$.

As in the proof of Lemma \ref{lemma:serre fibration} we may apply \cite[\S11.3]{Steenrod} to extend $(f,F')$ to $(f,F'')$ with
$$
\xymatrix{\Delta^k\times \RR^k \times \xi \ar[d] \ar[r]^-{F''} & \RR^k\times \xi \ar[d] \\ \Delta^k \times X \ar[r]^-f & X.}
$$
Off hand there is no reason to expect that $F''$ maps the $k$-th face of $\Delta^k$ into $(\phi_k)_\#\aut_{C,\circ}(\RR^{k-1}\times\xi)$, which is required for $(f,F'')$ to define a $k$-simplex of $\widetilde{S}_\bullet\aut_{C,\circ}(\xi)$. But we can adjust $F''$, using that 
$$(\phi_k)_\# \colon \aut_{C,\circ}(\RR^{k-1}\times \xi) \to \aut_{C,\circ}(\RR^k\times \xi)$$
is highly connected when $\dim \xi - \dim X \gg k$, cf.~Proposition \ref{prop:connectivity}. The boundary $(k-2)$-sphere $\partial(d_k\Delta^k)$ is contained in $\partial(\Lambda^k)$ and the diagram
$$
\xymatrix{d_k \Delta^k \ar@{-->}[dr] \ar[r]^-{F''} & \aut_{C,\circ}(\RR^k\times \xi) \\ \partial(d_k\Delta^{k-1}) \ar[u] \ar[r]^-{F'} & \aut_{C,\circ}(\RR^{k-1}\times \xi) \ar[u]_-{(\phi_k)_\#}}
$$
represents an element of $\pi_{k-1}((\phi_k)_\#)$, which with our assumption on $\dim \xi$ is the zero group. Hence, $F''$ deforms to a map from $d_k\Delta^k$ into $\aut_{C,\circ}(\RR^{k-1}\times \xi)$ and using a collar to absorb the deformation, we have obtained the required $k$-simplex of $\widetilde{S}_\bullet \aut_{C,\circ}(\xi)$.

We remember that $\widetilde{S}_\bullet \aut_{C,\circ}(\xi^S)$ is the colimit of $\widetilde{S}_\bullet \aut_{C,\circ}(\xi^s)$, $\xi^s = \xi\times \RR^s$, as $s\to\infty$. The condition $\dim \xi^s - \dim X > k$ is satisfied for large $s$ because $X$ was assumed to be a finite $CW$-complex.
\end{proof}

The map that sends $(f,\widehat{f})\in \aut_{C,\circ}(\xi)$ into $(f,id\times \widehat{f})\in \aut_{C,\circ}(\RR^k\times \xi)$ induces a map from $S_k\aut_{C,\circ}(\xi)$ to $\widetilde{S}_k\aut_{C,\circ}(\xi)$ and gives rise to the $\Delta$-map
$$
\alpha_\bullet \colon S_\bullet \aut_{C,\circ}(\xi^S) \to \widetilde{S}_\bullet\aut_{C,\circ}(\xi^S).
$$
Both $\Delta$-sets are Kan $\Delta$-sets; for the target this is a consequence of Lemma \ref{lemma:kan condition}. Hence, the homotopy groups of their geometric realizations may be calculated from the combinatorial homotopy groups.

\begin{lemma}
$\pi_k(\alpha_\bullet)$ is an isomorphism.
\end{lemma}

\begin{proof}
Assuming $\dim \xi - \dim X \gg k$,
$$\pi_k(\widetilde{S}_\bullet \aut_{C,\circ}(\xi)) = Z_k\big(\widetilde{S}_\bullet \aut_{C,\circ}(\xi)\big) / \sim$$
where
$$
Z_k\big(\widetilde{S}_\bullet \aut_{C,\circ}(\xi)\big) = \set{F\in \widetilde{S}_k \aut_{C,\circ}(\xi)}{d_\mu F =id, \,\, \mu = 0,\ldots,k},
$$
and where $F_1 \sim F_2$ if there exists a $G\in \widetilde{S}_{k+1} \aut_{C,\circ}(\xi)$ with $d_k G = F_1$, $d_{k+1} G = F_2$ and $d_\mu G = 0$ for $\mu <k$. If
$$
F\colon \Delta^k \to \aut_{C,\circ}(\RR^k \times \xi)
$$
is an element of $Z_k\big(\widetilde{S}_\bullet \aut_{C,\circ}(\xi)\big)$, then
$$
F\circ d^\mu = \phi_\mu \circ (id \times d_\mu F) \circ \phi_\mu^{-1},\quad \mu =0,\ldots,k,
$$
and since $d_\mu F=id$, we may view $F$ as a map
$$
F\colon (S^k,*) \to \aut_{C,\circ}(\RR^k\times \xi).
$$
We claim that two such maps represent the same element of $\pi_k \widetilde{S}_\bullet \aut_{C,\circ}(\xi)$ if and only if they are homotopic as maps into $\aut_{C,\circ}(\RR^{k+1}\times \xi)$. To wit, write $\Delta^{k+1}$ as the join of the $(k-1)$-simplex $\langle v_0,\ldots,v_{k-1}\rangle = d_k d_{k+1} \Delta^{k+1}$ and the $1$-simplex $\langle v_k,v_{k+1}\rangle$, where $\{v_0,\ldots,v_{k+1}\}$ is the set of vertices of $\Delta^{k+1}$. The family of simplices
$$
\langle v_0,\ldots,v_{k-1} \rangle * \set{tv_k + (1-t)v_{k+1}}{0\leq t\leq 1}
$$
turns the above
$$
G\colon \Delta^{k+1} \to \aut_{C,\circ}(\RR^{k+1}\times\xi)
$$
into a map
$$
G\colon (S^k,*)\times I \to \aut_{C,\circ}(\RR^{k+1}\times \xi)
$$
related to $F_1$ and $F_2$ via the diagram
$$
\xymatrix{(S^k,*)\times \{1\} \ar[d]  \ar[r]^-{F_2} & \aut_{C,\circ}(\RR^k\times \xi) \ar[d]^-{(\phi_{k+1})_\#} \\
(S^k,*)\times I \ar[r]^-G & \aut_{C,\circ}(\RR^{k+1}\times \xi) \\
(S^k,*)\times \{0\} \ar[u] \ar[r]^-{F_1} & \aut_{C,\circ}(\RR^k\times \xi). \ar[u]_-{(\phi_k)_\#}}
$$
Since we are assuming that $\dim \xi - \dim X \gg k$,
$$
\aut_{C,\circ}(\xi) \to \aut_{C,\circ}(\RR^{k+1}\times \xi)
$$
is a $(k+1)$-equivalence. Consequently, $F_1$, $F_2$ and $G$ can be viewed as maps into $\aut_{C,\circ}(\xi)$ with $G$ a homotopy between $F_1$ and $F_2$. This completes the proof.
\end{proof}

\begin{corollary}
The geometric realization of $\widetilde{S}_\bullet\aut_{C,\circ}(\xi^S)$ is homotopy equivalent to the topological monoid $\aut_{C,\circ}(\xi^S)$. \hfill $\square$
\end{corollary}

\begin{remark}
It was pointed out in \cite[Appendix 1, \S3]{BLR} that the $\Delta$-group $\tDiff_\partial(M)_\bullet$ is in fact a simplicial group. In our formulation, the degeneracy operator
$$
s_\lambda \colon \tDiff_\partial(M)_k \to \tDiff_\partial(M)_{k+1}
$$
maps a $k$-simplex $(\varphi,\psi)$ into $(s_\lambda(\varphi),s_\lambda(\psi))$ where $s_\lambda(\psi)(x,y) = \psi(s^\lambda x,y)$ and where $s_\lambda(\varphi)(x,y)\in\Delta^{k+1}$ has components
\begin{align*}
s_\lambda(\varphi)_i(x,y) & = \varphi_i(s^\lambda(x),y),\quad 1\leq i\leq \lambda,\\
s_\lambda(\varphi)_{\lambda+1}(x,y) & = \sigma_\lambda(x)\varphi_\lambda(s^\lambda(x),y) + (1 - \sigma_\lambda(x))\varphi_{\lambda+1}(s^\lambda x,y), \\
s_\lambda(\varphi)_i(x,y) & = \varphi_{i-1}(s^\lambda(x)), \quad \lambda +1 < i \leq k+1.
\end{align*}
Here $(x,y)\in \Delta^{k+1}\times M$, $s^\lambda(x_1,\ldots,x_{k+1}) = (x_1,\ldots,\widehat{x}_{\lambda+1},\ldots,x_{k+1})$ and
$$
\sigma_\lambda(x_1,\ldots,x_{k+1}) = \frac{x_{\lambda+2} - x_{\lambda +1}}{x_{\lambda +2} - x_\lambda},
$$
with the conventions $\varphi_0 = 0$, $\varphi_{k+1} = 1$, $x_0 = 0$ and $x_{k+2} = 1$. It is the collar conditions that make the above formulas well-defined.
Indeed, the collar conditions for $\varphi$, listed in preparation to Lemma \ref{lemma:jacobian}, make the denominator of $\sigma_\lambda(x)$ cancel out:
\begin{align*}
(s_0\varphi)_1(x,y) & = x_1,& \mbox{if $x_2\sim 0$},\\
(s_\lambda\varphi)_{\lambda+1}(x,y) & = \varphi_\lambda(s^\lambda x,y) + x_{\lambda+1} - \frac{1}{2} (x_\lambda +x_{\lambda +2}), & \mbox{if $x_\lambda \sim x_{\lambda +2}$},\\
(s_k\varphi)_{k+1}(x,y) & = x_{k+1},& \mbox{if $x_k \sim 1$}.
\end{align*}
Differentiating the above expression for $s_\lambda(\varphi)$, it follows that $\aut_{\partial,\circ}(\tau_M)$ and its stabilization $\aut_{\partial,\circ}(\tau_M^S)$ admit degeneracy operators, so are simplicial monoids. In fact, if we use a collared version of $\aut_{C,\circ}(\xi)$, then it also becomes a simplicial monoid.
\end{remark}

\end{document}